\documentclass{amsart}
\usepackage{amsmath,amssymb,amsthm}
\usepackage{hyperref}
\usepackage{graphicx}
\usepackage{xcolor}
\usepackage{microtype}
\usepackage{mathtools}
\usepackage[english]{babel}

\usepackage[all]{xy}
\usepackage{tikz-cd}

\usepackage{marginnote}
\usepackage[normalem]{ulem}
\usepackage{soul}

\newcommand{\norm}[1]{\left\Vert#1\right\Vert}

\newcommand{\abs}[1]{\left\vert#1\right\vert}
\newcommand{\set}[1]{\left\{#1\right\}}

\newcommand{\duality}[1]{\left\langle#1\right\rangle}

\newcommand{\orb}{\operatorname{Orb}}

\newcommand{\eqnorm}[1]{\left\lvert\!\left\lvert\!\left\lvert #1 
  \right\rvert\!\right\rvert\!\right\rvert}

\newcommand{\pten}{\ensuremath{\widehat{\otimes}_\pi}}

\DeclareMathOperator{\lspan}{span}

\DeclareMathOperator{\rad}{rad}

\DeclareMathOperator{\Lip}{Lip}

\newcommand{\HL}{\mathcal HL}

\newcommand{\N}{\mathbb{N}}

\newcommand{\R}{\mathbb{R}}            
\newcommand{\C}{\mathbb{C}}

\DeclareMathOperator{\co}{co}
\newcommand{\cconv}{\overline{\co}}

\newcommand{\disk}{\mathbb{D}}
\newcommand{\torus}{\mathbb{T}}

\newtheorem{theorem}{Theorem}[section]
\newtheorem{lemma}[theorem]{Lemma}

\newtheorem{proposition}[theorem]{Proposition}
\newtheorem{corollary}[theorem]{Corollary}
\theoremstyle{definition}
\newtheorem{definition}[theorem]{Definition}
\newtheorem{example}[theorem]{Example}

\theoremstyle{remark}
\newtheorem{remark}[theorem]{Remark}
\numberwithin{equation}{section}

\title{Composition operators for holomorphic Lipschitz functions}

\author[V. Dimant]{Verónica Dimant}
\address[V. Dimant]{Departamento de Matem\'{a}tica y Ciencias, Universidad de San
		Andr\'{e}s, Vito Dumas 284, (B1644BID) Victoria, Buenos Aires,
		Argentina and CONICET} \email{vero@udesa.edu.ar}

\author[L. C. García-Lirola]{Luis C. Garc\'ia-Lirola}
\address[L. C. García-Lirola]{Departamento de Matemáticas, Universidad de Zaragoza, 50009, Zaragoza, Spain} 
\email{\texttt{luiscarlos@unizar.es}}
\urladdr{\url{https://personal.unizar.es/luiscarlos/}}

\author[J. Guerrero-Viu]{Juan Guerrero-Viu}
\address[J. Guerrero-Viu]{Departamento de Matemáticas, Universidad de Zaragoza, 50009, Zaragoza, Spain} 
\email{j.guerrero@unizar.es}

\author[A. Procházka]{Antonín Procházka}
\address[A. Procházka]{Universite Marie et Louis Pasteur, CNRS, LmB (UMR 6623), F-25000 Besançon, France.}
\email{antonin.prochazka@univ-fcomte.fr}



\begin{document}

\begin{abstract}
We study composition operators on spaces of holomorphic Lipschitz functions defined on the open unit ball of a complex Banach space. Our approach is based on the linearization of the symbol through the holomorphic Lipschitz-free spaces, which allow composition operators to be realized as adjoints of linear operators. For spaces with the bounded approximation property, we characterize composition operators between spaces of holomorphic Lipschitz functions vanishing at the origin and describe when  composition operators are onto isomorphisms.

We further investigate compactness and weak compactness properties of composition operators. In the finite-dimensional setting, compactness and weak compactness are shown to coincide, and a complete characterization is obtained in terms of the symbol. Finally, we analyze the asymptotic behavior of the iterates of composition operators, proving convergence to zero whenever the supremum norm of the symbol is less than one, and we extend several results to the case not vanishing at 0.
\end{abstract}

\maketitle

\section{Introduction}

The study of the interplay between the properties of a composition operator $C_\phi$ and its symbol $\phi$ began with Nordgren's work \cite{Nor} in 1968. Since then, the topic has undergone tremendous development within the contexts of analytic mappings and Lipschitz functions. Typical areas of focus in this type of research include the boundedness, injectivity, surjectivity, compactness and weak compactness of composition operators; the analysis of their spectra and the behavior of their iterations.

The space of Lipschitz holomorphic mappings defined on the open unit ball of a Banach space, along with some of its linearization properties, is addressed in \cite{ADGLM}. It provides a natural  setting for dealing with composition operators, which is the aim of this article. We study several usual  properties of these operators taking into account the results in the cases of spaces of bounded holomorphic mappings and Lipschitz functions (with the base point fixed at 0) as permanent points of comparison and inspiration.

In each of the three settings—Lipschitz functions vanishing at a base point, bounded holomorphic functions, and Lipschitz holomorphic mappings—there exists a canonical predual which allows for a natural linearization of the nonlinear function spaces. More precisely, the Lipschitz-free space over a pointed metric space (also called Arens-Eells space), see \cite{GK, Weaver2ed}, provides a predual for $\Lip_0(M)$; Mujica’s predual \cite{Mujica} plays the analogous role for spaces of bounded holomorphic functions $\mathcal H^\infty(U)$; and, in the holomorphic Lipschitz framework, the holomorphic Lipschitz-free space $\mathcal G_0(U)$ serves as a canonical predual of the space of holomorphic Lipschitz functions vanishing at $0$. In all three cases, every function in the corresponding function space admits a unique linear representation as a functional on the associated predual space. We refer the reader to \cite{GHT} for a recent survey comparing these constructions. 

We will mainly focus on the case in which the symbol $\phi$ is defined on the open unit ball $B_Y$ and satisfies $\phi(B_Y)\subseteq B_X$, for complex Banach spaces $X$ and $Y$. Thus $\phi$ acts naturally on the space $\mathcal HL(B_X)$ of holomorphic Lipschitz functions on $B_X$ (sometimes called the \emph{analytic Lipschitz algebra} and denoted $\Lip_A(B_X,1)$, see e.g. \cite{BM05, BM25}). We will endow this space with the norm $\norm{f}_L=L(f)+|f(0)|$, where $L(f)$ denotes the Lipschitz constant of $f$, which is also an algebra satisfying $\norm{fg}_L\leq 2\norm{f}_L\norm{g}_L$. However, throughout most of the paper  we work with the subalgebra $\mathcal HL_0(B_X)$ of functions satisfying  $f(0)=0$, where $\norm{f}_L=L(f)$, coming back to the general case in the final section. For a holomorphic Lipschitz symbol $\phi\colon B_Y\to B_X\subseteq X$ with $\phi(0)=0$, the composition operator $C_\phi\colon \mathcal HL_0(B_X)\to \mathcal HL_0(B_Y)$ is well-defined and bounded, and indeed it is the adjoint of the linearization of $\phi$ between the corresponding holomorphic-Lipschitz free spaces $\mathcal G_0(B_Y)$ and $\mathcal G_0(B_X)$. Exploiting this pre-adjoint approach, we obtain in Section~\ref{sect:composition operators HL0} that, for spaces with the bounded approximation property, the composition operators from $\mathcal HL_0(B_X)$ to $\mathcal HL_0(B_Y)$ are precisely the multiplicative operators that arise as adjoint of operators from $\mathcal G_0(B_Y)$ to $\mathcal G_0(B_X)$ (Theorem~\ref{thm:compositionoperatorequivalence}). This compares to similar results in the bounded holomorphic, resp. the  Lipschitz, setting -- \cite[Theorem~4.1]{Car}, resp. 
\cite[Theorem~7.23]{Weaver2ed}. The proof is based on the identification of multiplicative functionals on $\mathcal G_0(B_X)$ as evaluations on $\overline{B}_X$ (Proposition \ref{prop:multiplicativeisdelta}). As a consequence, we obtain a characterization of the symbols for which $C_\phi$ is an onto isomorphism between $\mathcal HL_0$-spaces (Corollary~\ref{c:BAP-characterizationOfCompositionIsomorphism}). 

In Section~\ref{Section:Compactness}, we turn to the study of compactness properties of composition operators on $\mathcal HL_0(B_X)$. Our main result here is Theorem \ref{theo:compactnesscomp}, which reveals a strong connection with the corresponding theory for the composition operator on $\mathcal H^\infty(B_X)$ developed in \cite{AGL}.  In the finite-dimensional case, we prove that compactness and $\ell_\infty$-strict singularity (also known as ``not fixing any copies of $\ell_\infty$", see the definition below) coincide for these operators, and we obtain a complete characterization: a composition operator $C_\phi\colon \mathcal HL_0(B_X)\to\mathcal HL_0(B_Y)$  is compact if and only if the symbol $\phi$ 
 satisfies $\norm{\phi}_\infty <1$ (Corollary \ref{cor:compactfindimcase}). This also echoes known characterizations in the Lipschitz case \cite{KamSch, JVVV, ACPcomp}. 

In Section~\ref{sect:iteration}, we analyze the behavior of the iterates $C^{(n)}_\phi=C_{\phi^{(n)}}$ of the composition operator $C_\phi\colon \mathcal HL_0(B_X)\to\mathcal HL_0(B_X)$. Note that these iterates always have norm 1 when acting on $\mathcal H^\infty(B_X)$. However, the norm of $C_\phi^{(n)}$ on $\mathcal HL_0(B_X)$ is $L(\phi^{(n)})$. We will show in Theorem \ref{thm:ergodic-phi-contracting} that if $\norm{\phi}_\infty<1$ then $L(\phi^{(n)})\underset{n\to\infty}{\to} 0$ and so $C_\phi^{(n)}$ converges to $0$ as an operator on the $\mathcal HL_0$ spaces. For finite-dimensional $X$, we also get a characterization in terms of the existence of invariant subsets of the unit sphere $S_X$ (Proposition \ref{prop:equivalent-convergent-0-findim}). 

In the last part of the paper, Section~\ref{sect:HL}, we come back to the study of  $\mathcal HL(B_X)$. This space also admits a canonical predual, denoted $\mathcal G(B_X)$. However, some properties that hold isometrically in the $\mathcal HL_0$-case now only hold isomorphically (e.g. every $f\in \mathcal HL(B_X, Y)$ admits a unique linearization $T_f\in \mathcal L(\mathcal G(B_X), Y)$ with $\frac{1}{2}\norm{f}_L\leq \norm{T_f}\leq \norm{f}_L$, see Proposition~\ref{prop:Gdef}). Still, transfer principles as Lemma~\ref{lemma: A tilde} allow us to prove a version of the previous results for $\phi$ not necessarily fixing $0$. We also show that there is a natural correspondence between the multiplicative functionals on $\mathcal HL_0(B_X)$, $\mathcal HL(B_X)$, and the algebra of uniformly continuous holomorphic functions on $B_X$, yielding a bijection between the spectra of these algebras (Proposition~\ref{prop:relacion-espectros}).

\subsubsection*{Notation}

$X, Y$ will stand for complex Banach spaces. We denote by $B_X$ (respectively, $S_X$) its open unit ball (respectively, unit sphere). For $A\subseteq X$ we denote $\operatorname{rad}(A)=\sup_{x\in A}\norm{x}$. The space of continuous linear maps from $X$ to $Y$ is denoted by $\mathcal L(X, Y)$, and $X^*=\mathcal L(X, \mathbb C)$. We write $\mathcal P(^m X)$ for the space of continuous $m$-homogeneous polynomials, that is, those $P\colon X\to \mathbb C$ so that  there exists a continuous  $m$-linear symmetric form $\check{P}\colon X\times \cdots \times X\to \mathbb C$ with $P(x)=\check{P}(x,\ldots, x)$. 
We say that $P\in \mathcal P(^m X)$ is of \emph{finite type} if  $P(x)=\sum_{j=1}^n [x_j^*(x)]^m $ for certain $x_j^*\in X^*$ and let $\mathcal P_f(^m X)$  be the space of finite type $m$-homogeneous polynomials. 
Moreover,  $\mathcal P(X)$ stands for the space of finite sums of continuous  homogeneous polynomials on $X$.

For a metric space $(M,d)$ and a Banach space $Y,$ let $\Lip(M,Y)$ be the vector space of all  $f\colon M \to Y$ such that $\norm{f(x_1) - f(x_2)} \leq C d(x_1,x_2)$ for some $C > 0$ and for all $x_1 \neq x_2 \in M$.
The smallest $C$ in the above definition is the {\em Lipschitz constant} of $f$,  and we write it as $L(f)$. Let $0 \in M$ denote an arbitrary fixed point. In order to get a normed space, we will be particularly interested in the subspace $\Lip_0(M,Y)$ consisting of those $f \in \Lip(M,Y)$ such that $f(0) = 0.$ In this way, $L(f) = 0$ if and only if $f = 0$, and so $L(\cdot)$ defines a norm on $\Lip_0(M,Y)$. 
For complex Banach spaces $X$ and $Y$ and an open set $U\subseteq X$,  denote by $\mathcal H^\infty(U,Y)$ the vector space of all $f\colon U \to Y$ such that $f$  is  holomorphic (i.e. complex
Fr\'echet differentiable)   and bounded  on  $U$, endowed with the supremum  norm. 
In both the Lipschitz and $\mathcal H^\infty$ situations, if the range $Y = \mathbb C,$ then the notation is shortened to $\Lip_0(M)$ and $\mathcal H^\infty(U)$. 

Recall that $X$ is said to have the Bounded Approximation Property (BAP) if there is $\lambda>0$ such that the identity $I\colon X\to X$ can be approximated by finite-rank operators in $\lambda B_{\mathcal L(X,X)}$ uniformly on compact sets (equivalently, pointwise). 
We refer the reader to \cite{Casazza} for examples and applications.

\medskip

\section{Preliminaries}\label{sect:preliminares}

In this section, we recall the definition of the holomorphic Lipschitz-free space (HL-free space, for short) introduced in \cite{ADGLM} and some of its properties. Similarly as in the case of Lipschitz-free spaces, it is possible to present the HL-free spaces in a more direct way, avoiding in particular Dixmier-Ng and Montel's theorems.
Although we will not include it here, the same construction is possible for Dineen-Mujica's predual of $\mathcal H^\infty(U)$.

Let $U\subseteq X$ be open, bounded, such that $0 \in U$.
We denote by $\mathcal HL_0(U,Y)$ the Banach space of all holomorphic Lipschitz functions $f\colon U \to Y$ satisfying $f(0)=0$, equipped with the norm $L(f)=\sup_{x\neq y} \frac{\norm{f(x)-f(y)}}{\norm{x-y}}$.
As usual, we shorten the notation $\mathcal HL_0(U):= \mathcal HL_0(U, \mathbb C)$. Note that $f\cdot g$ is holomorphic and
\[ L(f\cdot g) \leq 2 \rad(U) L(f) L(g)\]
for any $f,g\in \mathcal HL_0(U)$, where $\rad(U)=\sup_{x\in U}\norm{x}$. Thus, $\mathcal HL_0(U)$ is a non-unital Gelfand algebra in the terminology of \cite{Weaver2ed}.

For every $x\in U$ we define $\delta(x)\in \mathcal HL_0(U)
^*$ to be the evaluation functional.
Since $X^* \subseteq \mathcal HL_0(U)$ canonically, we clearly have that $\norm{\delta(x)-\delta(y)}=\norm{x-y}$. Thus, the map $\delta\colon U \to \mathcal HL_0(U)^*$ is an isometry. 
Further, $\delta$ is weak*-holomorphic, that is, for each $f\in \mathcal HL_0(U)$ the function $f\circ \delta\colon U\to \mathbb C$ is holomorphic, since $(f\circ \delta) (x) =\langle \delta(x), f\rangle = f(x)$ for each $x\in U$. According to Dunford's result (see e.g. \cite[Corollary~15.47]{DefantGarciaMaestreSevillaPeris} or  \cite[Exercise 8.D]{MujicaLibro}), every weak*-holomorphic map is indeed holomorphic. Hence, $\delta\colon U\to \mathcal HL_0(U)^*$ is holomorphic and isometric. Also, $\delta(0)=0$; consequently, $\delta\in \mathcal HL_0(U, \mathcal HL_0(U)^*)$.

The holomorphic Lipschitz-free space of $U$, for short HL-free space, is defined as 
\[ \mathcal G_0(U):=\overline{\lspan}\, \delta(U)\subseteq \mathcal HL_0(U)^*.\]

The HL-free space has linearization properties similar to those of the Lipschitz-free and holomorphic-free spaces.  
This was shown in \cite{ADGLM} for $\mathcal G_0(B_X)$, but the same argument works for $\mathcal G_0(U)$. 
For the sake of completeness, we provide a proof, whose benefit is that it avoids the use of Montel's theorem. 
We will denote by $\tau_0$ the \emph{topology of uniform convergence on compact subsets of $U$} and by $\tau_p$ the \emph{topology of pointwise convergence}.

\begin{proposition}\label{prop:G0def} Let $X, Y$ be complex Banach spaces and $U\subseteq X$ be an open bounded subset with $0\in U$. 
\begin{itemize}
   \item[(a)] For every $f\in \mathcal HL_0(U,Y)$ there exists a unique $T_f \in \mathcal L(\mathcal G_0(U),Y)$ such that $f=T_f\circ \delta$ and $L(f)=\norm{T_f}$.
    
    \item[(b)] The map $f\mapsto T_f$ is surjective and linear.  In particular, 
    \[\mathcal HL_0(U,Y)=\mathcal L(\mathcal G_0(U),Y) \text{ and } \mathcal HL_0(U)=\mathcal G_0(U)^*.\]

    \item[(c)] The topologies $w^*$, $\tau_0$ and $\tau_p$ coincide on $r\overline{B}_{\mathcal HL_0(U)}$, for any $r>0$.

   \item[(d)] For any $r>0$,
   \begin{align*}
       \mathcal G_0(U) &=\{\varphi\in \mathcal HL_0(U)^*: \varphi|_{r\overline{B}_{\mathcal HL_0(U)}} \text{ is } \tau_0\text{-continuous}\}\\
       & =\{\varphi\in \mathcal HL_0(U)^*: \varphi|_{r\overline{B}_{\mathcal HL_0(U)}} \text{ is } \tau_p\text{-continuous}\}.
   \end{align*}
   \end{itemize}    
\end{proposition}
\begin{proof}
(a) First, note that an interpolation argument shows that the set $\{\delta(x):x\in U\setminus\{0\}\}$ is linearly independent in $\mathcal G_0(U)$. Indeed, assume that $\sum_{j=1}^n \lambda_j \delta(x_j)=0$ for different points $x_j\in U\setminus\{0\}$ and $\lambda_j\in \mathbb C$. Let $x_0=0$ and $\lambda_0=0$. Take $x_{ij}^*\in S_{X^*}$ with $x_{ij}^*(x_i-x_j)=\norm{x_i-x_j}$ and define $f(x)=\sum_{j=0}^n \overline{\lambda_j}\prod_{i\neq j} \frac{x_{ij}^*(x_i-x)}{\norm{x_i-x_j}}$. Then $f\in \mathcal HL_0(U)$ and  $0=\langle f, \sum_{j=1}^n \lambda_j \delta(x_j)\rangle=\sum_{j=1}^n |\lambda_j|^2$.

    Now, given $f \in \mathcal HL_0(U,Y)$ we define $T_f(\delta(x)):=f(x)$ and we extend it linearly to $\lspan \delta(U)$.
    Once we show that $\norm{T_f}\leq L(f)$ in this domain, we use that   $T_f$ admits a norm preserving extension to $\mathcal G_0(U)$.
\[
\begin{split}
\norm{T_f\left(\sum\lambda_i\delta(x_i)\right)}&=\sup_{y^*\in B_{Y^*}} \abs{y^*\left(\sum \lambda_if(x_i)\right)} = \sup_{y^* \in B_{Y^*}} \abs{\sum \lambda_i y^*\circ f(x_i)}\\ 
&\leq L(f) \sup_{g \in B_{\mathcal HL_0(U)}}\left|\duality{g,\sum \lambda_i\delta(x_i)}\right|=L(f)\norm{\sum \lambda_i \delta(x_i)}.
\end{split}
\]
It follows that $\norm{T_f}\leq L(f)$ and $f=T_f\circ \delta$.
Also, $L(f)\leq \norm{T_f} \cdot L(\delta)=\norm{T_f}$.

(b) The map $f\mapsto T_f$ is clearly linear, and its surjectivity follows from the fact that $\delta \in \mathcal HL_0(U,\mathcal G_0(U))$. 

(c) 
 Let $(f_\gamma)_\gamma\subseteq r\overline{B}_{\mathcal HL_0(U)}$ converge $w^*$ to $f \in r\overline{B}_{\mathcal HL_0(U)}$. 
Then for every totally bounded $K \subseteq U$, $f_\gamma\to f$ uniformly on $K$ as $L(f_\gamma-f)\leq 2r$. 
In particular, $f_\gamma \to f$ pointwise. 
In other words, the two identities
$$
Id_1\colon \left(r\overline{B}_{\mathcal HL_0(U)},w^*\right) \to \left(r\overline{B}_{\mathcal HL_0(U)},\tau_0\right) \quad \mbox{and} \quad Id_2\colon \left(r\overline{B}_{\mathcal HL_0(U)},\tau_0\right) \to \left(r\overline{B}_{\mathcal HL_0(U)},\tau_p\right)
$$
are continuous. 
Since $\tau_p$ is Hausdorff, and $(r\overline{B}_{\mathcal HL_0(U)},w^*)$ is compact by Alaoglu's theorem, we have that $Id_1$ and $Id_2$ are homeomorphisms.

(d) 
Fix $\varphi\in \mathcal HL_0(U)^*$. 
Point (c) clearly implies that the restriction
$\varphi|_{r\overline{B}_{\mathcal HL_0(U)}}$ 
is $w^*$-continuous iff 
$\varphi|_{r\overline{B}_{\mathcal HL_0(U)}}$ is 
$\tau_0$-continuous, iff 
$\varphi|_{r\overline{B}_{\mathcal HL_0(U)}}$  is  $\tau_p$-continuous.
Now Banach-Dieudonn\'e implies that $\varphi|_{r\overline{B}_{\mathcal HL_0(U)}}$ 
is $w^*$-continuous iff $\varphi$ 
is $w^*$-continuous. 
This last condition is equivalent to $\varphi \in \mathcal G_0(U)$.
\end{proof}

\begin{remark}
    The fact that $\tau_p$ and $\tau_0$ coincide on $\overline{B}_{\mathcal HL_0(U)}$ is elementary using the total boundedness of compact sets. 
    The fact that $\overline{B}_{\mathcal HL_0(U)}$ is $\tau_0$-compact follows also from Montel's theorem (see e.g. \cite[Proposition 9.16]{MujicaLibro}).
\end{remark}

Recall that every holomorphic Lipschitz map $f\colon U\to Y$ extends uniquely to a map $f\colon \overline{U}\to Y$ with the same Lipschitz constant. In what follows, we show  that the commutativity of the diagrams proved in \cite{ADGLM} holds for these extensions too. 

\begin{lemma}\label{lemma:extension-border-HL0}
    Let $X,Y$ be complex Banach spaces and $U,V$ be open subsets of $X$ and $Y$ respectively, containing $0$. Given $f\in \mathcal HL_0(U,Y)$ we have:
    \begin{itemize}
        \item[(a)] There is a Lipschitz map $\widetilde{f}\colon \overline{U}\rightarrow Y$ such that $\widetilde{f}|_{U}=f$ and $L\big(\widetilde{f}\big)= L(f)$.
        \item[(b)] There is an isometry $\widetilde{\delta} \colon  \overline{U}\rightarrow \mathcal G_0(U)$ such that $\widetilde{\delta}|_{U}=\delta$ and $\|\widetilde{\delta}(x)\|=\norm{x}$ for all $x\in \overline{U}$.
        \item[(c)] There is a unique linear operator $T_f\in \mathcal L(\mathcal G_0(U), Y)$ such that the following diagram commutes.
\begin{equation*}
\xymatrix{
 \overline{U} \ar[r]^{\widetilde{f}}  \ar[d]_{\widetilde{\delta}}    &  Y  \\
   \mathcal{G}_0(U)  \ar[ru]_{T_f}  &
}
\end{equation*}
    \item[(d)] If $f(U)\subseteq  V$, there is a unique linear operator $\widehat{f}\coloneq T_{\delta_Y \circ f} \in \mathcal L(\mathcal G_0(U), \mathcal{G}_0(V))$  such that the following diagram commutes
    \begin{equation*}
        \xymatrix{
 \overline{U} \ar[r]^{\widetilde{f}}  \ar[d]_{\widetilde{\delta}_X }    &  \overline{V} \ar[d]^{\widetilde{\delta}_Y }  \\
   \mathcal{G}_0(U)  \ar[r]_{ \widehat{f}}   &  \mathcal{G}_0(V)
}
    \end{equation*} and $\|\widehat f\|=L(f)$.
    \end{itemize}
\end{lemma}

\begin{proof}
    (a) Just take $\widetilde{f}$ to be the (unique) Lipschitz extension of $f$ to $\overline{U}$.

    (b) Apply (a) to the $1$-Lipschitz function $\delta\colon U\to \mathcal G_0(U)$ to get $\widetilde{\delta}\colon \overline{U}\to \mathcal G_0(U)$. Since $\delta$ is an isometry, it is easy to check that so is $\widetilde{\delta}$. In particular, $\|\widetilde{\delta}(x)\|=\|\widetilde{\delta}(x)-\widetilde{\delta}(0)\|=\norm{x}$.

    (c) Let $T_f$ be the linearization of $f$ given by Proposition \ref{prop:G0def}. If $x\in \overline{U}$, we have 
    \begin{equation*}
        T_f\left(\widetilde{\delta}(x)\right)= T_f\left( \lim_{n\to\infty}\delta(x_n) \right) = \lim_{n\to\infty} T_f(\delta(x_n)) = \lim_{n\to\infty} f(x_n) = \tilde{f}(x),
    \end{equation*}
    where $(x_n)_n\subseteq U$ is any sequence converging to $x$. 

    (d) Since $f(U)\subseteq  V$ and $\widetilde{f}$ is Lipschitz by (a), we infer that $\widetilde f\left(\overline{U}\right)\subseteq \overline{V}$. The conclusion is then obtained by applying part (c) to the map $\delta_Y\circ f\in \mathcal HL_0(U, \mathcal G_0(V))$ and observing that $\widetilde{\delta_Y\circ f}= \widetilde{\delta}_Y \circ \widetilde{f}$. Moreover,  Proposition \ref{prop:G0def} gives $\|\widehat f\|=\|T_{\delta_Y \circ f}\|=L(\delta_Y\circ f)=L(f) $.
\end{proof}

From now on, we will use the same notation $f$ for both the original map and its extension $\widetilde{f}$. The same convention applies to the map $\widetilde{\delta}$.

We also recall the following elementary fact, which will be useful throughout the paper. Given $f\in \mathcal HL_0(U, Y)$, we have
\[\norm{df(x_1)(x_2)}=\lim_{h\to 0}\norm{\frac{f(x_1+hx_2)-f(x)}{h}}\leq L(f)\norm{x_2}, \qquad \forall x_1\in U, \forall x_2\in X\]
and therefore $\sup_{x\in U}\norm{df(x)}\leq L(f)$. If $U$ is also convex, then the mean value theorem yields the reverse inequality. Consequently, the map 
\begin{align*}
D \colon  \mathcal HL_0(U,Y) &\to \mathcal H^\infty(U, \mathcal L(X,Y))\\
 f&\mapsto df
\end{align*}
is an into isometry (although, in general, not onto; see the comments after  Proposition 2.1 in \cite{ADGLM}). That is, if $U\subseteq X$ is an open convex set containing 0 and $f\in \mathcal HL_0(U,Y)$,
\begin{equation}\label{eq:norm Lipschitz-differential on U}
    L(f)= \sup_{x\in U}\norm{df(x)}=\sup\{\|df(x_1)(x_2)\|:\, x_1\in U, x_2\in S_X\}.
\end{equation}
In particular, for $ f\in \mathcal HL_0(B_X,Y)$,
\begin{equation}\label{eq:norm Lipschitz-diferential}
    L(f)= \|df\|_\infty=\sup_{x\in B_X}\norm{df(x)}=\sup\{\|df(x_1)(x_2)\|:\, x_1, x_2\in B_X\}.
\end{equation}

These formulas for the Lipschitz norm imply that the closed unit ball of $\mathcal G_0(U)$ is the absolutely closed convex hull of the set of \emph{elementary molecules} $\frac{\delta(x)-\delta(y)}{\norm{x-y}}$. This was observed in \cite[Proposition 2.6]{ADGLM} for $U=B_X$. We include here a proof which also shows that, when $U$ is convex, it suffices to take the closed convex hull. 
We denote by $\cconv$ the closed convex hull of a given set, and by $\overline\Gamma$ its absolutely closed convex hull.

\begin{proposition}\label{prop:unitballconv} Let $U$ be an open bounded subset of a complex Banach space $X$, containing $0$. Then
\begin{equation*} \overline{B}_{\mathcal G_0(U)}=\overline{\Gamma}\left\{\frac{\delta(x_1)-\delta(x_2)}{\norm{x_1-x_2}} : x_1\not= x_2\in U\right\}.\end{equation*}
If moreover $U$ is convex, then 
\begin{align*}
\overline{B}_{\mathcal G_0(U)}&=\cconv\{d\delta(x_1)(x_2): x_1\in U, x_2\in S_X\}\\
&=\cconv\left\{\frac{\delta(x_1)-\delta(x_2)}{\norm{x_1-x_2}} : x_1\not= x_2\in U\right\}
\end{align*}
\end{proposition}

\begin{proof}
The first part follows using a standard separation argument and the fact that 
\[L(f)=\sup\left\{\left|\left\langle f, \frac{\delta(x_1)-\delta(x_2)}{\norm{x_1-x_2}}\right\rangle\right|: x_1\not= x_2\in U\right\}.\]

Now assume that $U$ is convex. Write $A=\{d\delta(x_1)(x_2): x_1\in U, x_2\in S_X\}$ and note that $\langle f, d\delta(x_1)(x_2)\rangle = df(x_1)(x_2)$, so  \eqref{eq:norm Lipschitz-differential on U} yields $\overline{B}_{\mathcal G_0(U)}=\overline{\Gamma}(A)$ . Also, $\mathbb TA=A$, since $df(x_1)(\lambda x_2)=\lambda df(x_1)(x_2)$ for each $\lambda\in \mathbb T$. Thus, 
\[ \overline{B}_{\mathcal G_0(U)}=\overline{\Gamma}(A) = \cconv(\mathbb TA)=\cconv(A). \]
Finally, note that, for $x_1\in U$, $x_2\in X$ and small $t>0$, the differentiability of $\delta\colon U\to\mathcal G_0(U)$ implies that
\[
 \norm{ \frac{\delta(x_1+tx_2)-\delta(x_1)}{t}-d\delta(x_1)(x_2)}\to 0, \mbox{ as } t\to 0.
    \]
Therefore, 
\[
A \subseteq \overline{\set{\frac{\delta(x_1)-\delta(x_2)}{\norm{x_1-x_2}}:x_1\not= x_2 \in U}},
\]
from where we get the last equality. 
\end{proof}
\subsection{Some facts on operator ideals}
Even though the notion of an operator ideal is classical, see \cite[p. 45]{Pietsch}, in the sequel we are going to use only one of its constituting properties. 
We say that a class  $\mathcal I$ of bounded linear operators has the \emph{ideal property} if, for all bounded linear operators $A,B$ and each $T \in \mathcal I$, we have $A\circ T\circ B \in \mathcal I$ whenever the composition makes sense, see \cite[p. 108]{DefantFloret}.

Besides the ideals of compact and weakly compact operators, in this paper we will also deal with some related classes of operators which we will now define.

Let $Z$ be a Banach space. 
We say that $T \in \mathcal L(X,Y)$ \emph{fixes a copy of $Z$} if $Z$ is isomorphic to a subspace $E \subseteq X$ and $T|_E$ is an isomorphism onto its image.
The operators which do not fix any copy of $Z$ are also known in the literature under the name \emph{$Z$-strictly singular operators} \cite{KaniaLaustsen}.
Clearly the class of $Z$-strictly singular operators has the ideal property.
While $Z$-strictly singular operators are not necessarily closed under addition for a general $Z$, 
when $Z=\ell_1$, resp. $Z=\ell_\infty$, it is known that $Z$-strictly singular operators form a closed operator ideal (see \cite[Proposition~2.5]{KaniaLaustsen}, resp. \cite[Corollary 1.4]{Rosenthal}).

We say that $T \in \mathcal L(X,Y)$ \emph{fixes a complemented copy of $Z$} if there exists a complemented subspace $E$ of $X$ such that $E$ is isomorphic to $Z$ and $T|_E$ is an isomorphism onto its image. 
This is easily seen to be equivalent to requiring that $Z$ is isomorphic to some $E \subseteq X$ such that $T|_E$ is an isomorphism onto its image and $T(E)$ is complemented in $Y$, see~\cite{Lemay}.

Again, it is clear that the class of operators which do not fix any complemented copy of $Z$ has the ideal property. 

In the diagram depicted in Figure~\ref{fig:Ideals} we relate the above notions to some classical operator ideals.

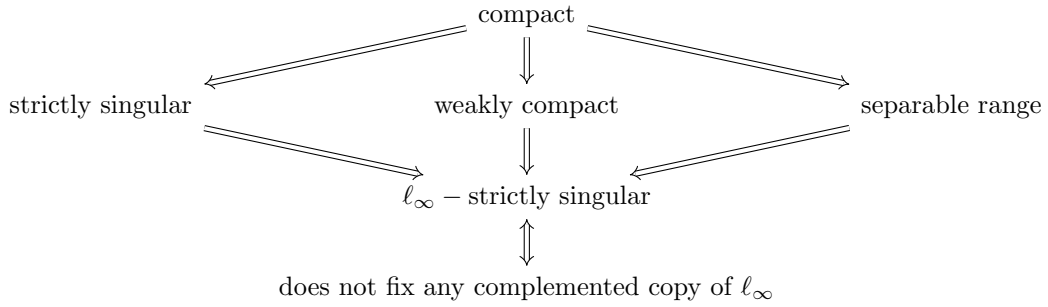
\begin{figure}[h]
    \centering

\[\begin{tikzcd}
	& {\text{compact}} & \\
	{\text{strictly singular}} & {\text{weakly compact}} & {\text{separable range}} \\
	& {\ell_\infty-\text{strictly singular}} \\
	& {\text{does not fix any complemented copy of } \ell_\infty}
	\arrow[from=1-2, to=2-1, Rightarrow]
	\arrow[from=1-2, to=2-2, Rightarrow]
	\arrow[from=1-2, to=2-3, Rightarrow]
	\arrow[from=2-1, to=3-2, Rightarrow]
	\arrow[from=2-2, to=3-2, Rightarrow]
	\arrow[from=2-3, to=3-2, Rightarrow]
	\arrow[from=3-2, to=4-2, Leftrightarrow]
\end{tikzcd}\]
    \caption{All downward implications are trivial. 
    The only upward implication follows from the injectivity of $\ell_\infty$.}
    \label{fig:Ideals}
\end{figure}

Clearly, when $T$ fixes a complemented copy of $Z$, then $T^*$ fixes a complemented copy of $Z^*$. 
The converse statement is true for $Z=\ell_1$. 
The following slightly stronger result is probably well known but we struggle to find a proper reference. 
\begin{lemma}\label{l:BessagaPelczynski}
If $T^*\colon X^* \to Y^*$ fixes a copy of $c_0$ then $T:Y\to X$ fixes a complemented copy of $\ell_1$.
\end{lemma}
We follow the usual proof of the special case $T^*=Id_{X^*}$ of this theorem, due to Bessaga and Pe\l czy\'nski~\cite{BessagaPelzcynski1958}, see also \cite[Proposition~2.e.8]{LindenstraussTzafriri}. 
\begin{proof}
Let $S\colon c_0 \to X^*$ be the isomorphic embedding such that $T^*\circ S$ is an isomorphic embedding.
Then $S^*\circ T^{**}\colon Y^{**} \to \ell_1$ is a surjective map.
By open mapping theorem, there is $C>0$ such that there are $(y_n)_n \subseteq CB_{Y^{**}}$ with $S^*\circ T^{**}y_n=e_n^*$, where $(e_n^*)$ is the unit vector basis of $\ell_1$.
Using Goldstine theorem, we find $(x_n)_n \subseteq CB_Y$ such that 
    $$\begin{cases} |\langle S^*\circ T^{**}x_n,e_k\rangle|< \frac1n, & k<n\\ |\langle S^*\circ T^{**}x_n,e_n\rangle-1|< \frac1n.\end{cases}$$
This means that $(S^*\circ T^{**}x_n)_n=(S^*\circ Tx_n)_n$ is $w^*$-null in $\ell_1$ but not norm-null.
Bessaga-Pe\l czy\'nski selection principle (see~\cite[Theorem 1.3.10]{AlbiacKalton}) gives a subsequence $(S^*\circ T^{**}x_{n_k})_k$ which will be congruent to a block basic subsequence of $\ell_1$ basis.
It follows, as is well known, that $F:=\operatorname{span}(S^*\circ T^{**}x_{n_k})_k$ is isomorphic to $\ell_1$ and complemented in $\ell_1$, say by a projection $P\colon\ell_1 \to F$.
We define $L_Y\colon F \to Y$ as the unique linear map which satisfies $L_Y(S^*\circ T^{**}x_{n_k}):=x_{n_k}$. 
Clearly this is a lifting of $S^*\circ T^{**}|_{(S^*\circ T^{**})^{-1}(F)}$.
Similarly, $L_X\colon F \to X$ is the unique linear map which satisfies $L_X(S^*\circ T^{**}x_{n_k}):= T^{**}x_{n_k}=Tx_{n_k}$; it is a lifting of the surjective map $S^*|_{(S^*)^{-1}(F)}$.
One can see from the previous considerations that $T|_{L_Y(F)}$ is an isomorphism and $T(L_Y(F))=L_X(F)$. 
Finally, the space $L_X(F)$ is complemented by the projection $L_X\circ P\circ S^*|_X$ (while the space $L_Y(F)$ is complemented by the projection $L_Y\circ P \circ S^* \circ T$). 
\end{proof}

In the sequel we will need the following known lemma about the compactness of tensor operators. For completeness, we include proofs of the statements for which we do not have a direct reference. 

\begin{lemma}\label{lemma:comptensorop} Let $X, Y, Z, W$ be Banach spaces and $T\colon X\to Z$, $S\colon Y\to W$ be non-zero
operators, and consider $T\otimes S\colon  X\pten Y\to Z\pten W$. Then 
\begin{itemize}
    \item[(a)] If $T\otimes S$ belongs to a class $\mathcal I$ with the ideal property, then $T$ and $S$ belong to $\mathcal I$.
    \item[(b)] If $T$ and $S$ are compact, then $T\otimes S$ is compact.
    \item[(c)] If $T$ is compact and $S$ is weakly compact, then $T\otimes S$ is weakly compact.
\end{itemize} 
\end{lemma}

\begin{proof}
(a) Let $y_0\in Y$ such that $Sy_0\neq 0$ and take $w^*\in W^*$ with $w^*(Sy_0)=1$. 
Consider the operators $i\colon X\to X\pten Y$ given by $i(x)=x\otimes y_0$ and $R\colon Z\pten W\to Z$ given by $R(z\otimes w)=w^*(w)z$. 
Then it is clear that $T=R\circ (T\otimes S)\circ i$, so $T\in \mathcal I$. Analogously, $S\in \mathcal I$.  

(b)    We have that 
    \begin{align*} (T\otimes S)(\overline{B}_{X\pten Y}) &= (T\otimes S)(\cconv(B_X\otimes B_Y)) \subseteq \cconv(T\otimes S)(B_X\otimes B_Y)\\
    &=\cconv(T(B_X)\otimes S(B_Y)).
    \end{align*}
Thus, it suffices to show that $T(B_X)\otimes S(B_Y)$ is relatively compact. For that, take sequences $(x_n)_n\subseteq B_X$ and $(y_n)_n\subseteq B_Y$. Since $T$ and $S$ are compact, there is a subsequence $(n_k)_k$ such that $(Tx_{n_k})_k$ and $(Sy_{n_k})_k$ converge to points $z\in Z$ and $w$ in $W$. Then
\[ \norm{Tx_{n_k}\otimes Sy_{n_k}-z\otimes w}\leq \norm{Tx_{n_k}\otimes Sy_{n_k}- z\otimes Sy_{n_k}}+\norm{z\otimes Sy_{n_k}-z\otimes w}\to 0 \]

(c) This is Corollary~1.1 in \cite{racher}.
\end{proof}

\medskip
\section{Composition Operators on \texorpdfstring{$\mathcal HL_0(B_X)$}{HL0(BX)}}\label{sect:composition operators HL0}

We begin by defining our object of study in this setting.

\begin{definition}{\label{def:compositionoperator}}
    Let $X, Y$ be complex Banach spaces and $U\subseteq X$, $V\subseteq Y$ be open subsets containing 0.
    A mapping $A \colon \HL_0(U) \to \HL_0(V)$ is called a \textit{composition operator} if there exists a map $\phi \in \mathcal HL_0(V, X)$ satisfying $\phi(V)\subseteq U$ such that $A (f) = f\circ \phi$, for every $f\in \mathcal HL_0(U)$.
    Since $\HL_0(U)$ separates points of $\delta(U)$, such $\phi$ is unique and we denote the associated composition operator as $C_\phi$.
\end{definition}

For $\phi\in \mathcal HL_0(V,X)$ with $\phi(V)\subseteq U$, Lemma \ref{lemma:extension-border-HL0} (d) gives us a linear map  $\widehat{\phi}: \mathcal G_0(V)\rightarrow \mathcal G_0(U)$ satisfying $\widehat{\phi}\circ \delta_Y=\delta_X\circ \phi$ and $\|\widehat{\phi}\|=L(\phi)$.

\begin{lemma}\label{lemma:compositionoperatorproperties}
      We always have $C_\phi= \widehat{\phi}^*$ and so $C_\phi$ is continuous with $\norm{C_\phi}=L(\phi)$.
     \end{lemma}

\begin{proof} Note that, for every $f\in \mathcal HL_0(U)$ and $y\in V$, we have
\begin{align*} \widehat{\phi}^*(f)(y)&=\langle \widehat{\phi}^*(f), \delta_Y(y)\rangle= \langle f, \widehat{\phi}(\delta_Y(y))\rangle =  \langle f, \delta_X(\phi(y))\rangle\\
&= f(\phi(y)) = C_\phi(f)(y).\end{align*}
 Thus $C_\phi= \widehat{\phi}^*$ and so $C_\phi$ is a bounded linear operator with $\norm{C_\phi}=\|\widehat{\phi}\|=L(\phi)$.  
\end{proof}

Note that given a holomorphic function $\phi\colon V\to X$ with $\phi(V)\subseteq U$, the composition operator $C_\phi\colon \mathcal H^\infty(U)\to\mathcal H^\infty (V)$ satisfies $\norm{C_\phi}=1$. 
This is always the case for multiplicative operators between Banach algebras with unit. However, in the Lipschitz-holomorphic case it is possible to construct composition operators with arbitrarily large or arbitrarily small norm,   as the following example shows. 

\begin{example}\label{e:CompositionArbitraryNorm} A composition operator with large norm / A composition operator with small norm.
   \begin{enumerate}
       \item Given $n\in \N$, pick $x\in S_X$ and $y^*\in S_{Y^*}$. Define $\phi \colon B_Y \rightarrow X$ by $\phi(y)=y^*(y)^{n} x$ for all $y\in B_Y$. It is clear that $\phi \in \mathcal{H}L_0(B_Y,X)$ since $\phi$ is a homogeneous polynomial. Moreover, 
    \begin{equation*}
        \norm{\phi(y)} =\norm{y^*(y)^{n}x}\leq  \norm{y^*}^{n}\norm{y}^{n} \norm{x}<1, \quad \forall y\in B_Y,
    \end{equation*}
    showing that $\phi(B_Y)\subseteq B_X$. It is clear that $\norm{\phi}_\infty=1$ but we see that $L(\phi)=n$. Indeed, the differential of this mapping satisfies $$d\phi(y_1)(y_2)=n y^*(y_1)^{n-1}y^*(y_2) x, \quad \forall y_1\in B_Y,\ \forall y_2\in S_Y.$$
    Hence, using \eqref{eq:norm Lipschitz-diferential} we conclude that $\|C_\phi\|=L(\phi)=\norm{d\phi}_\infty=n$.

    \item Given $0<\varepsilon<1$, pick a linear map $\phi\in\mathcal L(Y,X)$ with $\|\phi\|=\varepsilon$. Then, the composition operator $C_\phi \colon \mathcal HL_0(B_X)\rightarrow \mathcal{H}L_0(B_Y)$ satisfies $\|C_\phi\|=L(\phi)=\|\phi\|=\varepsilon$.
   \end{enumerate}
\end{example}

\subsection{Multiplicative and composition operators. }

As already mentioned, the space $\mathcal HL_0(B_X)$ is a Gelfand algebra: if $f,g\in\mathcal HL_0(B_X)$ then $f\cdot g\in\mathcal HL_0(B_X)$ with $L(f\cdot g)\le 2L(f)L(g)$. Given that composition operators respect the product (i.e., $C_\phi(f\cdot g)=C_\phi(f)\cdot C_\phi(g)$),  it is interesting to know when a multiplicative continuous linear mapping from $\mathcal HL_0(B_X)$ to $\mathcal HL_0(B_Y)$ is actually a composition operator. Notice that in the case of Lipschitz-free spaces over bounded metric spaces, Theorem~7.23 in~\cite{Weaver2ed} claims that for every normal algebra homomorphism $T\colon \Lip_0(M)\to \Lip_0(N)$ there is some $g\in \Lip_0(N,M)$ such that $T=C_g$. Analogously, Theorem~4.1 in \cite{Car} claims that every weak*-weak*-continuous algebra homomorphism $T\colon \mathcal H^\infty(U)\to \mathcal H^\infty(V)$ is a composition operator, for absolutely convex bounded open sets $U,V$ under the hypothesis of bounded approximation property.

With this goal in mind we start by presenting some  results about linear multiplicative (and therefore continuous) forms. Given a Gelfand algebra $\mathcal A$, we say that a non-zero map $\varphi : \mathcal A\rightarrow \C$ is a \textit{character} if it is both linear and multiplicative.
\begin{lemma}\label{l:MultiplicativeFormsOnGelfandAlgebras}
If $\mathcal A$ is a Gelfand algebra with $\norm{fg}\leq C\norm{f}\norm{g}$ for each $f,g\in \mathcal A$, then every character $\varphi \colon \mathcal A\to \mathbb C$ is continuous and satisfies $\norm{\varphi}\leq C$.
\end{lemma}
\begin{proof}
The fact that a multiplicative linear map is necessarily bounded is standard so we only sketch the argument. 
Indeed, let $\mathcal A_1= \mathcal A\oplus \mathbb C$ endowed with the product $(f,\lambda)\cdot(g,\mu)= (fg+\lambda g + \mu f, \lambda\mu)$ and the norm $\norm{(f,\lambda)}_1=C\norm{f}+|\lambda|$. It is easy to check that $(\mathcal A_1,\norm{\cdot}_1)$ is a unital Banach algebra. Moreover, we can consider $\varphi_1\colon \mathcal A_1\to \mathbb C$ given by $\varphi_1(f,\lambda)=\varphi(f)+\lambda$. Then it follows that $\varphi_1$ is a character on $\mathcal A_1$ with $\varphi_1(1)=1$. A well-known result yields that $\varphi_1$ is continuous with $\norm{\varphi_1}=1$ (see e.g. \cite[Proposition 2.22]{Douglas}). Thus, for $f\in \mathcal A$, \[ |\varphi(f)|=|\varphi_1(f,0)|\leq \norm{(f,0)}_1 = C\norm{f}.\]     
\end{proof}

For a Gelfand algebra $\mathcal A$, we denote by $\mathcal M_0(\mathcal A)$ the set of characters $\varphi\colon \mathcal A \rightarrow\C$ together with the zero map. Hence, for every $f\in \mathcal A$ and $\varphi\in \mathcal M_0(\mathcal A)$ define $\widehat{f}(\varphi)=\varphi(f)$. The map $\Gamma_{\mathcal A} \colon f \rightarrow \widehat{f}$ is called the \textit{Gelfand transform}. It is a non-expansive algebra homomorphism from $\mathcal A$ into $\Lip_0(\mathcal M_0(\mathcal A))$, which is in general not injective. However, for the cases where $\mathcal A$ is $\mathcal HL_0(U)$ or $\mathcal{H}L(U)$ (see definition in Section \ref{sect:HL}), the Gelfand transform is injective since we clearly have $\{ \delta(x) : x\in U\}\subseteq \mathcal M_0(\mathcal A)$. Another example of algebras where the Gelfand transform is always injective are the commutative $C^*$-algebras. In the above settings, we  obtain the following useful result.

\begin{proposition}\label{prop:automaticcontinuity}
    Let $\mathcal A,\mathcal B$ be Gelfand algebras and $T\colon \mathcal A\rightarrow \mathcal B$ be a linear multiplicative map. If the Gelfand transform $\Gamma_{\mathcal B}$ is injective, then $T$ is bounded.
\end{proposition}

\begin{proof}
    We show that $T$ is bounded using the closed graph theorem.  Assume that $f_n\to 0$ in $\mathcal A$ and $T(f_n)\to g$ in $\mathcal B$. Fix $\varphi\in \mathcal M_0(\mathcal B)$. Hence, it is clear that $\varphi\circ T\in \mathcal M_0(\mathcal A)$ and by Lemma \ref{l:MultiplicativeFormsOnGelfandAlgebras}, both $\varphi$ and $\varphi\circ T$ are continuous. Then,
    \begin{equation*}
        |\varphi(g)|\leq |\varphi(g-T(f_n))|+|\varphi(T(f_n))| \leq \norm{\varphi}\norm{g-T(f_n)}+\norm{\varphi\circ T}\norm{f_n}\to 0.
    \end{equation*}
    Since $\varphi$ was arbitrary and $\Gamma_{\mathcal B}$ is injective we conclude that $g=0$.
     This proves that $T$ is bounded.
\end{proof}

Lemma~\ref{l:MultiplicativeFormsOnGelfandAlgebras} already implies that a multiplicative linear form $\varphi$ on $\HL_0(U)$ is continuous and satisfies $\norm{\varphi}\leq 2\rad(U)$. 
But we can obtain a better estimate.
For that, we define the \emph{spectral radius} of $f\in \mathcal HL_0(U)$ as $r(f)\coloneqq\lim_{n\to\infty} L(f^n)^{1/n}$. The existence of this limit can be shown by adapting the usual arguments on Banach algebras to the Gelfand algebra case, for instance by considering the submultiplicative sequence $b_n =2\rad(U)L(f^n)$. Then $a_n=\log b_n$ defines a subadditive sequence. By the classical Fekete's lemma \cite{Fekete} the limit $c=\lim_{n\to\infty}\frac{a_n}{n}$ exists, and so $r(f)=\lim_{n\to\infty} b_n^{1/n} = e^c$. 

\begin{proposition}\label{prop:multiplicativebounded} Let $U$ be a bounded open subset of a complex Banach space with $0\in U$. 
\begin{itemize}
\item[(a)] For each $f\in \mathcal HL_0(U)$, we have $r(f)\leq \|f\|_\infty\le \rad(U)L(f)$. 
\item[(b)] Let $\varphi$ be a  multiplicative linear form on $\HL_0(U)$. 
Then $\varphi\in \mathcal HL_0(U)^*$ and 
\[ |\varphi(f)|\leq r(f) \qquad \forall f\in \mathcal HL_0(U).\]
Therefore $\norm{\varphi}\leq \rad(U)$. 
\item[(c)] Let $Y$ be a complex Banach space and $V\subseteq Y$ be an open set containing $0$. 
Let $A:\HL_0(U)\to \HL_0(V)$ be multiplicative and linear. 
Then $A$ is bounded.
\end{itemize}    
\end{proposition}
Notice that Example~\ref{e:CompositionArbitraryNorm} provides multiplicative and linear maps with arbitrarily large norm. Therefore, in (c) above one cannot expect an estimate for $\norm{A}$ analogous to the one in (b), depending only on the geometry of $U$.
\begin{proof}
(a) Given $n\in \mathbb N$, take $g_n(z)=z^n$ for $z\in\mathbb C$. Note that $L(g_n|_{R\mathbb D
}) = n R^{n-1}$ for each $R>0$. Thus, for $f\in \mathcal HL_0(U)$, 
\[ L(f^n)=L(g_n\circ f)\leq L(g_n|_{\norm{f}_\infty \mathbb D}) L(f) = n \norm{f}_\infty^{n-1}L(f)\]
and so, noting that $\|f\|_\infty\le \rad(U) L(f)$ we have 
\[ r(f)\leq \lim_{n\to\infty} n^{\frac{1}{n}} \norm{f}_\infty^{\frac{n-1}{n}} L(f)^{\frac{1}{n}} = \norm{f}_\infty\le\rad(U) L(f).\]
(b) Using Lemma~\ref{l:MultiplicativeFormsOnGelfandAlgebras}, we see that $\varphi$ is bounded. 
Thus it  satisfies
 \[ |\varphi(f)|^n =|\varphi(f^n)|\leq \norm{\varphi}L(f^n)\]
and then $|\varphi(f)|\leq \norm{\varphi}^{1/n} L(f^n)^{1/n}$. Taking limit with $n\to\infty$ we conclude that $|\varphi(f)|\leq r(f)$, for every $f\in \mathcal HL_0(U)$. 

(c) It follows from Proposition \ref{prop:automaticcontinuity} and the fact that $\{\delta(x):x\in U\}\subseteq \mathcal M_0(\mathcal HL_0(U))$. 
\end{proof}

\begin{remark} \label{rmk:acotada-norma-infinito}
    As a consequence of the previous proposition, if $\varphi\in \mathcal HL_0(U)^*$ is  multiplicative, then $|\varphi(f)|\le\|f\|_\infty$, for all $f\in \mathcal HL_0(U)$. 
    On the other hand, a non-multiplicative linear form $\varphi\in \mathcal HL_0(U)^*$ could be unbounded in the set $\{f\in \mathcal HL_0(U): \|f\|_\infty\le 1\}$. Indeed, consider for instance, a sequence $(z_n)_n\subseteq\disk$ such that $z_n\to 1$. Define, for each $n$, the functional $\varphi_n\in \mathcal HL_0(\disk)^*$ given by $\varphi_n(f)=f'(z_n)$. Since $(\varphi_n)_n$ is contained in the $w^*$-compact set $\overline B_{\mathcal HL_0(\disk)^*}$ it has a  subnet $(\varphi_{n_\alpha})_\alpha $ which is $w^*$-convergent to an element $\varphi\in \overline B_{\mathcal HL_0(\disk)^*}$. To see that $\varphi$ is not bounded in the set $\{f\in \mathcal HL_0(\disk): \|f\|_\infty\le 1\}$, take the sequence of functions $(f_k)_k$ in this set given by $f_k(z)=z^k$ and note that $\varphi(f_k)=\lim_\alpha k z_{n_\alpha}^{k-1}=k$.
\end{remark}

\begin{example} \label{ex:delta_x} Note that any $x\in \overline{U}$ yields a multiplicative functional $\varphi=\delta(x)$  with $\norm{\varphi}= \norm{x}$. In particular, this shows that there are multiplicative forms with any norm between $0$ and $\rad(U)$. \end{example}

For the case $U=B_X$, Proposition \ref{prop:multiplicativebounded} yields the following result. 

\begin{corollary}\label{cor:multiplicativebounded}
    Let $X$ be a complex Banach space and let $\varphi\in \mathcal{H}L_0(B_X)^*$. If $\varphi$ is also multiplicative, then $\norm{\varphi}\leq 1$.
\end{corollary}

In the following results we will restrict our attention to the case where $U=B_X$, although the same statements remain valid for open bounded sets $U$ that are convex and balanced. 
Indeed, in this situation $U=B_{(X,\eqnorm{\cdot})}$,for some equivalent norm $\eqnorm{\cdot}$ on $X$. Consequently, $f\in \mathcal HL_0(U)$ if and only if $f\in \mathcal HL_0(B_{X,\eqnorm{\cdot}})$, and $\mathcal G_0(U)$ is a renorming of $\mathcal G_0(B_{(X,\eqnorm{\cdot})})$.  

In general, there are many multiplicative elements of $\mathcal HL_0(B_X)^*\setminus \delta(\overline B_X)$. This can be deduced from the connection with the spectrum of a Banach algebra established in Corollary \ref{cor:biyeccion-espectro} (see also the examples below the corollary).
We shall prove, however, that whenever$X$ has the bounded approximation property, the only multiplicative elements in $\mathcal G_0(B_X)$ are the evaluations at points of $\overline B_X$. The next proposition establishes this fact; it is an adaptation of \cite[Proposition~3.3]{Car} to our context and should also be compared with \cite[Lemma~7.22]{Weaver2ed}.

\begin{proposition}\label{prop:multiplicativeisdelta}
   Let $X$ be a complex Banach space with the BAP. Let $\varphi\in \mathcal G_0(B_X)$ be multiplicative. Then, there is some $x_0\in \overline{B}_X$ such that $\varphi=\delta(x_0)$.
\end{proposition}

In the proof of Proposition~\ref{prop:multiplicativeisdelta} we will use the following lemma of independent interest.
\begin{lemma}\label{l:PolynomialsSeparatePoints}
    Let $X$ be a complex Banach space. Then,
    \begin{itemize}
        \item[(a)] The homogeneous polynomials on $X$ separate points of $\mathcal G_0(B_X)$.
        \item[(b)] The homogeneous polynomials of finite type on $X$ separate points of $\mathcal G_0(B_X)$ provided $X$ has the BAP.
        \item[(c)] The dual space $X^*$ separates the multiplicative elements of $\mathcal G_0(B_X)$ provided $X$ has the BAP. 
        In other words, if $X$ has the BAP and if $\varphi\neq \psi \in \mathcal G_0(B_X)$ are multiplicative, then there exists $x^*\in X^*$ such that $\duality{x^*,\varphi}\neq \duality{x^*,\psi}$.
    \end{itemize}
\end{lemma} 

\begin{proof}
     (a) Let $\mu\in \mathcal G_0(B_X)$.  
    It follows immediately from \cite[Proposition 3.1 (a)]{ADGLM} and from the fact that $\tau_0$ and $w^*$ coincide on bounded sets of $\HL_0(B_X)$, that there is a polynomial $P=\sum_{m=0}^N P_m$ on $X$ such that $\duality{P,\mu}\neq 0$. Clearly, $\duality{P_m,\mu}\neq 0$ for some $m$.
    
     (b) Given $\mu\in \mathcal G_0(B_X)$, we know, by (a), that there are some $m\in \N$ and $P \in \mathcal P(^mX)$ such that $\duality{P,\mu}\neq 0$.
    Moreover, since $X$ has the $\lambda$-BAP for some $\lambda\ge 1$, we can find a net of finite rank operators $(T_\alpha)_\alpha\subseteq\mathcal F(X,X)$ with $ \sup_{\alpha} \norm{T_\alpha}\leq \lambda$ such that $(T_\alpha)$ converges pointwise to the identity map.
    Consider, for each $\alpha$, $P_\alpha=P\circ T_\alpha$, which is a polynomial of finite type (see, e.g., \cite[Exercise 2.K (b)]{MujicaLibro}). On the one hand, since $P$ is continuous, $(P_\alpha)_\alpha$ converges pointwise to $P$. On the other hand, for any $\alpha$, \begin{align*}
        L(P_\alpha|_{B_X})&=\sup_{x,y\in B_X} \frac{|P(T_\alpha(x))-P(T_\alpha(y))|}{\norm{x-y}} \\&= \norm{T_\alpha}^m \sup_{x,y\in B_X} \frac{\left|P\left(\frac{1}{\norm{T_\alpha}}T_\alpha(x)\right)-P\left(\frac{1}{\norm{T_\alpha}}T_\alpha(y)\right)\right|}{\norm{x-y}}\\&\leq \norm{T_\alpha}^m L(P|_{B_X}) \sup_{x,y} \frac{|T_\alpha(x-y)|}{\norm{T_\alpha}\norm{x-y}} = \norm{T_\alpha}^m L(P|_{B_X})\leq \lambda^m L(P|_{B_X}), 
    \end{align*}
    so $P|_{B_X}$ and the net $(P_\alpha|_{B_X})_\alpha$ are   in $\lambda^m L(P|_{B_X})\overline{B}_{\mathcal HL_0(B_X)}$. 
    Since $w^*$ and $\tau_p$ coincide on bounded sets of $\HL_0(B_X)$ we get that $P_\alpha \overset{w^*}{\to} P$.
    Thus $\duality{P_\alpha,\mu}\neq 0$ for some $\alpha$.

    In order to prove (c) let $\varphi\neq \psi \in \mathcal G_0(B_X)$ be multiplicative. 
    Let $P\in \mathcal P_f(^mX)$ be obtained by application of point (b) so that $\duality{P,\varphi-\psi}\neq 0$. 
    We have that $P(x)=\sum_{i=1}^N x_i^*(x)^m$ for every $x \in X$.
    Therefore $\duality{P,\mu}=\sum_{i=1}^N \duality{(x_i^*)^m,\mu}$ for every $\mu \in \mathcal G_0(B_X)$.
    In particular, it follows that $\duality{(x_i^*)^m,\varphi}\neq \duality{(x_i^*)^m,\psi}$ for some $i \leq N$.
    Now we use the multiplicativity of $\varphi$ and $\psi$ to get the conclusion.
\end{proof}

\begin{proof}[Proof of Proposition~\ref{prop:multiplicativeisdelta}]     First of all, since $X^*\subseteq \mathcal HL_0(B_X)$, we can restrict $\varphi$ to $X^*$ obtaining $\varphi|_{X^*}\in X^{**}$. Actually, $\varphi|_{X^*}$ belongs to $X$. Indeed, consider $R:\mathcal G_0(B_X)\to X^{**}$ the restriction given by $R(\varphi)=\varphi|_{X^*}$ and  the usual  commutative  diagram for the identity map:
     \begin{equation*}
\xymatrix{
 B_X \ar[r]^{id}  \ar[d]_{\delta}    &  X  \\
   \mathcal{G}_0(B_X)  \ar[ru]_{T_{id}}  &
}
\end{equation*}
Since $T_{id}(\delta(x))=x=R(\delta(x))$ for all $x\in B_X$, by the uniqueness of $T_{id}$ we  have $T_{id}=R$. 
Then, there is some $x_0\in X$ such that $\varphi|_{X^*}=x_0$. Furthermore, applying Corollary \ref{cor:multiplicativebounded}, we have \begin{equation*}
        \norm{x_0}=\sup_{x^*\in S_{X^*}} |x^*(x_0)| = \sup_{x^*\in S_{X^*}} |\varphi|_{X^*}(x^*)| \leq \norm{\varphi}\leq 1,
    \end{equation*}
    so $x_0\in \overline{B}_X$.
    Now application of Lemma~\ref{l:PolynomialsSeparatePoints} shows that $\varphi=\delta(x_0)$.
\end{proof}

Now we arrive to the promised result which is a version of \cite[Theorem 4.1]{Car} and \cite[Theorem 7.23]{Weaver2ed}. 

\begin{theorem}\label{thm:compositionoperatorequivalence}
    Let $X, Y$ be complex Banach spaces. 
    Let $A \colon \mathcal HL_0(B_X)\rightarrow \mathcal HL_0(B_Y)$ be a multiplicative linear mapping and consider the following statements.
    \begin{enumerate}
        \item $A$ is a composition operator.
        \item There is a linear continuous map $T\colon \mathcal G_0(B_Y)\rightarrow \mathcal G_0(B_X)$ such that $A=T^*$.
        \item $A$ is $\tau_0$-$\tau_0$-continuous.
        \item $A$ is $\tau_p$-$\tau_p$-continuous.
    \end{enumerate}
    Then 
    \begin{itemize}
        \item $(1)$ implies $(2)$, $(3)$ and $(4)$,
        \item $((3)$ or $(4))$ implies $(2)$.
    \end{itemize}
    When $X$ has the BAP, the statements $(1)$-$(4)$ are equivalent.
\end{theorem}

\begin{proof}     
    $(1)\Rightarrow(2)$ We have already checked in Lemma~\ref{lemma:compositionoperatorproperties} that if $A=C_\phi$ for $\phi\in \mathcal HL_0(B_Y,X)$ with $\phi(B_Y)\subseteq B_X$, then $A=\widehat{\phi}^*$. 
   
    $(2)$ and $X$ has BAP $\Rightarrow(1)$ Given $y\in B_Y$, observe that $T(\delta_Y(y))\in \mathcal G_0(B_X)$ and 
    \begin{align*}
        T(\delta_Y(y))(f\cdot g)&=\delta_Y(y)(T^*(f\cdot g))=\delta_Y(y)(A(f\cdot g))= \delta_Y(y)(A(f)\cdot A(g))\\&= A(f)(y)\cdot A(g)(y) 
        = T(\delta_Y(y))(f)\cdot T(\delta_Y(y))(g),
    \end{align*}
    for every $f,g\in  \mathcal HL_0(B_X)$ since $T^*=A$ is multiplicative. Hence, $T(\delta_Y(y))\in\mathcal G_0(B_X)$ is multiplicative, so there exists $x_y\in \overline{B}_X$ such that $T(\delta_Y(y))=\delta_X(x_y)$, by Proposition~\ref{prop:multiplicativeisdelta}. Therefore, we can define a map $\phi\colon B_Y \rightarrow \overline{B}_X$ by $\phi(y)=x_y$, for every $y\in B_Y$. Given $f\in \mathcal HL_0(B_X)$, we have 
    \begin{align*}
        A(f)(y)&= \delta_Y(y)(A(f)) = \delta_Y(y)(T^*(f)) = T(\delta_Y(y))(f) \\&= \delta_X(x_y)(f) = f(x_y)= f(\phi(y)), \quad \text{for all } y\in B_Y.
    \end{align*}
    Thus, 
    \begin{equation*}
        A(f)=f\circ \phi,  \quad \text{for all }  f\in \mathcal HL_0(B_X).
    \end{equation*} 
    Furthermore, for each $x^*\in X^*\subseteq \mathcal HL_0(B_X)$ since $x^*\circ \phi=A(x^*)$ belongs to $\mathcal HL_0(B_Y)$ we derive that $\phi$ is weakly-holomorphic, which implies $\phi$ is holomorphic \cite[Theorem 8.12 (b)]{MujicaLibro}. Moreover, given $y_1,y_2\in B_Y$, we take $x^*\in S_{X^*}$ with $x^*(\phi(y_1)-\phi(y_2))=\norm{\phi(y_1)-\phi(y_2)}$ and obtain
    \begin{align*}
        \norm{\phi(y_1)-\phi(y_2)}&= |x^*(\phi(y_1))-x^*(\phi(y_2))| = |A(x^*)(y_1)-A(x^*)(y_2)|\\&\leq L(A(x^*))\norm{y_1-y_2}\leq \norm{A}\norm{y_1-y_2}. 
    \end{align*}
   Hence, $\phi$ is Lipschitz and one can easily check that $\phi(0)=0$. 
   Therefore, $\phi\in \mathcal HL_0(B_Y,X)$ and note that $\phi(B_Y)\subseteq \overline{B}_X$. 
 
Recall that, as a consequence of the Maximum Modulus Theorem,  a holomorphic map $\phi$ with $\phi(0)=0$ and $\phi(B_Y)\subseteq \overline{B}_X$ satisfies $\phi(B_Y)\subseteq B_X$ (see e.g. \cite[Exercise 5.G]{MujicaLibro}).
 We thus conclude that $A=C_\phi$ is a composition operator. 

    $(1)\Rightarrow(4)$ By the hypothesis, there is some $\phi \in \HL_0(B_Y,X)$ with $\phi(B_Y)\subseteq B_X$ such that $A=C_\phi$. 
    Fix $f\in \mathcal{H}L_0(B_X)$ and take a net $(f_\alpha)_\alpha\subseteq \mathcal{H}L_0(B_X)$ converging pointwise to $f$. Then, for all $y\in B_Y$
    \begin{equation*}
        \lim_\alpha A(f_\alpha)(y) = \lim_\alpha f_\alpha(\phi(y)) = f(\phi(y)) = A(f)(y), 
    \end{equation*}
    meaning that $A$ is $\tau_p$-$\tau_p$-continuous.

    $(1)\Rightarrow(3)$ By the hypothesis, there is some $\phi \in \HL_0(B_Y,X)$ with $\phi(B_Y)\subseteq B_X$ such that $A=C_\phi$. Fix $f\in \mathcal{H}L_0(B_X)$ and take a net $(f_\alpha)_\alpha\subseteq \mathcal{H}L_0(B_X)$ converging to $f$ uniformly on compacts of $X$.
    Let $K\subseteq B_Y$ be compact and $\varepsilon>0$.
    Then there exists $\gamma_0$ such that $\sup_{y\in K}\abs{f_\gamma \circ \phi(y)-f\circ \phi(y)}=\sup_{x\in \phi(K)} \abs{f_\gamma(x)-f(x)}<\varepsilon$ for all $\gamma\geq \gamma_0$ and we are done.

    $(4)$ or $(3)$ $\Rightarrow(2)$ 
    Using Proposition~\ref{prop:multiplicativebounded}, we see that $A$ is bounded. 

    Let $\tau=\tau_p$ if we assume (4), and $\tau=\tau_0$ if we assume (3).
    To prove that (2) holds, we
    consider the adjoint operator $A^* \colon \mathcal HL_0(B_Y)^*\rightarrow \mathcal HL_0(B_X)^*$, and define $T=A^*|_{\mathcal G_0(B_Y)}$. 
    Once we prove that $A^*(\mathcal G_0(B_Y))\subseteq \mathcal G_0(B_X)$ we derive $T^*=A$ ending the proof of (2). 
    Given $\varphi\in \mathcal G_0(B_Y)$, the point is to show that $T(\varphi)|_{\overline{B}_{\mathcal HL_0(B_X)}}$ is $\tau$-continuous (see Proposition~\ref{prop:G0def}). 
    For that, take $f\in \overline{B}_{\mathcal HL_0(B_X)}$ and a net $(f_\alpha)_\alpha\subseteq \overline{B}_{\mathcal HL_0(B_X)}$ $\tau$-converging to $f$. 
    Hence, 
    \begin{equation*}
        \lim_{\alpha} \duality{T(\varphi),f_\alpha}= \lim_\alpha \duality{A^*(\varphi),f_\alpha} =  \lim_\alpha \duality{\varphi,A(f_\alpha)} = \duality{\varphi,A(f)}=\duality{T(\varphi),f}
    \end{equation*}
    since $A$ is $\tau$-$\tau$-continuous and $\varphi|_{\norm{A}\overline{B}_{\mathcal HL_0(B_Y)}}$ is $\tau$-continuous (due to Proposition~\ref{prop:G0def}). 
    Thus, $T(\varphi)|_{\overline{B}_{\mathcal HL_0(B_X)}}$ is $\tau$-continuous, concluding the proof.
\end{proof}

\subsection{Onto and into composition operators}

Recall that, for a Lipschitz function $f\in \Lip_0(N, M)$ with associated composition operator $C_f\colon \Lip_0(M)\to \Lip_0(N)$, it is well known (see Proposition 2.25 in \cite{Weaver2ed}) that
\begin{itemize}
\item $C_f$ is onto if and only if $f$ is bi-Lipschitz onto its image.
\item $C_f$ is injective if and only if $f(N)$ is dense in $M$.
\item $C_f$ is an onto isomorphism if and only if $f$ is a bi-Lipschitz bijection between $N$ and $M$. 
\end{itemize}

Here we  analyze the corresponding statements for a function $\phi\in \mathcal HL_0(V,X)$ with $\phi(V)\subseteq U\subseteq X$ and its associated composition operator $C_\phi\colon \mathcal HL_0(U)\to\mathcal HL_0(V)$. For $U\subseteq \mathbb C$ and the unital algebra $\mathcal HL(U)$, this has been addressed recently in \cite{BM25}.

We need first to prove a lemma, which is of independent interest. 
Indeed, it seems to be unknown whether such lemma holds without the assumption that $\phi^{-1}$ is Lipschitz, see \cite{Suffridge}.

\begin{lemma}\label{lemma:biLipschitz-biHolomorphic} Let $X, Y$ be complex Banach spaces and $U\subseteq X$, $V\subseteq Y$ be open subsets. Let $\phi \colon V\to U$ be a bijective bi-Lipschitz holomorphic map. Then $\phi^{-1}$ is holomorphic. 
\end{lemma} 

\begin{proof}
It suffices to check that $y^*\circ \phi^{-1}$ is holomorphic for each $y^*\in S_{Y^*}$. Let  $x_0\in U$ and denote $y_0=\phi^{-1}(x_0)$. Note that, for each $y\in Y$, we have
\[ \norm{d\phi(y_0)(y)}=\lim_{h\to 0} \frac{\norm{\phi(y_0+hy)-\phi(y_0)}}{\norm{h}} \geq \frac{1}{L(\phi^{-1})}\norm{y}.\]
Thus, $d\phi(y_0)$ is an isomorphism onto its image. Let $T\colon d\phi(y_0)(Y)\to Y$ be its inverse. 
We have $T(d\phi(y_0)(y))=y$ and $\norm{T}\leq L(\phi^{-1})$. Consider, for $y^*\in S_{Y^*}$, the functional $y^* \circ T\colon d\phi(y_0)(Y)\to  \mathbb C$ and let $x^*\in X^*$ be a Hahn-Banach extension of $y^*\circ T$. We will show that $y^*\circ \phi^{-1}$ is holomorphic at $x_0$ with $d(y^*\circ\phi^{-1})(x_0)=x^*$. To this end, we mimic a usual argument in the proof of the Inverse Function Theorem. Since $\phi$ is holomorphic, we have 
\[\phi(y)-\phi(y_0)=d\phi(y_0)(y-y_0)+\norm{y-y_0}\Phi(y), \quad \forall y\in V\]
where $\Phi(y)\to 0$ as $y\to y_0$. Thus, for $x\in U$ we have 
\[ x- x_0 = d\phi(y_0)(\phi^{-1}(x)-\phi^{-1}(x_0)) + \norm{\phi^{-1}(x)-\phi^{-1}(x_0)} \Phi(\phi^{-1}(x))\]
and so
\begin{align*} x^*(x-x_0)&= (x^*\circ d\phi(y_0))(\phi^{-1}(x)-\phi^{-1}(x_0))+ \norm{\phi^{-1}(x)-\phi^{-1}(x_0)} (x^*\circ \Phi)(\phi^{-1}(x)) \\
&= y^*(\phi^{-1}(x)-\phi^{-1}(x_0))+ \norm{\phi^{-1}(x)-\phi^{-1}(x_0)} (x^*\circ \Phi)(\phi^{-1}(x)).
\end{align*}
Therefore
\begin{align*} \frac{|(y^*\circ \phi^{-1})(x)-(y^*\circ \phi^{-1})(x_0)-x^*(x-x_0)|}{\norm{x-x_0}} &= \frac{\norm{\phi^{-1}(x)-\phi^{-1}(x_0)}}{\norm{x-x_0}}|(x^*\circ \Phi)(\phi^{-1}(x))| \\
&\leq  L(\phi^{-1}) |(x^*\circ \Phi)(\phi^{-1}(x))|
\end{align*}
which goes to $0$ as $x\to x_0$ by the definition of $\Phi$. We conclude that $\phi^{-1}$ is holomorphic at every $x_0\in U$.  
\end{proof}

Now we present some necessary or sufficient conditions for the injectivity or surjectivity of a composition operator. 
Recall that it is a standard fact that a bounded linear operator $T$ is an  isomorphic embedding (that is, injective with closed range) if and only if $T^*$ is onto, see e.g. Exercise 2.49 in \cite{checos}.

\begin{proposition}\label{p:CompositionOnto}
    Let $X,Y$ be complex Banach spaces and $U\subseteq X$, $V\subseteq Y$ be open bounded subsets containing $0$. Let $\phi\in \mathcal HL_0(V,X)$ with $\phi(V)\subseteq U$. 
    \begin{itemize}
     \item[(a)] If  $C_\phi$ is onto (equiv. $\widehat{\phi}$ is an isomorphism onto its image),  then $\phi$ is bi-Lipschitz onto its image and $\norm{d\phi(y_1)(y_2)}\geq L(\phi^{-1})^{-1} $ for every $y_1\in V$ and $y_2\in S_Y$. 
     \item[(b)] If $\phi(V)$ is dense in $U$, or $\phi(V)$ has non-empty interior and $U$ is connected, then $C_\phi$ is injective.
     \item[(c)] If $C_\phi$ is injective, then $\phi(V)$ has dense span.
     \item[(d)] If $\phi$ is a bi-Lipschitz bijection between $V$ and $U$, then $\phi^{-1}$ is holomorphic and $C_\phi$ is an onto isomorphism. 
     \item[(e)] If $U$ and $V$ are connected, $\phi(V)=U$ and $C_\phi$ is an onto isometry, then $\phi$ is the restriction of an onto linear isometry $T\colon Y\to X$.
     \end{itemize}
\end{proposition}

\begin{proof}
    (a) Given $y_1,y_2\in V$, consider $y^*\in S_{Y^*}$ with $y^*(y_1-y_2)=\norm{y_1-y_2}$. Then, since $y^*\in \mathcal HL_0(V)$ and $C_\phi$ is onto, we can find some $f\in \mathcal HL_0(U)$ with $C_\phi(f)=y^*$. So,
    \begin{align*}
        \norm{y_1-y_2}&=y^*(y_1-y_2)=y^*(y_1)-y^*(y_2) = C_\phi(f)(y_1)-C_\phi(f)(y_2)\\&= f(\phi(y_1))-f(\phi(y_2)) \leq L(f)\norm{\phi(y_1)-\phi(y_2)}.
    \end{align*}
    Moreover, by the Open Mapping Theorem, we may choose each $f$ so that $L(f)\leq C$ for some constant $C>0$,  since every $y^*$ under consideration has norm one. Thus $L(\phi^{-1})\leq C$. 
    Finally, note that
    \[\norm{d\phi(y_1)(y_2)}= \lim_{h\to0} \norm{\frac{\phi(y_1+hy_2)-\phi(y_1)}{h}}\geq \frac{1}{L(\phi^{-1})} \quad \forall y_1\in V, y_2\in S_Y.\]
   
    (b) Let $f\in \mathcal HL_0(U)$ with $f\neq 0$. If $\phi(V)$ is dense in $U$, there is $y\in V$ such that $C_\phi(f)(y)=f(\phi(y))\neq 0$. Thus $\ker C_\phi=0$.   

    On the other hand, assume that there is an open set $U'\subseteq \phi(V)$. Then any $f\in \HL_0(U)$ such that $C_\phi(f)=0$ satisfies $f|_{U'}=0$ and so $f=0$ (\cite[Proposition~5.7]{MujicaLibro}). 

    (c) If $Z=\overline{\operatorname{span}} \phi(V)\neq X$, then by Hahn-Banach, there is a functional $0\neq x^*\in X^*$ such that $x^*\equiv 0$ on $Z$. Therefore,  $C_\phi(x^*)= x^*\circ \phi=0$ and hence $C_\phi$ is not injective.

    (d) We know from Lemma \ref{lemma:biLipschitz-biHolomorphic}  that $\phi^{-1}\colon U\to V$ is a holomorphic map. Observing that $$C_{\phi^{-1}} \circ C_{\phi} = \widehat{\phi^{-1}}^*\circ \widehat{\phi}^*=(\widehat{\phi}\circ \widehat{\phi^{-1}})^*= \widehat{(\phi\circ\phi^{-1})}^*=\widehat{Id_U}^*=Id_{\mathcal HL_0(U)},$$ and $$C_{\phi} \circ C_{\phi^{-1}} = \widehat{\phi}^*\circ \widehat{\phi^{-1}}^*=(\widehat{\phi^{-1}}\circ \widehat{\phi})^*= \widehat{(\phi^{-1}\circ\phi)}^*=\widehat{Id_V}^*=Id_{\mathcal HL_0(V)},$$
    we obtain the desired conclusion.

    (e) If $C_\phi$ is an onto isometry, so is $\widehat{\phi}$. Moreover,
    \[ \norm{\phi(x)-\phi(y)}=\norm{\widehat{\phi}(\delta(x))-\widehat{\phi}(\delta(y))}=\norm{\delta(x)-\delta(y)}=\norm{x-y}, \quad \forall x,y\in V,\]
    that is, $\phi\colon V\to \phi(V)=U$ is an isometry with $\phi(0)=0$. By Mankiewicz theorem \cite{Mankiewicz}, there is a $\mathbb R$-linear map $T\colon Y\to X$ such that $\phi=T|_V$. Finally, it is easy to check that $d\phi(0)= T$, so actually $T$ is $\mathbb C$-linear.
\end{proof}

\begin{example} Let us see that the converses of  Proposition \ref{p:CompositionOnto} (a) and (c) do not hold. 
\begin{enumerate}
    \item[(1)]  There is a bi-Lipschitz onto its image mapping $\phi$ with $C_\phi$ not onto.

    Let $U=\mathbb D$ and $V=\frac{1}{2}\mathbb D$, and consider $\phi\colon V\to U$ the identity map, which is clearly bi-Lipschitz onto its image. Then $C_\phi(f) = f|_V$ for each $f\in \mathcal HL_0(U)$. Let $g\in \mathcal HL_0(V)$ be given by $g(z)=\frac{z}{1-z}$. Then the only holomorphic function $f\colon \mathbb D\to\mathbb C$ such that $f|_V=g$ is $f(z)=\frac{z}{1-z}$, which is not Lipschitz in $U$. Thus, $C_\phi$ is not onto.

    \item[(2)] There is a mapping $\phi$ with dense image span and not injective $C_\phi$.

    Let $U=B_{\C^2}$ (with the sup norm) and $V=\mathbb D$, and consider $\phi\colon \mathbb D \to \C^2$ given by $\phi(z)=(z,z^2/2)$. Then, $\phi\in\mathcal HL_0(\disk,\C^2)$, $\phi(V)\subseteq U$ and $\{\phi(1),\phi(i)\}$ is a basis of $\C^2$ but $C_\phi f=0$ for $f(z_1,z_2)=z_1^2-2 z_2$.
\end{enumerate}
  
\end{example}

Next, we come back to mappings defined on unit balls and we characterize  when $C_\phi$ is an onto isomorphism under the hypothesis of BAP, as a consequence of Theorem~\ref{thm:compositionoperatorequivalence}. 

\begin{corollary}\label{c:BAP-characterizationOfCompositionIsomorphism}
    Let $X,Y$ be complex Banach spaces such that $Y$ has the BAP. Let $\phi\in \mathcal HL_0(B_Y,X)$ with $\phi(B_Y)\subseteq B_X$. Then, the following assertions are equivalent.
    \begin{enumerate}
        \item $C_\phi\colon  \mathcal HL_0(B_X)\rightarrow \mathcal HL_0(B_Y)$ is an onto isometry.
        \item $C_\phi \colon \mathcal HL_0(B_X)\rightarrow \mathcal HL_0(B_Y)$ is an onto isomorphism.
        \item $\phi \colon B_Y \rightarrow B_X$ is a bijection and $\phi^{-1}\in \mathcal HL_0(B_X,Y)$.
        \item $\phi$ is the restriction of an onto linear isometry $T\colon Y\to X$. 
    \end{enumerate}
\end{corollary}
    
\begin{proof}
    (1)$\Rightarrow$(2) is clear, and it is easy to check that (4)$\Rightarrow$(1).

    $(2)\Rightarrow(3)$ Since $C_\phi$ is an onto isomorphism there is a linear and continuous map $A\colon\mathcal HL_0(B_Y)\rightarrow \mathcal HL_0(B_X)$ such that $A\circ C_\phi=Id_{\mathcal HL_0(B_X)}$ and $C_\phi \circ A=Id_{\mathcal HL_0(B_Y)}$. It is easy to see that $A$ is multiplicative. Indeed, given $f,g\in \mathcal HL_0(B_Y)$, take $h,i \in \mathcal HL_0(B_X)$ such that $C_\phi(h)=f$ and $C_\phi(i)=g$. Hence,
    \begin{equation*}
        A(f\cdot g)=A(C_\phi(h)C_\phi(i))= A(C_\phi(h\cdot i)) = h\cdot i = A(C_\phi(h))A(C_\phi(i))=A(f)A(g).
    \end{equation*}
    Since $C_\phi$ is an onto isomorphism, we have that $\widehat{\phi}$ is an onto isomorphism too (see Exercise 2.41 in~\cite{checos}).
    Thus $(\widehat{\phi}^{-1})^*=A$.
  
    Therefore, $A=C_\varphi$ for some $\varphi\in \mathcal HL_0(B_X,Y)$ with $\varphi(B_X)\subseteq B_Y$, thanks to Theorem~\ref{thm:compositionoperatorequivalence}. Hence, $$C_{Id_{B_X}}= Id_{\mathcal HL_0(B_X)} = C_\varphi \circ C_\phi=C_{\phi \circ \varphi }$$ and $$C_{Id_{B_Y}}= Id_{\mathcal HL_0(B_Y)} = C_\phi \circ C_\varphi=C_{\varphi \circ \phi }.$$ So, $\phi\circ \varphi = Id_{B_X}$, and $\varphi \circ \phi = Id_{B_Y}$, which proves (3).
    
    (3)$\Rightarrow$(4) is a known consequence of Cartan's linearity theorem, see e.g. \cite[Lemma~2.1]{Arazy}, since $\phi(0)=0$.
\end{proof}

\medskip

\section{Compactness of composition operators on \texorpdfstring{$\mathcal HL_0$}{HL0}}\label{Section:Compactness}

The compactness of composition operators has been extensively studied in the settings of Lipschitz spaces and spaces of bounded holomorphic functions.
In the Lipschitz framework, the functions $f\colon M\to N$ for which the associated composition operator $C_f\colon \Lip_0(M)\to \Lip_0(N)$ is compact were characterized in \cite{KamSch, JVVV} for certain metric spaces $M$ and $N$, and later in full generality in \cite{ACPcomp}. 
Furthermore, the compactness of $C_f$  turns out to be equivalent to several other properties, including weak compactness, strict singularity, cosingularity, and $\ell_\infty$-strict singularity
\cite{ACPcomp,Lemay}.

Along the same lines, for $\phi\in\mathcal H^\infty(B_Y,X)$ satisfying $\phi(B_Y)\subseteq B_X$, it is shown in \cite[Proposition~3]{AGL} that the composition operator $C_\phi\colon \mathcal H^\infty(B_X)\rightarrow \mathcal H^\infty(B_Y)$ is compact if and only if $\|\phi\|_\infty <1$ and $\phi(B_Y)$ is a relatively compact set in $X$. 
Moreover, whenever $\phi(B_Y)$ is relatively compact and $C_\phi$ is weakly compact, then $C_\phi$ is actually compact. 
Motivated by these results, we address the corresponding problem in our setting.

We begin with some general observations concerning compactness properties of composition operators and their preadjoints. 
Since $\overline{B}_{\mathcal G_0(V)}$ is the absolutely closed convex hull of the elementary molecules $\frac{\delta(x)-\delta(y)}{\norm{x-y}}$, we obtain the following characterization, analogous to those in \cite[Proposition 3.4]{Mujica} and \cite[Theorem 2.3]{CPJV} (see also \cite[Proposition~2.1]{ACPcomp}).  

\begin{proposition}\label{prop:compunitball}
 Let $U, V$ be bounded open subsets of  complex Banach spaces $X, Y$ containing $0$, and $\phi\in \mathcal HL_0(V, X)$ with $\phi(V)\subseteq U$. The following statements are equivalent:
\begin{itemize}
\item[(1)] $C_\phi\colon \mathcal HL_0(U)\to \mathcal HL_0(V)$ is a compact (resp. weakly compact) operator.
\item[(2)] $\widehat{\phi}\colon \mathcal G_0(V)\to \mathcal G_0(U)$ is a compact (resp. weakly compact) operator.
\item[(3)] The set 
\[\left\{\frac{\delta(\phi(y_1))-\delta(\phi(y_2))}{\norm{y_1-y_2}}: y_1, y_2\in V, y_1\neq y_2\right\}\]
is relatively compact (resp. relatively weakly compact) in $\mathcal G_0(U)$. 
\end{itemize}
If moreover $V$ is convex, they are also equivalent to:
\begin{itemize}
\item[(4)] The set
\[\left\{\widehat\phi(d\delta(y_1)(y_2)): y_1\in V, y_2\in S_Y\right\}\]
is relatively compact (resp. relatively weakly compact) in $\mathcal G_0(U)$. 
\end{itemize}
\end{proposition}

\begin{proof}
Just follow the proof of Proposition 2.1 in \cite{ACPcomp}, using the expressions for $\overline{B}_{\mathcal G_0(V)}$ obtained in Proposition \ref{prop:unitballconv}. 
\end{proof}

Under the assumption that $V$ is convex and $\overline{\phi(V)}\subseteq U$, we can get the following characterization. 

\begin{proposition}\label{prop:compRelComp}
     Let $U, V$ be bounded open subsets of  complex Banach spaces $X, Y$ containing $0$, and $\phi\in \mathcal HL_0(V,X)$. Assume $V$ is convex and  
   $\overline{\phi(V)}\subseteq U$. Then the following assertions are equivalent:
    \begin{itemize}
\item[(1)] $C_\phi\colon \mathcal HL_0(U)\to \mathcal HL_0(V)$ is compact. 
\item[(2)] $\phi(V)$ is relatively compact and $d\phi(V)(S_Y)$ is relatively compact.
    \end{itemize}
\end{proposition}

\begin{proof}
(2) $\Rightarrow$ (1). For $y_1\in V$ and $y_2\in S_Y$ we have
\[ \widehat{\phi}(d\delta(y_1)(y_2))=d(\widehat{\phi}\circ \delta)(y_1)(y_2)=d(\delta\circ \phi)(y_1)(y_2) = d\delta(\phi(y_1))(d\phi(y_1)(y_2)).\]    
Therefore,
\begin{align*}
\left\{ \widehat{\phi}(d\delta(y_1)(y_2)):y_1\in V, y_2\in S_Y\right\} \subseteq d\delta(\overline{\phi(V)})(\overline{d\phi(V)(S_Y)}),
\end{align*}
which is compact since it is the image of the compact set $\overline{\phi(V)}\times \overline{d\phi(V)(S_Y)}\subseteq U\times X$ under the continuous map $(y_1,y_2)\mapsto d\delta(y_1)(y_2)$.
By Proposition \ref{prop:compunitball}.(4), we conclude that $C_\phi$ is compact.

(1) $\Rightarrow$ (2). We have the commutative diagram 
\[\begin{tikzcd}[ampersand replacement=\&]
	V \&\& X \\
	{\mathcal G_0(V)} \&\& {\mathcal G_0(U)}
	\arrow["\phi", from=1-1, to=1-3]
	\arrow["{\delta_Y}"', from=1-1, to=2-1]
	\arrow["{\widehat{\phi}}"', from=2-1, to=2-3]
	\arrow["{\beta_X}"', from=2-3, to=1-3]
\end{tikzcd}\]
where $\beta_X:=T_{Id_{U}}$ is the canonical quotient onto $X$ such that $\beta_X\circ \delta_X=Id_{U}$. We have $\phi(V)=\beta_X\circ \widehat{\phi}\circ \delta_Y(V)$. 
Since $\delta_Y(V)$ is bounded, and $\widehat{\phi}$ is compact, we get that $\phi(V)$ is relatively compact. 

Also, taking differentials on $\beta_X\circ \delta_X=I_{U}$ we have that $\beta_X(d\delta_X(x_1)(x_2))=x_2$ for each $x_1\in U$ and $x_2\in X$. Thus
\begin{align*} d\phi(V)(S_Y)
&= \beta_X(\{d\delta_X(\phi(y_1))(d\phi(y_1)(y_2)): y_1\in V, y_2\in S_Y\})
\end{align*}
Thanks to Proposition \ref{prop:compunitball}, we conclude that $d\phi(V)(S_Y)$ is relatively compact. 
\end{proof}

Let us point out that the assumption $\overline{\phi(V)}\subseteq U$ is needed in the proof of (2)~$\Rightarrow$~(1). Although there  are situations in which (1) already implies that $\overline{\phi(V)}\subseteq U$ (see Corollary~\ref{cor:compactfindimcase} and Lemma~\ref{lemma:Tfixeslinfty}), it is not clear to us whether this must hold in general.

Under the same hypothesis as in the previous proposition, we also have  some information about the weak compactness of $C_\phi$.

\begin{proposition}\label{prop:compRelWComp}
     Let $U, V$ be bounded open subsets of  complex Banach spaces $X, Y$ containing $0$, and $\phi\in \mathcal HL_0(V,X)$. Assume $V$ is convex and  
   $\overline{\phi(V)}\subseteq U$. 
    \begin{itemize}
\item[(a)] If $C_\phi\colon \mathcal HL_0(U)\to \mathcal HL_0(V)$ is weakly compact, then $\phi(V)$ and $d\phi(V)(S_Y)$ are relatively weakly compact sets. 
\item[(b)] If $\phi(V)$ is relatively compact and $d\phi(V)(S_Y)$ is relatively weakly compact, then \\ $C_\phi\colon \mathcal HL_0(U)\to \mathcal HL_0(V)$ is weakly compact.
    \end{itemize}   
\end{proposition}

In spite of the previous statements, note that there exists $\phi$ such that both $\phi(V)$ and $d\phi(V)(S_Y)$ are relatively weakly compact while $C_\phi^L$ fails to be weakly compact, see Example \ref{ex:ell2x/2}. 

Observe also that if $X$ is reflexive, then the condition that $d\phi(V)(S_Y)$ is relatively weakly compact in Proposition \ref{prop:compRelWComp} (b) is automatically satisfied. 

\begin{proof}
   (a) Assume $C_\phi\colon \mathcal HL_0(U)\to \mathcal HL_0(V)$ is weakly compact. As in the proof of Proposition \ref{prop:compRelComp}, we have $\phi=\beta_X\circ \widehat{\phi}\circ \delta_Y$ and so  it follows that $\phi(V)$ is relatively weakly compact. Analogously, $d\phi(V)(S_Y)=\beta_X(\{\widehat{\phi}(d\delta(y_1)(y_2)): y_1\in V, y_2\in S_Y\})$ is relatively weakly compact as being the image of a relatively weakly compact set by a linear map.

   (b) We will show that the map $d\delta\colon (U, \norm{\cdot})\times (X, w)\to (\mathcal G_0(U),w)$ given by $(x_1,x_2)\mapsto d\delta(x_1)(x_2)$ is continuous. Indeed, assume that $u_\alpha\to u$ in norm and $x_\alpha\to x$ weakly.  For $f\in \mathcal HL_0(U)$ we have that $df\colon (U, \norm{\cdot})\to X^*$ is continuous. Thus $\norm{df(u_\alpha)-df(u)}\to 0$. It follows that
   \begin{align*}
    |\langle f, d\delta(u_\alpha)(x_\alpha)-d\delta(u)(x)\rangle| &= |df(u_\alpha)(x_\alpha)-df(u)(x)|\\
    &\leq \norm{df(u_\alpha)-df(u)}\norm{x_\alpha} + |df(u)(x_\alpha-x)|\to 0.
   \end{align*}
   This proves the claim. Now, the same argument as in the proof of the previous proposition shows that 
\[   
\left\{ \widehat{\phi}(d\delta(y_1)(y_2)):y_1\in V, y_2\in S_Y\right\} \subseteq d\delta\left(\overline{\phi(V)}\right)\left(\overline{d\phi(V)(S_Y)}^w\right)
\]
and we may apply Proposition \ref{prop:compunitball} to finish the proof. 
\end{proof}

In this setting, given $\phi\in \mathcal HL_0(V,X)$ with $\phi(V)\subseteq U$  we may consider the composition operator acting on $\mathcal HL_0$ spaces: $C_\phi \colon \mathcal HL_0(U)\rightarrow \mathcal HL_0(V)$ as well as  on $\mathcal H^\infty$ spaces: $C_\phi\colon \mathcal H^\infty(U)\rightarrow \mathcal H^\infty(V)$. To avoid ambiguity, we will use the notations $C_\phi^L$ and $C_\phi^\infty$ whenever necessary.
As we shall see, these two settings are closely related. To study this connection, we also consider the composition operator $C_\phi^{X^*}\colon \mathcal H^\infty(U, X^*)\to \mathcal H^\infty(V, X^*)$. 

\begin{lemma}\label{lemma:factor} Let $U, V$ be bounded open subsets of  complex Banach spaces $X, Y$ containing $0$, and $\phi\in \mathcal HL_0(V, X)$ with $\phi(V)\subseteq U$.  
Then, the following diagram commutes:

\[\begin{tikzcd}
	{\mathcal HL_0(U)} &&&& {\mathcal HL_0(V)} \\
	{\mathcal H^\infty(U,X^*)} && {\mathcal H^\infty(V,X^*)} && {\mathcal H^\infty(V, Y^*)}
	\arrow["{C_\phi^L}", from=1-1, to=1-5]
	\arrow["d"', from=1-1, to=2-1]
	\arrow["d", from=1-5, to=2-5]
	\arrow["{C_\phi^{X^*}}"', from=2-1, to=2-3]
	\arrow["M"', from=2-3, to=2-5]
\end{tikzcd}\]
where $d$ maps each function to its differential, and $M\colon \mathcal H^\infty(V, X^*)\to \mathcal H^\infty(V, Y^*)$ is given by $M(g)(y)= g(y)\circ d\phi(y)$. 
In particular if V is convex then
\[
C_\phi^L=d^{-1}\circ M \circ C_\phi^{X^*}\circ d
\]
where $d^{-1}$ is the isometry defined on $d(\HL_0(V))$ equipped with the norm induced by $\mathcal H^\infty(V,Y^*)$.

\end{lemma}

\begin{proof}
Note that for $g\in \mathcal H^\infty(V, X^*)$ and $y\in V$ we have \[\norm{M(g)(y)} \leq \norm{g(y)}_{X^*} \norm{d\phi(y)}\leq \norm{d\phi}_\infty\cdot \norm{g}_\infty=L(\phi)\norm{g}_\infty,\] so $M$ is bounded with $\|M\|\le L(\phi)$. 

Now, for $f\in \mathcal HL_0(U)$ and $y\in V$ we have
\begin{align*} (d\circ C_\phi^L)(f)(y)&= d(f\circ \phi)(y)= df(\phi(y))\circ d\phi(y)\\& = M(df\circ \phi)(y)
= \left(M \circ C_\phi^{X^*} \circ d\right)(f)(y),
\end{align*}
that is, the diagram is commutative.
\end{proof}

\begin{corollary}\label{cor:ideals}
Assume that $V$ is convex and  let $\mathcal I$ be 
a class of operators satisfying the ideal property.
    If $C_\phi^{X^*} \in \mathcal I$ then $C_\phi^L \in  {\mathcal I}$.
\end{corollary}

\begin{remark}\label{rem:adjoints} Recall that the linearization process $f\mapsto T_f$ defines an isometry from $\mathcal H^\infty (U, X^*)$ onto $\mathcal L(\mathcal G^\infty(U), X^*)= (\mathcal G^\infty(U)\pten X)^*$. The operators that appear in the diagram in Lemma \ref{lemma:factor} are adjoint operators. Indeed, one can easily check that:
\begin{itemize}
    \item $d\colon \mathcal HL_0(U)\to\mathcal H^\infty(U,X^*)$ is the adjoint of the quotient operator  
    \begin{align*}
d_*\colon \mathcal G^\infty(U)\pten X&\to \mathcal G_0(U)\\
\delta(x)\otimes u&\mapsto d\delta(x)(u)\end{align*} 
(see \cite[Proposition 2.7]{ADGLM}). 
    \item $C_\phi^{X^*}\colon \mathcal H^\infty(U, X^*)\to \mathcal H^\infty(V, X^*)$ is the adjoint of the operator 
    \begin{align*}
        \widehat{\phi}\otimes I\colon \mathcal G^\infty(V)\pten X\to \mathcal G^\infty(U)\pten X.
    \end{align*}
Indeed, denoting $T=\widehat{\phi}\otimes I$, we have:
\begin{align*} 
    \duality{T^*f,\delta(y)\otimes x}
    &= \langle f, T(\delta(y)\otimes x)\rangle = \duality{ f, \widehat{\phi}(\delta(y))\otimes x} = \langle f, \delta(\phi(y))\otimes x)\rangle \\
    &= 
   \duality{f(\phi(y)),x}
    = \duality{(C_\phi^{X^*}f)(y),x}.
    \end{align*}
\item $M\colon \mathcal H^\infty(V, X^*)\to \mathcal H^\infty (V,Y^*)$ is the adjoint of the operator 
\begin{align*}
M_*\colon\mathcal G^\infty(V)\pten Y&\to  \mathcal G^\infty (V)\pten X\\
\delta(y)\otimes v &\mapsto \delta(y)\otimes d\phi(y)(v)=\delta(y)\otimes \left(\widehat{d\phi}(\delta(y))(v)\right)
\end{align*}
\end{itemize}
Predualizing, we get the commutative diagram:
\[\begin{tikzcd}[ampersand replacement=\&]
	{\mathcal G_0(V)} \&\&\&\& {\mathcal G_0(U)} \\
	{\mathcal G^\infty(V)\widehat{\otimes}_\pi Y} \&\& {\mathcal G^\infty(V)\widehat{\otimes}_\pi X} \&\& {\mathcal G^\infty(U)\widehat{\otimes}_\pi X}
	\arrow["{\widehat{\phi}}", from=1-1, to=1-5]
	\arrow["{d_*}", from=2-1, to=1-1]
	\arrow["{M_*}", from=2-1, to=2-3]
	\arrow["{\widehat{\phi}\otimes I}", from=2-3, to=2-5]
	\arrow["{d_*}"', from=2-5, to=1-5]
\end{tikzcd}\]
\end{remark}

We isolate some typical situations where we can say more about (weak) compactness of  the operator $C_\phi^{X^*}\colon \mathcal H^\infty(U, X^*)\to \mathcal H^\infty(V, X^*)$ in terms of the same property of the operator $C_\phi^\infty$.

\begin{lemma}\label{lemma:CphiLvsInfty} Let $U, V$ be bounded open subsets of  complex Banach spaces $X, Y$, and $\phi\in \mathcal HL(V, X)$ with $\phi(V)\subseteq U$, $\phi\neq 0$.  Then,
\begin{itemize}
\item[(a)] $C_\phi^{X^*}$ is compact if and only if $C_\phi^\infty$ is compact and $\dim(X)<+\infty$.
\item[(b)] If $C_\phi^{X^*}$ is weakly compact, then $C_\phi^{\infty}$ is weakly compact and $X$ is reflexive.
\item[(c)] If $C_\phi^\infty$ is compact and $X$ is reflexive, then $C_\phi^{X^*}$ is weakly compact.
\item[(d)] If  $C_\phi^\infty$ is weakly compact and $X$ is finite dimensional, then $C_\phi^{X^*}$ is weakly compact.
\end{itemize}
\end{lemma}

\begin{proof}
Consider $\widehat{\phi}\colon \mathcal G^\infty(V)\to\mathcal G^\infty(U)$. We have already mentioned that $C_\phi^{X^*}=(\widehat{\phi}\otimes I)^*$. Thus the result follows from Lemma \ref{lemma:comptensorop} and the theorems of Gantmacher and  Schauder.
\end{proof}

\subsection{The scalar case \texorpdfstring{$X=Y=\C$}{X=Y=C}}
Shapiro~\cite{shapiro} showed that $C_\phi^L : \mathcal{H}L_0(\disk)\rightarrow \mathcal HL_0(\disk)$ is compact if and only if $\norm{\phi}_\infty<1$ (see also \cite{BM05} for an extension to other subsets of $\mathbb C$).
Together with the previously cited~\cite[Proposition~3]{AGL}, it is  readily seen that this happens if and only if $C_\phi^\infty\colon \mathcal H^\infty(\disk)\to \mathcal H^\infty(\disk)$ is (weakly) compact.  
Further, this is equivalent to the fact that $C_\phi^\infty\colon \mathcal H^\infty(\disk)\to \mathcal H^\infty(\disk)$ fixes no copy of $\ell_\infty$
among several other related properties, see Figure~\ref{fig:Ideals}. 
This has been proved by Bonet, Doma\'nski and Lindstr\"om~\cite{BDL} but it follows in this setting also from a general result of Bourgain~\cite[Theorem~1]{Bourgain} which claims the following: \emph{Let $Y$ be any Banach space and let $T\colon \mathcal H^\infty(\disk)\to Y$ be non-weakly-compact. Then $T$ fixes a copy of $\ell_\infty$}. 
Since the spaces $\HL_0(\disk)$ and $\mathcal H^\infty(\disk)$ are (isometrically) isomorphic the same is true for operators $T\colon\HL_0(\disk)\to Y$.

We will strengthen Shapiro's result showing that if $C_\phi^L$ fixes no copy of $\ell_\infty$ then $\norm{\phi}_\infty<1$. To this end, we need some auxiliary results. The first one is a classical theorem of Denjoy and Wolff which can be found in \cite{CowenPommerenke}. 

For a holomorphic map $\phi\colon \disk\rightarrow \disk$, we say that $z\in \overline{\disk}$ is a \emph{fixed point} if $\displaystyle\lim_{\xi \to1^{-}}\phi(\xi z)=z$. Clearly, if $z\in \disk$ this is equivalent to saying that $\phi(z)=z$, but when $z\in\torus$ this is subtly different from being a fixed point in the usual sense, since we cannot guarantee that $\phi$ has a continuous extension to the whole $\overline{\disk}$. 

Furthermore, if $z\in \overline{\disk}$ is a fixed point of $\phi$, then the limit  $\displaystyle\lim_{\xi\to 1^-} \phi'(\xi z)$ 
exists as a consequence of the Julia-Wolff-Carathéodory theorem (see \cite{CowenPommerenke}).

\begin{theorem}[Denjoy-Wolff]\label{thm:denjoy-wolff}
    If $\phi\colon \disk \rightarrow \disk$ is holomorphic and it is not the identity, then $\phi$ has a unique fixed point  $a\in \overline{\disk}$ such that $\displaystyle\lim_{\xi \to 1^-} |\phi'(\xi a)|\leq1$.
\end{theorem}

If, in addition, we take $\phi \in \mathcal HL_0(\disk)$, then $\phi$ admits a Lipschitz extension to $\overline{\disk}$, and moreover $0$ is a fixed point of $\phi$. Consequently, we obtain the following corollary.

\begin{corollary}\label{cor:denjoy-wolff-modificado}
    Let $\phi\in \mathcal HL_0(\disk)$ such that $\phi(\disk)\subseteq\disk$ and $\phi$ is not a rotation. If there is some $a\in \C$ with $|\phi(a)|=1=|a|$, then $\displaystyle\lim_{\xi \to 1^{-}}|\phi'(\xi a)|>1.$  
\end{corollary}

\begin{proof}
    Assume there exists $a\in \torus$ such that $\phi(a)\in \torus$ and take $\theta\in \R$ satisfying $e^{i\theta}\phi(a)=a$. Consider the function $\varphi \colon \disk \rightarrow \disk$ given by $\varphi(z)=e^{i\theta}\phi(z)$ for all $z\in \disk$. Clearly, $\varphi\in \mathcal HL_0(\disk)$ and $\varphi(a)=a$. Moreover $\varphi\neq I$ since $\phi$ is not a rotation, $\varphi(0)=0$, and $|\varphi'(0)|=\displaystyle\lim_{z\to0} \frac{|\varphi(z)|}{|z|}\leq 1$ due to the Schwarz lemma. Hence, Theorem~\ref{thm:denjoy-wolff} implies that $\displaystyle\lim_{\xi\to 1^{-}}|\varphi'(\xi a)|>1$. The result follows since $|\varphi'(z)|= |\phi'(z)|$ for all $z\in \disk$.
\end{proof}

We thus obtain the following characterization of composition operators on $\mathcal HL_0(\mathbb D)$. Observe that the proof below does not rely on the deep theorem of  Bourgain mentioned above.

\begin{theorem}\label{thm:compactscalarcase}
    Let $\phi \in \mathcal{H}L_0(\disk)$ such that $\phi(\disk)\subseteq \disk$.
    Then, the following assertions are equivalent. 
    \begin{enumerate}
    \item $C_\phi^L \colon \mathcal HL_0(\disk)\rightarrow \mathcal HL_0(\disk)$ is compact.
    \item $C_\phi^L \colon \mathcal HL_0(\disk)\rightarrow \mathcal HL_0(\disk)$ does not fix any copy of $\ell_\infty$.
    \item $\norm{\phi}_\infty<1$.    
    \item $C_\phi^\infty \colon \mathcal H^\infty(\disk)\rightarrow \mathcal H^\infty(\disk)$ is compact.
\end{enumerate}
\end{theorem}

\begin{proof}
    First, observe the equivalence $(3)\Leftrightarrow (4)$ is well known (see, for instance, \cite[Exercise 2.6.10]{shapiro-libro}). 
    Also, the implication $(4)\Rightarrow(1)$
    is exactly the Corollary~\ref{cor:ideals} for $X=Y=\C$ and the ideal of compact operators; and the implication $(1)\Rightarrow (2)$ is obvious.

    So the main task is to prove         
    $(2) \Rightarrow (3)$. Suppose that $\norm{\phi}_\infty=1$. 
    If $\phi$  is a rotation, then $C_\phi^L$ is an onto isomorphism that clearly fixes every copy of $\ell_\infty$ contained in $\HL_0(\disk)$. That such a copy indeed exists is proved in \cite{ADGLM}.
    We proceed assuming that $\phi$ is not a rotation.
    By the compactness of the closed disk, there is some $y_0\in \torus$ with $|\phi(y_0)|=1$. We pick a sequence $(\xi_n)_n\subseteq (0,1)$ converging to $1$, and consider the sequence $(y_n)_n\subseteq \disk$ given by $y_n=\xi_n y_0$ for all $n\in\N$. 
    We know that $\lim_{n\to\infty} |\phi'(y_n)|> 1$ by Corollary \ref{cor:denjoy-wolff-modificado}, and this will lead to the desired conclusion.
    First, since $\lim_{n\to\infty} |\phi(y_n)|=1$, we may assume (see e.g. \cite[Chapter VII, Section 1]{Garnett}) that $(\phi(y_n))_n$ is an interpolating sequence for $\mathcal H^\infty(\disk)$ (up to taking a subsequence if necessary), that is, the operator $S_{(\phi(y_n))_n}\colon\mathcal H^\infty(\disk)\to \ell_\infty$ given by $f\mapsto (f(\phi(y_n))_n)$ is onto. 
    Thus, following temporarily the corresponding part of the proof of~\cite[Theorem 1]{BDL}, by the proof of ``b) implies c)'' of Theorem~III.E.4 in \cite{Wojtaszczyk}, there is a sequence $(h_k)_k\subseteq \mathcal H^\infty(\disk)$ such that $(S_{\phi(y_n)_n}h_k)_k$ is the unit vector basis in $\ell_\infty$ and there is a constant $M>0$ such that 
    \[
    \sum_{k=1}^\infty \abs{h_k(z)}\leq M, \quad \mbox{ for every }z\in \disk.
    \]
    We can therefore define a linear map $T:\ell_\infty \to \mathcal H^\infty(\disk)$ as
    $T\xi=\sum_{k=1}^\infty \xi_k h_k$.
    We have 
    \[
    \norm{T\xi}_\infty \leq M\norm{\xi}_\infty 
    \]
    and so $\norm{T}\leq M$.
    We clearly have that $S_{(\phi(y_n))_n}\circ T=Id_{\ell_\infty}$. 
    So $T$ is an isomorphic embedding.

    The upper part of the diagram below is obtained from Lemma~\ref{lemma:factor} applied to the case $X=Y=\C$.
    In the lower part $M_{\phi'}\colon \mathcal H^\infty(\mathbb D) \to \mathcal H^\infty(\mathbb D)$ denotes the multiplication operator defined by $f\mapsto f\cdot \phi'$, whereas $M_{\phi'(y_n)^{-1}}:\ell_\infty \to \ell_\infty$ is given by $(x_n)_n\mapsto (x_n\phi'(y_n)^{-1})_n$. This operator is well defined, because we may assume that $\phi'(y_n)\neq 0$, since $\lim_{n\to\infty} \abs{\phi'(y_n)}>1$.
    It is therefore straightforward to verify that the diagram commutes. 
\[\begin{tikzcd}
	{\mathcal HL_0(\mathbb D)} &&&& {\mathcal HL_0(\mathbb D)} \\
	{\mathcal H^\infty(\mathbb D)} && {\mathcal H^\infty(\mathbb D)} && {\mathcal H^\infty(\mathbb D)} \\
	\\
	{\ell_\infty} && {\ell_\infty} && {\ell_\infty}
	\arrow["{C_\phi^L}", from=1-1, to=1-5]
	\arrow["d"', from=1-1, to=2-1]
	\arrow["d", from=1-5, to=2-5]
	\arrow["{C_\phi^\infty}"', from=2-1, to=2-3]
	\arrow["{S_{(\phi(y_n))_n}}"', from=2-1, to=4-3]
	\arrow["{M_{\phi'}}"', from=2-3, to=2-5]
	\arrow["{S_{(y_n)_n}}", from=2-3, to=4-3]
	\arrow["{S_{(y_n)_n}}", from=2-5, to=4-5]
	\arrow["T", from=4-1, to=2-1]
	\arrow["{Id_{\ell_\infty}}", from=4-1, to=4-3]
	\arrow["{M_{\phi'(y_n)^{-1}}}", from=4-5, to=4-3]
\end{tikzcd}\]
In particular we see that $C_\phi^L$ fixes the copy $d^{-1}\circ T(\ell_\infty)$ of $\ell_\infty$ which we wanted to prove.
\end{proof}

\subsection{The vector case} 
We now arrive at the main result of this section, which establishes a series of implications concerning the compactness of composition operators. In the finite-dimensional setting, all these conditions are equivalent. 

\begin{theorem}\label{theo:compactnesscomp} 
Let $X, Y$ be complex Banach spaces and let $\phi\in \mathcal HL_0(B_Y, X)$ with $\phi(B_Y)\subseteq B_X$. Consider the following statements: 
\begin{enumerate}
\item[(1)] $C_\phi^L\colon \mathcal HL_0(B_X)\to \mathcal HL_0(B_Y)$ is compact.
\item[(2)] $C_\phi^L\colon \HL_0(B_X) \to \HL_0(B_Y)$ does not fix any copy of $\ell_\infty$ and $\phi(B_Y)$ is relatively compact.
\item[(3)] $\norm{\phi(y)}<\norm{y}$ for every $y\in \overline{B}_Y\setminus\{0\}$ and $\phi(B_Y)$ is relatively compact.
\item[(4)] $C_\phi^\infty\colon \mathcal H^\infty(B_X)\to \mathcal H^\infty(B_Y)$ is compact.
\end{enumerate}
Then $(1)\Rightarrow (2)\Rightarrow (3)$. If $\dim(Y)<\infty$, then $(3)\Rightarrow (4)$. If $\dim(X)<\infty$, 
then $(4)\Rightarrow (1)$. 
Thus, for $X$ and $Y$ finite-dimensional, all the statements are equivalent.
\end{theorem}
\begin{proof}

$(1) \Rightarrow$ (2)  It is clear that compactness implies $\ell_\infty$-singularity.
Moreover,  $\phi(B_Y)=\beta_X\circ \widehat{\phi}\circ \delta_Y(B_Y)$ where $\beta_X:=T_{Id_{B_X}}$ denotes the canonical quotient onto $X$ satisfying $\beta_X\circ \delta_X=Id_{B_X}$.
By Lemma~\ref{lemma:compositionoperatorproperties} and (1), the operator $\widehat{\phi}$ is compact.
Since $\delta_Y(B_Y)$ is bounded, it follows that $\phi(B_Y)$ is relatively compact.

   $(2) \Rightarrow (3)$ Suppose that there exists $y_0\in \overline{B}_Y\setminus\{0\}$ with $\|\phi(y_0)\|=\norm{y_0}$. Pick $x^*\in S_{X^*}$ such that $x^*(\phi(y_0))=\norm{y_0}$ and consider  the linear map $i \colon \disk \rightarrow B_Y$ defined by $i(z)=z\frac{y_0}{\norm{y_0}}$ for all $z\in \disk$. 
   Then, $x^*\circ \phi \circ i\in \mathcal{H}L_0(\disk)$ and $(x^*\circ \phi \circ i)(\disk)\subseteq \disk$. Furthermore, $(x^*\circ \phi \circ i)(\norm{y_0})=\norm{y_0}$ which implies $\|x^*\circ \phi \circ i\|_\infty =1$ (this is obvious if $\|y_0\|=1$ and if $y_0\in B_Y\setminus\{0\}$ we appeal to  Schwarz's lemma to derive that $x^*\circ \phi \circ i$  is a rotation). 
   Then $C_{x^*\circ \phi \circ i}^L$ fixes a copy of $\ell_\infty$ by Theorem~\ref{thm:compactscalarcase}. Since we have $C_{x^*\circ \phi \circ i}^L = C_i^L\circ C_\phi^L \circ C_{x^*}^L$ we see that $C_\phi^L$ fixes a copy of $\ell_\infty$ too.

If $\dim(Y)<\infty$, the condition $\norm{\phi(y)}<\norm{y}$ for every $y\in \overline{B}_Y\setminus\{0\}$ is equivalent to $\|\phi\|_\infty<1$, and hence \cite[Proposition 3]{AGL} gives us the implication $(3)\Rightarrow (4)$. 

Finally, assume that  $\dim(X)<\infty$ and $C_\phi^\infty$ is compact. Then $C_\phi^{X^*}$ is compact by Lemma \ref{lemma:CphiLvsInfty}. Now Corollary~\ref{cor:ideals} for the ideal of compact operators yields the compactness of $C_\phi^L$. Thus, (4)$\Rightarrow$(1), finishing the  proof. 
\end{proof}

Concerning the final part of the previous proof, note that $C_\phi^{X^*}$ can be compact only if $X$ is finite-dimensional, by Lemma \ref{lemma:comptensorop}. 

We  now present examples showing that some of the implications in Theorem \ref{theo:compactnesscomp} fail for general Banach  spaces $X, Y$. The first one shows that $(3)\not\Rightarrow (4)$. 

\begin{example} 
Let $\phi\colon B_{\ell_1}\to \mathbb C$ be the linear map given by $\phi(y)=\sum_{n=1}^\infty(1-1/n)y_n$. Then $\phi\in \mathcal HL_0(B_{\ell_1})$ and $\phi(B_{\ell_1})\subseteq \mathbb D$. Indeed, it is clear that $|\phi(y)|<\norm{y}$ for each $y\in\overline{B}_{\ell_1}\setminus\{0\}$. Also, $\phi(B_{\ell_1})$ is relatively compact as it is contained in $\mathbb D$.  However, $\norm{\phi}_\infty=1$, so $C_\phi^\infty$ is not compact \cite{AGL}. In fact, since $\phi$ is linear, we will get from Lemma~\ref{lemma:Tfixeslinfty} that $C_\phi^L$ fixes a copy of $\ell_\infty$.
\end{example}

Now we provide an example of a composition operator such that $C_\phi^L$ is weakly compact and $\phi(B_Y)$ is relatively compact, but $C_\phi^L$ is not compact. This shows that the equivalence between (1) and (2) in \cite[Proposition 3]{AGL} for $\mathcal H^\infty$ spaces does not hold in the case of $\mathcal HL_0$. Moreover, in this example  $C_\phi^\infty$ is  compact, so this exhibits that  (4) $\not\Rightarrow$ (1) in Theorem \ref{theo:compactnesscomp}.
We have been informed by Mathis Lemay and Colin Petitjean  that they have independently obtained a similar example.

\begin{example}\label{ex:ell2} Consider $\phi\colon B_{\ell_2}\to \ell_2$ given by $\phi(y)=(\frac{c}{n} y_n^n)_{n=1}^\infty$, where $0<c<1$ is such that $c\sum_{n=1}^\infty \frac{1}{n^2}<1$. Then $\phi$ is holomorphic since it is weakly holomorphic, and  
\begin{align*}
\norm{\phi(x)-\phi(y)}_2^2 =\sum_{n=1}^\infty \frac{c^2}{n^2}|x_n^n-y_n^n|^2 \leq \sum_{n=1}^\infty c^2 |x_n-y_n|^2 = c^2 \norm{x-y}_2^2, \qquad \forall x, y\in B_{\ell_2},
\end{align*}
where we have used that the function $z\mapsto z^n$ is $n$-Lipschitz on $\overline{\mathbb D}$. Hence, $L(\phi)\leq c$, and consequently $\phi\in \mathcal HL_0(B_{\ell_2}, \ell_2)$. 

Also, $\phi(B_{\ell_2})$ is relatively compact, since 
\[ \sum_{n=N}^\infty \left|\frac{c}{n}y_n^n\right|^2 \leq c^2\sum_{n=N}^\infty \frac{1}{n^2}\underset{N\to\infty}{\longrightarrow} 0\]
uniformly for $y\in B_{\ell_2}$. In addition, $\norm{\phi}_\infty \leq c<1$ and therefore $C_\phi^\infty$ is compact by \cite{AGL}. 

Next, observe that $\norm{d\phi}_\infty =L(\phi)=c<1$, so $d\phi(B_{\ell_2})(S_{\ell_2})$ is bounded in $\ell_2$ and hence relatively weakly compact.
Thus, Proposition \ref{prop:compRelWComp} (b) yields that $C_\phi^L$ is weakly compact.

Finally, we show that $C_\phi^L$ is not compact. To this end, take $f_n=e_n^*\in \mathcal HL_0(B_{\ell_2})$. Note that $\norm{f_n}=1$ but if $n\neq m$ then  
\begin{align*}    
L(C_\phi(f_n)-C_\phi(f_m)) & \geq \norm{d((f_n-f_m)\circ \phi)(e_n)}= \norm{(f_n-f_m)\circ d\phi(e_n)}\\
&=\sup_{u\in B_{\ell_2}} |(f_n-f_m)(d\phi(e_n)(u))| \\
&= \sup_{u\in B_{\ell_2}} |(f_n-f_m)(c u_n e_n)|
= c
\end{align*}
Therefore, $(C_\phi(f_n))_n$ does not have any convergent subsequence.
\end{example}

Note that if $Y$ is finite-dimensional, then $\phi(B_Y)$ is relatively compact, since $\phi$ is a Lipschitz function. Additionally, if $X$ is reflexive, we can give the following characterization.

\begin{proposition}
\label{prop:comp and weakly comp} Let $X, Y$ be complex Banach spaces, where $X$ is reflexive and $\dim(Y)<\infty$. Let $\phi\in \mathcal HL_0(B_Y, X)$ such that $\phi(B_Y)\subseteq B_X$.  
Then, the following assertions are equivalent.
\begin{enumerate}
    \item $C_\phi^L\colon \mathcal HL_0(B_X)\to\mathcal HL_0(B_Y)$ is weakly compact.
    \item $C_\phi^L\colon \HL_0(B_X) \to \HL_0(B_Y)$ does not fix any copy of $\ell_\infty$.
    \item $C_\phi^\infty\colon \mathcal H^\infty(B_X)\to\mathcal H^\infty(B_Y)$ is weakly compact.
    \item $C_\phi^\infty\colon \mathcal H^\infty(B_X)\to\mathcal H^\infty(B_Y)$ is compact.
    \item $\norm{\phi}_\infty<1$.
\end{enumerate}
\end{proposition}

\begin{proof}
Since $\phi(B_Y)$ is relatively compact, the equivalence between $(3)$, $(4)$ and $(5)$ is proved in \cite[Proposition 3]{AGL}. Also, $(1)\Rightarrow (2)$ is trivial  and (2)$\Rightarrow$(4) follows from Theorem~\ref{theo:compactnesscomp} (note that, since $Y$ is finite-dimensional, $\norm{\phi}_\infty<1$ whenever $\norm{\phi(y)}<1$ for all $y\in S_Y$). 
Finally, let us show (4)$\Rightarrow$(1). If $C_\phi^\infty$ is compact, then  Lemma \ref{lemma:CphiLvsInfty} and the reflexivity of $X$ yield that $C_\phi^{X^*}$ is weakly compact. Now, $C_\phi^L$ is weakly compact by  Corollary \ref{cor:ideals}.
\end{proof}

When both $X$ and $Y$ are finite-dimensional spaces, Theorem \ref{theo:compactnesscomp} and Proposition~\ref{prop:comp and weakly comp} yield the following.

\begin{corollary}\label{cor:compactfindimcase}
    Let $X,Y$ be complex finite-dimensional Banach spaces. Let $\phi \in \mathcal{H}L_0(B_Y,X)$ such that $\phi(B_Y)\subseteq B_X$ and consider $C_\phi$ the associated composition operator. Then, the following assertions are equivalent. 
    \begin{enumerate}
        \item $C_\phi^L \colon \mathcal HL_0(B_X)\rightarrow \mathcal HL_0(B_Y)$ is compact. 
         \item $C_\phi^L \colon \mathcal HL_0(B_X)\rightarrow \mathcal HL_0(B_Y)$ does not fix any copy of $\ell_\infty$. 
        \item $C_\phi^\infty \colon \mathcal H^\infty(B_X)\rightarrow \mathcal H^\infty(B_Y)$ is  weakly compact.
        \item $C_\phi^\infty \colon \mathcal H^\infty(B_X)\rightarrow \mathcal H^\infty(B_Y)$ is compact.
        \item $\norm{\phi}_\infty<1$.
    \end{enumerate}
\end{corollary}

During the preparation of this paper we  learned that Fred\'eric Bayart  independently proved the equivalence $(1)\Leftrightarrow(5)$  above.

Observe that the equivalent conditions of the previous corollary do not guarantee  $C_\phi$ to be compact (or weakly compact) as an operator from $\Lip_0(B_X)$ to $\Lip_0(B_Y)$, since this would require $\phi$ to be uniformly locally flat, and hence $\phi\equiv 0$ \cite{JVVV, ACPcomp}.

\subsection{The linear case.}

When the symbol $\phi$ of the composition operator $C_\phi^L$ is linear, we can improve some of the previous results with simpler proofs. 

\begin{lemma}\label{lemma:Tisometry} Let $X,Y$ be complex Banach spaces,  and let $T\in \mathcal L(Y,X)$ satisfy $T(B_Y)=B_X$. Then $C_T^L\colon \mathcal HL_0(B_X)\to\mathcal HL_0(B_Y)$ is an isometry. In particular, $C_T^L$ fixes a copy of $\ell_\infty$.  
\end{lemma}

\begin{proof}
For $f\in \mathcal HL_0(B_X)$ we have
\begin{align*}
    L\left(C_T^L(f)\right)&= \norm{d(f\circ T)}_\infty = \sup_{y\in B_Y} \norm{df(Ty)\circ T}\\
    &= \sup_{y,v\in B_Y} \norm{df(Ty)(Tv)} = \sup_{x,u\in B_X} \norm{df(u)(v)} = \norm{df}_\infty = L(f).
\end{align*}
\end{proof}

\begin{lemma}\label{lemma:Tfixeslinfty} Let $X, Y$ be complex Banach spaces and $T\in \mathcal L(Y,X)$ with $\norm{T}\leq 1$. If  $C_T^L$ does not fix a copy of $\ell_\infty$, then $\overline{T(B_Y)}\subseteq B_X$.
\end{lemma}

\begin{proof}
    Assume that $C_T^L$ does not fix a copy of $\ell_\infty$ and that there exists $x\in S_X\cap \overline{T(B_Y)}$. Then we may take $x^*\in S_{X^*}$ with $x^*(x)=1$ and it is clear that $\norm{x^*\circ T}=1$, so $x^*\circ T(B_Y)=\mathbb D$. But $C_{x^*\circ T}^L=C_T^L\circ C_{x^*}^L$ does not fix a copy of $\ell_\infty$, which contradicts Lemma \ref{lemma:Tisometry}.  
\end{proof}

\begin{proposition}\label{prop:compLineal}
    Let $X, Y$ be complex Banach spaces and $T\in \mathcal L(Y,X)$ with $\norm{T}\leq 1$. Then the following assertions are equivalent:
    \begin{itemize}
\item[(1)] $C_T^L\colon \mathcal HL_0(B_X)\to \mathcal HL_0(B_Y)$ is compact. 
\item[(2)] $C_T^\infty\colon \mathcal H^\infty(B_X)\to \mathcal H^\infty(B_Y)$ is compact. 
\item[(3)] $T$ is compact and $\norm{T}<1$.
    \end{itemize}
If moreover $X$ has the Schur property, they are also equivalent to:
\begin{itemize}
\item[(4)] $C_T^L$ is weakly compact.
\end{itemize} 
\end{proposition}

\begin{proof} 
The equivalence (2)$\Leftrightarrow$(3) follows from \cite{AGL}, while the implication  (3)$\Rightarrow$(1) is a consequence of Proposition \ref{prop:compRelComp}. Assume now that $C_T^L$ is compact. By Lemma \ref{lemma:Tfixeslinfty}, we have $\overline{T(B_Y)}\subseteq B_X$. We can apply Proposition \ref{prop:compRelComp} to get that $T$ is compact. Then $\overline{T(B_Y)}$ is a compact subset of $B_X$, so $\norm{T}<1$.

Finally, suppose that $C_T^L$ is weakly compact. Lemma~\ref{lemma:Tfixeslinfty} again yields $\overline{T(B_Y)}\subseteq B_X$. Then Proposition \ref{prop:compRelWComp} shows that $T(B_Y)$ is relatively weakly compact. If $X$ has the Schur property, then $T(B_Y)$ is indeed norm-compact proving that (3) holds. 
\end{proof}

The following example presents an operator $T\in \mathcal L(Y,X)$ which is weakly compact and satisfies $\norm{T}<1$, but for which $C_T^L$ fixes a copy of $\ell_\infty$. In particular, $C_T^L$ is not weakly compact. Observe that this also shows that the converse of Proposition~\ref{prop:compRelWComp} (b) fails.

\begin{example}\label{ex:ell2x/2} Let $T\colon \ell_2\to\ell_2$ be given by $T(x)=x/2$. 
Clearly, $T$ is weakly compact. 
For each $a=(a_n)_n\in \ell_\infty$, consider the function $f_a(x)= \sum_{n=1}^\infty a_n x^2_n$, $x\in B_{\ell_2}$. Note that for $x,u\in B_{\ell_2}$ we have $df_a(x)(u)=\sum_{n=1}^\infty 2a_n x_n u_n$. Thus, 
\[L(f_a)=\norm{df_a}_\infty =\sup_{x,u\in B_{\ell_2}} |df_a(x)(u)|  \leq 2\norm{a}_\infty\]
and $L(f_a)\geq \sup_n |df_a(e_n)(e_n)|= 2\sup_n|a_n|=2\norm{a}_\infty$. 

This shows that the subspace $Z=\{f_a: a\in \ell_\infty\}\subseteq \mathcal HL_0(B_{\ell_2})$ is isomorphic to $\ell_\infty$. Finally, it is clear that
\[ L\left(C_T^L(f_a)\right)= L(f_{a/4})= \frac{1}{2}\norm{a}_\infty =\frac{1}{4} L(f_a)\]  
so $C_T^L|_Z$ is an isomorphism onto its image. 
\end{example}

We exhibited in Example \ref{ex:ell2} a composition operator $C_\phi^L$ which is weakly compact but not compact. Now, we provide another example with linear symbol. This is an adaptation of  \cite[Example 3]{AGL}, where the same behavior is shown in the $\mathcal H^\infty$ case. 
Recall that a function $f\colon B_X\to Y$ is said to be \textit{weakly uniformly continuous} on $B_X$ if for every $\varepsilon>0$ there are $\delta>0$ and $x_1^*,\ldots, x_n^*\in X^*$ such that $\norm{f(x)-f(y)}<\varepsilon$ whenever $|x_i^*(x-y)|<\delta$ for all $i=1,\ldots,n$.  

\begin{example}
    
    Let $X=T^*$ be the original Tsirelson space. Recall that $X$ is reflexive and that every continuous polynomial $P\colon X\to X^*$ is weakly uniformly continuous on $B_X$ 
    (this follows from \cite[Theorem 2.5]{GonzaloJaramillo} combined with \cite[Proposition 2.12]{AHV}; see also \cite[Corollary 3.64]{HajekJohanis}). 
    Let $\phi\colon X\to X$ be given by $\phi(x)=x/2$. It is clear that $\phi\in\HL_0(B_X, X)$ and $\phi(B_X)\subseteq B_X$.
    Since $\phi(B_X)$ is not relatively compact,
     by Proposition~\ref{prop:compLineal} we have that $C_\phi^L\colon \mathcal HL_0(B_X)\to\mathcal HL_0(B_X)$ is not compact. 
    By Corollary \ref{cor:ideals}, to prove that $C_\phi^L$ is weakly compact it suffices to show that this is the case for $C_\phi^{X^*}$. 

    To this end, note first that for $f\in \mathcal H^\infty(B_X, X^*)$ and $x\in B_X$ the Taylor series expansion of $f$ yields
    \[ C_\phi^{X^*}(f)(x)= f(x/2) =\sum_{m=0}^\infty P_m(f)(0)(x/2) = \sum_{m=0}^\infty \frac{P_m(f)(0)(x)}{2^m}\]
    and this series converges uniformly in $B_X$ since $\sup_{x\in B_X}\norm{P_m(f)(0)(x)}\leq \norm{f}_\infty$ (see e.g. \cite[Corollary 7.4]{Mujica}). In consequence, $C_\phi^{X^*}(f)$ belongs to $\mathcal A_{wu}(B_X, X^*)$ (the subspace of weakly uniformly continuous elements of $ \mathcal H^\infty(B_X, X^*)$). 

    Any function in $\mathcal A_{wu}(B_X, X^*)$ has a continuous extension to $\overline B_X$ endowed with the weak topology, and recall that this is a compact set due to the
   reflexivity of $X$. Hence, $\mathcal A_{wu}(B_X, X^*)$ can be seen as a closed subspace of $C((\overline B_X, w), X^*)$ (continuous functions from $(\overline B_X, w)$ to $X^*$).

   Let $(f_n)\subseteq \overline B_{\mathcal H^\infty(B_X, X^*)}$. Our aim is to select a weakly convergent subsequence of $(C_\phi^{X^*}(f_n))_n$.
   Let $(x_m)_m$ be a dense set in $B_X$ (recall that $X$ is separable). Appealing  to the reflexivity of $X^*$ and applying a typical diagonal procedure we can obtain a subsequence $(f_{n_k})_k$ such that $(f_{n_k}(x_m))_k$ is $w$-convergent in $X^*$ for each~$m$. 
   Also, since $\mathcal H^\infty(B_X, X^*)=(\mathcal G^\infty(B_X)\pten X)^*$ is a dual space, the sequence $(f_{n_k})_k$ should have a subnet which is $w^*$-convergent to a function $g\in \overline B_{\mathcal H^\infty(B_X, X^*)}$. Clearly, this implies $f_{n_k}(x_m)\overset{w}{\to} g(x_m)$ for each $m$. Now, the sequence $(f_{n_k})_k$ is uniformly bounded, then it is equicontinuous \cite[Proposition 9.15]{Mujica}. Therefore, $f_{n_k}(x)\overset{w}{\to} g(x)$ for every $x\in B_X$.

   To finish the proof we show that $(C_\phi^{X^*}(f_{n_k}))_k$ converges weakly to $C_\phi^{X^*}(g)$ in $\mathcal A_{wu}(B_X, X^*)$. Since the weak topology of $\mathcal A_{wu}(B_X, X^*)$ is induced by the one of $C((\overline B_X, w), X^*)$, we equivalently have to prove the weak convergence in this latter space. 
   By the compactness of $(\overline B_X, w)$ and the reflexivity of $X$ we have that the set $\{\delta(x)\otimes y:\, x,y\in \overline B_X\}$ is a James boundary for $C((\overline B_X, w), X^*)$. Hence, by Rainwater-Simons theorem \cite[Theorem 3.134]{FHHMZ} it is enough to see that $C_\phi^{X^*}(f_{n_k})(x)(y)\overset{w}{\to} C_\phi^{X^*}(g)(x)(y)$, for all  $x,y\in \overline B_X$.
   This is clear from the previous arguments since $C_\phi^{X^*}(f_{n_k})(x)(y)=f_{n_k}(x/2)(y)$ and $C_\phi^{X^*}(g)(x)(y)= g(x/2)(y)$.

\end{example}   

\begin{remark}
    Recall that, by Lemma~\ref{lemma:compositionoperatorproperties} we have $C_\phi^L={\widehat{\phi}}^*$. 
    Hence, by the Schauder theorem, the Gantmacher theorem and Lemma~\ref{l:BessagaPelczynski}, the results in this section concerning compactness, weak compactness, and $\ell_\infty$-strict singularity of $C_\phi^L$ can equivalently be formulated in terms of compactness,  weak-compactness, and the property of fixing no complemented copies of $\ell_1$ for $\widehat{\phi}$, respectively.

   We do not know, however, the answer to the following question: if $\widehat{\phi}$ fixes a copy of~$\ell_1$, must it fix a complemented copy of $\ell_1$?
    Note that this implication  holds for linearizations of Lipschitz maps, see~\cite{Lemay}.
\end{remark}

\section{Iterates of composition operators on \texorpdfstring{$\mathcal HL_0$}{HL0}} \label{sect:iteration}

In this section, we assume that $X=Y$ is any complex Banach space. Given a composition operator $C_\phi : \mathcal HL_0(B_X)\rightarrow \mathcal HL_0(B_X)$ and  $n\in\N$, we denote by $C_\phi^{(n)}$ the $n$-th iterate of $C_\phi$, that is,
\begin{equation*}
    C_{\phi}^{(n)}= C_\phi \circ \cdots \circ C_\phi.
\end{equation*}
Observe that $C_\phi^{(n)}$ is again an operator on $\mathcal HL_0(B_X)$ and that, clearly, $C_\phi^{(n)}= C_{\phi^{(n)}}$, where $\phi^{(n)}=\phi\circ \cdots \circ \phi$. We are interested in determining whether $\|C_{\phi}^{(n)}\|= L(\phi^{(n)})$  converges to $0$, or equivalently, whether $C_\phi^{(n)}$ converges (in norm) to $C_0=0$. Note that this phenomenon has no analogue for composition operators on $\mathcal H^\infty(B_X)$, since in that setting all iterates of $C_\phi$ have norm one.

To this end, we start with a result that may be of independent interest.

Let $\HL_{00}(B_X,Y)$ be the subspace of those $f\in \HL_0(B_X,Y)$ such that $df(0)=0$, which we call  \emph{functions with no linear component}. Note that $\mathcal L(X,Y)$ is canonically seen (through restriction to $B_X$) as a subspace of $\HL_0(B_X,Y)$, because for a linear mapping, its norm and its Lipschitz constant coincide.

\begin{lemma}\label{l:HL0directSum}
For all complex Banach spaces $X,Y$ we have 
\[\HL_0(B_X,Y)=\mathcal L(X,Y)\oplus \HL_{00}(B_X,Y).\]
Moreover, the projection $\Pi\colon \HL_0(B_X,Y) \to \mathcal L(X,Y)$ is given by $\Pi f=df(0)$ and satisfies $\norm{\Pi f}\leq \norm{f}_\infty$.
Additionally, we have 
\begin{equation}\label{e:estimate_of_derivative}
\norm{df(x)}\leq \norm{f}_\infty+2L(f) \norm{x}    
\end{equation}
 for every $x\in B_X$ and $f\in \HL_0(B_X,Y)$. 
\end{lemma}

\begin{proof}
    Let $f \in \HL_0(B_X,Y)$. 
    We define $\Pi f:=d f(0)$. 
    It is evident that $\norm{df(0)} \leq L(f)=\sup_{x\in B_X}\norm{df(x)}$.
    Moreover, using the power series expansion of $f$ at $0$, and observing that $f(0)=0$, we have that $f(x)-df(0)(x)=\sum_{m=2}^\infty P_m(x)$ for all $x\in B_X$, where $P_m\in \mathcal P(^m X,X)$ for all $m\geq 2$. Notice that the radius of convergence of the Taylor series of $f$ at $0$ is at least $1$, since $f$ is bounded on $B_X$ (see \cite[Theorem~7.13]{MujicaLibro}). 
    It is now clear that $f-\Pi f \in \HL_{00}(B_X,Y)$.

    The inequality $\norm{\Pi f}\leq \norm{f}_\infty$ follows from the Cauchy integral formulas (see \cite[Corollary~7.4]{MujicaLibro}).
    To prove \eqref{e:estimate_of_derivative}, note that  $\norm{I-\Pi}\leq 2$, then
    \[
    \begin{aligned}
    \norm{df(x)}&\leq \norm{d(\Pi f)(x)}+\norm{d(I-\Pi)f(x)}=\norm{\Pi f}+\norm{d(I-\Pi)f(x)}\\
    &\leq \norm{f}_\infty + \norm{d(I-\Pi)f}_\infty\norm{x}=\norm{f}_\infty +L((I-\Pi)f)\norm{x}\\
    &\leq \norm{f}_\infty +2L(f)\norm{x}.
    \end{aligned}  
    \]
    On the second line above we have used $\norm{\Pi f}\leq\norm{f}_\infty$ and the Schwarz lemma applied on $d(I-\Pi)f$. 
    The use of the Schwarz lemma is justified as $(I-\Pi)f \in \HL_{00}(B_X,Y)$.
\end{proof}

\begin{remark}
    The decomposition in Lemma~\ref{l:HL0directSum} has a pre-adjoint counterpart whenever $Y$ is a dual space.
    Let us denote $\beta_X:=T_{Id_{B_X}}$ the canonical quotient onto $X$ such that $\beta_X\circ \delta_X=Id_{B_X}$, also known as the \emph{barycenter map} in the Lipschitz-free community.
    One can easily check that $\beta_X^*$ is the canonical embedding of $X^*$ into $\HL_0(B_X)$. 
    It is also well known and easy to see that $\beta_X\circ d\delta_X(0)=Id_{B_X}$. 
    Thus $P=d\delta_X(0)\circ \beta_X$ is a projection (see Proposition 2.6 in \cite{ADGLM}). 
    Now when $Y=\C$ we have $P^*f=df(0)=\Pi f$ for $f\in \mathcal HL_0(B_X)$, as $\langle f, d\delta(0)h\rangle = df(0)h$. 
    
    Similarly, when $Y=Z^*$ for some Banach space $Z$, 
    we will have $\Pi_{Y}=(P_\C \pten Id_Z)^*$. 
    We leave the details to the interested reader.
\end{remark}

Now we are ready to give the first condition assuring that $C_\phi^{(n)}$ converges to 0.

\begin{theorem}\label{thm:ergodic-phi-contracting}
    Let $X$ be any complex Banach space,  let  $\phi \in \mathcal{H}L_0(B_X,X)$ with $\phi(B_X)\subseteq B_X$ and consider $C_\phi \colon  \mathcal HL_0(B_X)\rightarrow \mathcal HL_0(B_X)$ the associated composition operator. If $\norm{\phi}_\infty<1$, then $\lim_{n\to\infty}L\left(\phi^{(n)}\right)=0$. In other words, $\|C_{\phi}^{(n)}\|\to 0$.
\end{theorem}

\begin{proof}
We aim to estimate $L(\phi^{(n)})=\norm{d\phi^{(n)}}_\infty$.
By the chain rule,
\[
d\phi^{(n)}(x)=d\phi(\phi^{(n-1)}(x))\circ d\phi(\phi^{(n-2)}(x))\circ \ldots \circ d\phi(\phi(x))\circ d\phi(x).
\]
Set $r:=\norm{\phi}_\infty$.
By the Schwarz lemma, $\norm{\phi^{(k)}}_\infty \leq r^k$.
Hence, using \eqref{e:estimate_of_derivative}, we obtain
\[
\begin{aligned}
\sup_{x\in B_X}\norm{d\phi^{(n)}(x)}&\leq (r+2L(\phi)r^{n-1})\cdot\ldots\cdot(r+2L(\phi)r)\cdot (r+2L(\phi))\\
&=\prod_{k=0}^{n-1} (r+2L(\phi)r^k)=r^{n}\prod_{k=0}^{n-1} (1+2L(\phi)r^{k-1}).
\end{aligned}
\]
 Now choose $1<s<\frac{1}{r}$. 
   Since $r^k\to 0$, there exists $k_0\in\N$ such that $1+2L(\phi)r^k \leq s$ for every $k\geq k_0$.  Therefore, for every $n\geq k_0+1$, 
    \begin{align*}
        0\leq \norm{C_{\phi}^{(n)}}=L\left(\phi^{(n)}\right)\leq r^{n}\prod_{k=0}^{n-1}(1+2L(\phi)r^{k-1}) \leq  K(sr)^{n-1-k_0} \xrightarrow[n \to \infty]{} 0,
    \end{align*}
    because $0<sr<1$ and $K= r^{k_0+1}\prod_{k=0}^{k_0}(1+2L(\phi) r^{k-1})$ is a constant.  
\end{proof}

Let us show that if $\norm{\phi}_\infty=1$, the iterations of the composition operator may not tend to 0, for example when $\phi$ has a fixed point different from 0.

\begin{remark}
    Let $\phi\in \mathcal HL_0(B_X,X)$ with $\phi(B_X)\subseteq B_X$ such that there is some $x\in \overline{B}_X\setminus\{0\}$ satisfying $\phi(x)=x$. Then,
    \begin{align*}
        \norm{C_\phi^{(n)}}=L\left(\phi^{(n)}\right) \geq \frac{\norm{\phi^{(n)}(x)-\phi^{(n)}(0)}}{\norm{x-0}}= \frac{\norm{\phi^{(n)}(x)}}{\norm{x}}=1, \quad \forall n\in \N,
    \end{align*}
    since the Lipschitz constant does not change when we extend $\phi$ to the whole $\overline{B}_X$, thanks to Lemma \ref{lemma:extension-border-HL0}.
    Hence, $\|C_{\phi}^{(n)}\|$ does not converge to $0$. For example, given $x_0\in S_X$ and $x^*\in S_{X^*}$ with $x^*(x_0)=1$, consider $\phi\in\HL_0(B_X,X)$ given by  $\phi(x)=x^*(x)^mx_0$ (for any $m\in\N)$. Then, $\phi$ satisfies this condition because $\phi(x_0)=x_0$. 
\end{remark}

We now show that, although the condition $\|\phi\|_\infty<1$ is not necessary for the norm of the operator $C_{\phi^{(n)}}$ to converge to zero, it must hold for some iterate $\phi^{(n)}$.

\begin{theorem}\label{theo:equivalent-convergent-0}
   Let $X$ be any complex Banach space,  let  $\phi \in \mathcal{H}L_0(B_X,X)$ such that $\phi(B_X)\subseteq B_X$ and consider $C_\phi :  \mathcal HL_0(B_X)\rightarrow \mathcal HL_0(B_X)$ the associated composition operator.  Then the following assertions are equivalent.
    \begin{enumerate}
        \item $\lim_{n\to\infty} \norm{C_\phi^{(n)}}=0$.
        \item $\lim_{n\to\infty} L\left(\phi^{(n)}\right)=0$.
        \item $\lim_{n\to\infty} \norm{\phi^{(n)}}_\infty =0$.
        \item There is some $n\in \N$ such that $\norm{\phi^{(n)}}_\infty<1$.
    \end{enumerate}
\end{theorem}

\begin{proof}
    The equivalence between $(1)$ and $(2)$ is immediate. Moreover, $(2)\Rightarrow (3)$ follows from the fact that $\norm{f}_\infty\leq L(f)$ for all $f\in \mathcal{H}L_0(B_X,Y)$ where $Y$ is any complex Banach space, while $(3)\Rightarrow (4)$ is straightforward.

 It remains to prove $(4)\Rightarrow(2)$. Fix $n_0\in \N$ such that $\norm{\phi^{(n_0)}}_\infty<1$. Applying Theorem \ref{thm:ergodic-phi-contracting} to the map  $\phi^{(n_0)} : B_X\rightarrow B_X$, we obtain $$\lim_{n\to\infty} L\left(\left(\phi^{(n_0)}\right)^{(n)}\right)=0.$$ Now set $ K= \max_{1\leq k< n_0} L\left( \phi^{(k)} \right)$. Observe that $K>0$, since
    $L(\phi)\geq \norm{\phi}_\infty >0$; otherwise $\phi=0$, in which case the conclusion is trivial.
   Given $\varepsilon>0$, choose $n_1$ such that $L\left(\left(\phi^{(n_0)}\right)^{(n)}\right)<\frac{\varepsilon}{K}$ for all $n\geq n_1$. Set $n_2=n_1n_0\in \N$. For every $n\geq n_2$, there exist $m\in \N$ and $k\in \N_0$ such that $n=mn_0+k$ with $m\geq n_1$ and $k<n_0$. Therefore,
   \begin{align*}
       L\left(\phi^{(n)} \right)&= L\left( \phi^{(mn_0)}\circ \phi^{(k)}\right) \leq L\left( \phi^{(mn_0)}\right) L\left(\phi^{(k)}\right) \\&= L\left(\left(\phi^{(n_0)}\right)^{(m)}\right) L\left(\phi^{(k)}\right)< \frac{\varepsilon}{K}K=\varepsilon.
   \end{align*}
   Hence, $\lim_{n\to\infty}  L\left(\phi^{(n)}\right)=0$.
\end{proof}

Note that the equivalent assertions of the previous theorem do not imply $\norm{\phi}_\infty<1$. We first give a counterexample for $\dim(X)\geq 2$.

\begin{example}
    If $\dim(X)\geq 2$, pick $x^*\in S_{X^*}$ and $x_0\in S_X$ such that $x^*(x_0)=0$. Consider $\phi : B_X\rightarrow B_X$ given by $\phi(x)=x^*(x)x_0$. It is clear that $\phi$ is linear so $\phi\in \mathcal HL_0(B_X)$ and  $\norm{\phi}_\infty =\norm{x^*}\norm{x_0}=1 $. Furthermore, 
    \begin{align*}
        \phi^{(2)}(x)= \phi(x^*(x)x_0)= x^*(x)x^*(x_0)x_0=0, \quad \forall x\in B_X.
    \end{align*}
    Hence, $\phi^{(n)}=0$ for all $n\geq 2$, and so $\|C_{\phi}^{(n)}\|=0$, for all $n\geq 2$.   
\end{example}

We also present a counterexample for the 1-dimensional case.

\begin{example}
    Let $\phi\colon\disk\to \C$ be given by $\phi(z)=\frac{iz+z^2}{2}$. Then, $\phi(\disk)\subseteq\disk$ and $\norm{\phi}_\infty=1$. Moreover, since the only point $z\in\overline{\disk}$ at which $\phi$ attains its norm is $z=i$, which is not in the image of $\phi$, it follows that $\|\phi^{(2)}\|_\infty<1$.
\end{example}

In Theorem \ref{theo:equivalent-convergent-0} we have seen two different characterizations of $\lim_{n\to\infty} \|C_\phi^{(n)}\|=0$, for any complex Banach space $X$. Furthermore, in the finite-dimensional case,  we can add new equivalent conditions to the preceding ones. In order to simplify the proof of the next proposition, we include there the assertions (3) and (4) from Theorem \ref{theo:equivalent-convergent-0}. Note that one of the equivalent conditions in Proposition \ref{prop:equivalent-convergent-0-findim} only involves the map $\phi$ and not its iterations.

\begin{definition}
    Let $X$ be a complex Banach space, and let $\phi\in \mathcal HL_0(B_X,X)$ such that $\phi(B_X)\subseteq B_X$. Given $x\in \overline{B}_X$ we define the \emph{orbit} of $x$ under the map $\phi$ as the set
    \begin{equation*}
        \orb_\phi(x)\coloneqq \{x\}\cup\{ \phi^{(n)}(x) : n\in \N\} \subseteq \overline{B}_X.
    \end{equation*}
\end{definition}

\begin{proposition}\label{prop:equivalent-convergent-0-findim}
    Let $X$ be a finite-dimensional complex Banach space and let $\phi\in \mathcal HL_0(B_X,X)$ with $\phi(B_X)\subseteq B_X$. Then the following assertions are equivalent.
     \begin{enumerate}
        \item There exists $n\in \N$ such that $\norm{\phi^{(n)}}_\infty<1$.
        \item $\lim_{n\to\infty} \norm{\phi^{(n)}}_\infty =0$.
        \item $\lim_{n\to\infty} \norm{\phi^{(n)}(x)}=0$, for all $x\in \overline{B}_X$.
        \item $\phi(C)\not\subseteq C$, for all $ \emptyset\neq C\subseteq S_X$.
        \item  $\orb_\phi(x)\not\subseteq S_X$, for all $x\in S_X$.
    \end{enumerate}
\end{proposition}

\begin{proof}
    The implication $(1)\Rightarrow(2)$ follows from Theorem \ref{theo:equivalent-convergent-0} and $(2)\Rightarrow(3)\Rightarrow(4)\Rightarrow(5)$ is straightforward, observing that $\phi(C)\subseteq C$ implies that $\phi^{(n)}(C)\subseteq C$ for all $n\in \N$, and $\phi(\orb_\phi(x))\subseteq \orb_\phi(x)$ for all $x\in \overline{B}_X$.

    It remains to show $(5)\Rightarrow (1)$. Assume that (5) holds. Then, for each $x\in \overline{B}_X$ there is  $n\in \N$ such that $\phi^{(n)}(x)\in B_X$ (notice that  $\orb_\phi(x)\subseteq B_X$ for  $x\in B_X$). Hence,
    \begin{equation*}
        \overline{B}_X\subseteq \bigcup_{n\in\N} \left\{ x\in \overline{B}_X : \norm{\phi^{(n)}(x)}<1\right\}.
    \end{equation*}
    Since $\overline{B}_X$ is compact and the sets $\left\{ x\in \overline{B}_X : \norm{\phi^{(n)}(x)}<1\right\}$ are open, there exist  $n_1,\ldots,n_k\in\N$ such that 
    \begin{equation*}
        \overline{B}_X\subseteq \bigcup_{i=1}^k \left\{ x\in \overline{B}_X : \norm{\phi^{(n_i)}(x)}<1\right\}= \left\{ x\in \overline{B}_X : \norm{\phi^{(n_0)}(x)}<1\right\},
    \end{equation*}
    where $n_0=\max\{n_1,\ldots, n_k\}$.  Thus, $\norm{\phi^{(n_0)}(x)}<1$ for all $x\in \overline{B}_X$ and so $\norm{\phi^{(n_0)}}_\infty<1$, due to the compactness of $\overline{B}_X$.
    
\end{proof}

Note that Proposition \ref{prop:equivalent-convergent-0-findim} does not hold in the infinite-dimensional case, as the following example shows.

\begin{example}\label{e:OrbitsLeaveSphere} Let $\phi\colon B_{c_0}\to B_{c_0}$ given by $\phi((x_n)_{n=1}^\infty)=(\frac{n}{n+1} x_n)_{n=1}^\infty$. Since $\phi$ is the restriction of a bounded linear operator it is clear that $\phi\in \mathcal HL_0(B_{c_0}, c_0)$. Moreover, $\phi(S_{c_0})\subseteq B_{c_0}$ and therefore $\phi(C)\not\subseteq C$ for all $\emptyset\neq C\subseteq S_{c_0}$. On the other hand,
\[L(\phi^{(n)})=\norm{\phi^{(n)}}_\infty \geq \sup_{k\in\mathbb N}\norm{\phi^{(n)}(e_k)}= \sup_{k\in \mathbb N}\left(\frac{k}{k+1}\right)^n =1.\]
Hence, $L(\phi^{(n)})=1$ for each $n\in \mathbb N$. Note that the same example can also be considered for $\phi\colon B_{\ell_2}\to B_{\ell_2}$. 
\end{example}

It will be interesting to find conditions ensuring that $C_\phi$ is \emph{power bounded}, that is, there is $C>0$ such that $L(\phi^{(n)})\leq C$ for each $n\in \mathbb N$. We already know that this is the case if $\norm{\phi^{(n)}}_\infty<1$ for certain $n$. However, if $\phi\in \mathcal HL_0(B_X,X)$ with  $\norm{\phi^{(n)}}_\infty =1$ for all $n$, we show in the following example that $C_\phi$ could be power bounded or not. 

\begin{example} A linear $\phi$ yields $C_\phi$ power bounded / A particular polynomial $\phi$ produces $C_\phi$ not power bounded.
    \begin{enumerate}
\item If $\phi\in\mathcal HL_0(B_X,X)$ is linear with $\norm{\phi}_\infty =1$, then $C_\phi$ is power bounded. Indeed, $L(\phi^{(n)})=\|\phi^{(n)}\|_\infty\le \|\phi\|_\infty^n = 1$.

\item Let $\phi=p|_\mathbb D$ where $p$ is a polynomial in $\mathbb C$ of degree at least 2 such that $p(0)=0$, $p(\mathbb D)\subseteq \mathbb D$, and there is $x_0\in \mathbb T$ with $p(x_0)=x_0$ (e.g. $p(z)=z^2$). Then $C_\phi$ is not power bounded. Indeed,  we know from Corollary \ref{cor:denjoy-wolff-modificado}  that $|p'(x_0)|>1$. Then, given $n\in \mathbb N$, 
\begin{align*}
    \left|(\phi^{(n)})'(x_0)\right| &= \prod_{j=0}^{n-1} \left|\phi'(\phi^{(j)}(x_0))\right| = \prod_{j=0}^{n-1} |\phi'(x_0)| =  |\phi'(x_0)|^n \underset{n\to\infty}{\longrightarrow}  +\infty
\end{align*}
meaning that $L(\phi^{(n)})\underset{n\to\infty}{\longrightarrow}   +\infty$. 
Clearly, an analogous statement holds if $z_0$ is a fixed point of $\phi^{(k)}$ for some $k\in \mathbb N$.
\end{enumerate}
\end{example}

\section{Composition operators in \texorpdfstring{$\mathcal HL(B_X)$}{HL(BX)}} \label{sect:HL}

In this section we deal with composition operators between spaces of holomorphic Lipschitz functions on $B_X$ (which do not necessarily fix the point 0). 

We denote by
\[
\mathcal HL(B_X,Y)  =\{f\in \Lip(B_X, Y): f \text{ is holomorphic on }   B_X\};
\] as usual we put $\mathcal HL(B_X) $ in the scalar-valued case $Y=\mathbb C$. This is a Banach space endowed with the norm
\[
\|f\|_L= L(f) + \|f(0)\|.
\]
Note that $\|f\|_\infty=\sup_{x\in B_X}\|f(x)\|\le \|f\|_L$. Indeed, if $x\in B_X$,
\[
\|f(x)\|\le \|f(x)-f(0)\|+ \|f(0)\|\le L(f) \|x\| +\|f(0)\| \le \|f\|_L.
\]

We denote by $\mathcal A_u(B_X,Y)$ the set of holomorphic  uniformly continuous mappings from $B_X$ to $Y$, which is a Banach space with the norm $\|\cdot\|_\infty$. Since Lipschitz mappings are uniformly continuous, the previous inequality shows that the canonical inclusion $\mathcal HL(B_X,Y)\hookrightarrow \mathcal A_u(B_X,Y)$ has norm~1.

The space  $\mathcal HL(B_X)$ is often called the \emph{analytic Lipschitz algebra}, denoted by $\Lip_A(B_X, 1)$ (see, e.g., \cite{BM05, BM25}), and is usually equipped with the norm $\norm{\cdot}_\infty+L(\cdot)$. Moreover, in \cite{ADGLM}, the space $\mathcal HL(B_X,Y)$ is considered with the norm $\max\{L(f), \|f(0)\|\}$. These three norms are equivalent. We prefer to work with  $\norm{\cdot}_L$ since, on the one hand, it behaves well with multiplicative operators, and, on the other hand, $\mathcal HL_0(B_X, Y)$ embeds isometrically into $(\mathcal HL(B_X,Y), \norm{\cdot}_L)$.

As commented in \cite{ADGLM}, the space $\mathcal HL(B_X)$ admits a predual, denoted by $\mathcal G(B_X)$, for which the usual linearization properties hold, just as in the case of $\mathcal HL_0(B_X)$ and $\mathcal G_0(B_X)$. Since the norm on $\mathcal HL(B_X)$ considered here differs from the one used in \cite{ADGLM}, we summarize the relevant facts in the following proposition.

First, for every $x\in B_X$, we define $\delta(x)\in \mathcal HL(B_X)
^*$ to be the evaluation functional at $x$.
Since $X^*$ embeds canonically into $\mathcal HL(B_X)$, we have $\norm{\delta(x)-\delta(y)}=\norm{x-y}$, and therefore the map $\delta\colon B_X \to \mathcal HL(B_X)^*$ is an isometry. However, in this case, $\norm{\delta(x)}=1$ for all $x\in B_X$, and so $\norm{\delta}_L=2$. 
As before, $\delta$ is holomorphic by \cite[Corollary~15.47]{DefantGarciaMaestreSevillaPeris}.
We define
\[
\mathcal G(B_X)=\overline{\lspan}\, \delta(B_X)\subseteq \mathcal HL(B_X)^* 
\] 
and observe that most of the properties of $\mathcal G_0(B_X)$ admit analogous counterparts for $\mathcal G(B_X)$.

\begin{proposition}\label{prop:Gdef} Let $X, Y$ be complex Banach spaces. 
\begin{itemize}
   \item[(a)] For every $f\in \mathcal HL(B_X,Y)$ there exists a unique $T_f \in \mathcal L(\mathcal G(B_X),Y)$ such that $f=T_f\circ \delta$, and $$\frac12 \|f\|_L\le \max\{L(f),\norm{f(0)}\}\le \|T_f\| \le  \|f\|_L.$$
    
    \item[(b)] The map $f\mapsto T_f$ is surjective and linear.  In particular, 
    \[\mathcal HL(B_X,Y)\cong\mathcal L(\mathcal G(B_X),Y) \text{ and } \mathcal HL(B_X)\cong\mathcal G(B_X)^*.\]

    \item[(c)] The topologies $w^*$, $\tau_0$ and $\tau_p$ coincide on $r\overline{B}_{\mathcal HL(B_X)}$, for any $r>0$.

    \item[(d)] For any $r>0$,
    \begin{align*}    
    \mathcal G(B_X)&=\{\varphi\in \mathcal HL(B_X)^*: \varphi|_{r\overline{B}_{\mathcal HL(B_X)}} \text{ is } \tau_0\text{-continuous}\}\\
    &=\{\varphi\in \mathcal HL(B_X)^*: \varphi|_{r\overline{B}_{\mathcal HL(B_X)}} \text{ is } \tau_p\text{-continuous}\}.\end{align*}
\end{itemize}    
\end{proposition}

\begin{proof}
The proof follows the same steps as in Proposition \ref{prop:G0def}, except for the norm estimates, which we show here. First, note that the norm of $T_f$ is determined by its values on the dense subspace $\lspan\delta(B_X)$. 
Since $\norm{y^* \circ f}_L\leq \norm{y^*}\norm{f}_L$ for $y^* \in Y^*$, an argument analogous to the one used in the proof of Proposition~\ref{prop:G0def} shows that $\|T_f(\sum \lambda_i\delta(x_i))\|\le \|f\|_L \| \sum \lambda_i\delta(x_i)\|$. Thus, $\|T_f\|\le \|f\|_L$. On the other hand,  it is clear that 
 $\|f\|_L= L(f) + \|f(0)\| \le 2 \max\{L(f),\norm{f(0)}\}$. Also, $L(f)=L(T_f\circ \delta)\le \|T_f\| L(\delta)=\|T_f\|$ and $\|f(0)\|=\|T_f( \delta(0))\|\le \|T_f\| \|\delta(0)\|=\|T_f\|$.
\end{proof}

Recall, as we mentioned in Section \ref{sect:composition operators HL0}, that every holomorphic Lipschitz map $f\colon B_X\to Y$ extends uniquely to a map $\widetilde{f}\colon \overline{B}_X\to Y$ with the same Lipschitz constant.
Except in the following result, which is the $\mathcal HL$-version of Lemma \ref{lemma:extension-border-HL0}, we will identify $f$ and $\widetilde{f}$ and simply write $f$.

\begin{lemma}\label{lemma:extension-border-HL}
    Let $X,Y$ be complex Banach spaces. Given $f\in \mathcal HL(B_X,Y)$ we have:
    \begin{itemize}
        \item[(a)] There is a Lipschitz map $\widetilde{f}\colon \overline{B}_X\rightarrow Y$ such that $\widetilde{f}|_{B_X}=f$ and $L\big(\widetilde{f}\big)= L(f)$.
        \item[(b)] There is an isometry
        $\widetilde{\delta} \colon  \overline{B}_X\rightarrow \mathcal G(B_X)$ such that $\widetilde{\delta}|_{B_X}=\delta$ and $\|\widetilde{\delta}(x)\|=1$ for all $x\in \overline{B}_X$.
        \item[(c)] There is a unique linear operator $T_f\in \mathcal L(\mathcal G(B_X), Y)$ such that the following diagram commutes.
\begin{equation*}
\xymatrix{
 \overline{B}_X \ar[r]^{\widetilde{f}}  \ar[d]_{\widetilde{\delta}}    &  Y  \\
   \mathcal{G}(B_X)  \ar[ru]_{T_f}  &
}
\end{equation*}
    \item[(d)] If $f(B_X)\subseteq  B_Y$, there is a unique linear operator $\widehat{f}\coloneq T_{\delta_Y \circ f} \in \mathcal L(\mathcal G(B_X), \mathcal{G}(B_Y))$ such that the following diagram commutes
    \begin{equation*}
        \xymatrix{
 \overline{B}_X \ar[r]^{\widetilde{f}}  \ar[d]_{\widetilde{\delta}_X }    &  \overline{B}_Y \ar[d]^{\widetilde{\delta}_Y }  \\
   \mathcal{G}(B_X)  \ar[r]_{ \widehat{f}}   &  \mathcal{G}(B_Y)
}
    \end{equation*} 
    \end{itemize}
\end{lemma}

Using the same arguments as in Proposition \ref{prop:unitballconv}, we refine the $\mathcal HL$ version of Proposition 2.6 of \cite{ADGLM}, referred in that article at the end of Section 3.

\begin{proposition}\label{prop:unitballconvHL} 
Let $X$ be a complex Banach space. Then
\begin{align*}
\overline{B}_{\mathcal G(B_X)}&=\overline{\Gamma}(\{d\delta(x_1)(x_2)+\delta(0): x_1\in B_X, x_2\in S_X\})\\
&=\overline{\Gamma}\left(\left\{\frac{\delta(x_1)-\delta(x_2)}{\norm{x_1-x_2}}+\delta(0) : x_1\not= x_2\in B_X\right\}\right)
\end{align*}
\end{proposition}

\begin{proof} Note that the set  $\{d\delta(x_1)(x_2)+\delta(0): x_1\in B_X, x_2\in S_X\}$ is 1-norming for $\mathcal HL(B_X)$ since 
\begin{align*}
    \norm{f}_L&=\sup \{ |\langle f, d\delta(x_1)(x_2)\rangle| : x_1\in B_X,\ x_2\in S_X\} + |\langle f, \delta(0)\rangle|\\&=\sup \{ |\langle f, d\delta(x_1)(x_2)\rangle| + |\langle f, \delta(0)\rangle| : x_1\in B_X,\ x_2\in S_X\}\\&= \sup \{ |\langle f, d\delta(x_1)(x_2)+\delta(0)\rangle|  : x_1\in B_X,\ x_2\in S_X\},
\end{align*}
for all $f\in \mathcal HL(B_X)$. This yields the first equality. For the second one, observe that $$\{d\delta(x_1)(x_2)+\delta(0): x_1\in B_X, x_2\in S_X\}\subseteq \overline{\left\{\frac{\delta(x_1)-\delta(x_2)}{\norm{x_1-x_2}}+\delta(0) : x_1\not= x_2\in B_X\right\}},$$
by the same argument as in Proposition \ref{prop:unitballconv}.
\end{proof}

Next, we provide the definition of a composition operator on $\mathcal HL(B_X)$, which should be compared with Definition~\ref{def:compositionoperator}.

\begin{definition}
    Let $X,Y$ be complex Banach spaces. A mapping $A \colon \mathcal HL(B_X)\rightarrow \mathcal{H}L(B_Y)$ is called a \textit{composition operator} if there exists a map $\phi \in \mathcal HL(B_Y, X)$ satisfying $\phi(B_Y)\subseteq \overline B_X$ such that $A (f) = f\circ \phi$, for every $f\in \mathcal HL(B_X)$. As usual, we denote this operator by $C_\phi$.
\end{definition}

Note that, for $\mathcal HL$ spaces, we consider a more general notion of composition operator, since we allow the symbol $\phi$ to take values in the \textbf{closed} unit ball of $X$. Recall that, by the Maximum Modulus Theorem, if $\phi \in \mathcal HL(B_Y, X)$ with $\phi(B_Y)\subseteq \overline B_X$ then exactly one of the following alternatives holds: $\phi(B_Y)\subseteq  B_X$ or $\phi(B_Y)\subseteq S_X$. Observe that in the latter case the composition operator is still well defined. Indeed, given $f\in \mathcal HL(B_X)$, it is clear that $f\circ\phi$ is Lipschitz. Moreover, $f$ is uniformly continuous and therefore can be uniformly approximated (with respect to the supremum norm) by a sequence of polynomials $(P_n)$.  Consequently, $P_n\circ\phi \to f\circ\phi$  in the Banach space $\mathcal A_u(B_X,Y)$, and since each $P_n\circ\phi$ is holomorphic, it follows that $f\circ\phi$ is holomorphic as well.

\begin{lemma}\label{lemma:compositionoperatorpropertiesHL}
     Let $X,Y$ be complex Banach spaces and let $\phi \in \mathcal HL(B_Y, X)$ with $\phi(B_Y)\subseteq \overline B_X$. Then, the composition operator $C_\phi \in \mathcal L( \mathcal HL(B_X), \mathcal{H}L(B_Y))$ satisfies $C_\phi= \widehat{\phi}^*$. Furthermore,
     \begin{itemize}
         \item[(a)] If $\norm{\phi}_L\leq 1$, then $\norm{C_\phi}=1$. 
         \item[(b)] If $\norm{\phi}_L>1$, then $\max\{L(\phi),1\}\leq \norm{C_\phi}\leq \norm{\phi}_L$.
     \end{itemize} 
\end{lemma}

\begin{proof}
    Observe that $C_\phi=(\widehat{\phi})^*=(T_{\delta_X\circ \phi})^*$ by Lemma \ref{lemma:extension-border-HL}. Hence,  $$\norm{C_\phi}\geq \max\{L(\delta_X\circ \phi),\norm{\delta_X(\phi(0))}\}=\max\{ L(\phi),1\},$$ since $\delta_X$ is an isometry and $\norm{\delta_X(x)}=1$ for all $x\in \overline{B}_X$. Now, for all $f\in \mathcal HL(B_X)$,  \begin{align*}
        \norm{f\circ \phi}_L &= L(f\circ \phi)+\norm{f(\phi(0))}\leq L(f)L(\phi)+\norm{f(\phi(0))-f(0)}+\norm{f(0)}\\&\leq L(f)L(\phi)+L(f)\norm{\phi(0)}+\norm{f(0)}= L(f)\norm{\phi}_L+\norm{f(0)}.
    \end{align*}
    Therefore, if $\norm{\phi}_L\leq 1$, then $\norm{f\circ \phi}\leq \norm{f}_L$, for all $f\in \mathcal HL(B_X)$ and consequently $\norm{C_\phi}=1$. Otherwise, $\norm{f\circ \phi}\leq \norm{\phi}_L\norm{f}_L$, for all $f\in \mathcal HL(B_X)$ and thus $\max\{L(\phi),1\}\leq \norm{C_\phi}\leq \norm{\phi}_L$.
\end{proof}

\subsection{Multiplicative and composition operators}
As in Section \ref{sect:composition operators HL0}, the space $\mathcal HL(B_X)$ is also a Gelfand algebra: if $f,g \in \mathcal HL(B_X)$, then $f \cdot g \in \mathcal HL(B_X)$ and $\|f \cdot g\|_L \le 2 \|f\|_L \|g\|_L .$ In addition, $\mathcal HL(B_X)$ is unital, since it contains the constant function $1$, unlike $\mathcal{H}L_0(B_X)$.
As before, composition operators preserve products: $
C_\phi(f \cdot g) = C_\phi(f) \cdot C_\phi(g).$
Consequently,  it is natural to ask under which conditions a continuous linear multiplicative mapping from $\mathcal HL(B_X)$ into $\mathcal HL(B_Y)$ must be a composition operator. 

The following lemma will allow us to transfer to $\mathcal HL(B_X)$ some results previously established for $\mathcal HL_0(B_X)$.

\begin{lemma}\label{lemma: A tilde}
    Let $X,Y$ be complex Banach spaces. Let $A:\mathcal HL_0(B_X)\to \mathcal HL_0(B_Y)$ be a continuous linear mapping and define $\widetilde{A}: \mathcal HL(B_X)\to \mathcal HL(B_Y)$ by
    \[
    \widetilde{A}(f)=A(f-f(0)) + f(0).
    \] Then,
    \begin{itemize}
        \item[(a)] $\widetilde{A}$ is a continuous linear mapping satisfying $\|\widetilde{A}\|=1$ when $\|A\|\le 1$ and $\|\widetilde{A}\|= \|A\|$ when $\|A\|> 1$.
        \item[(b)] If $A$ is multiplicative, then so is $\widetilde{A}$.
    \end{itemize}
\end{lemma}

\begin{proof}
 (a) It is clear that $\widetilde{A}$ is well defined and linear. Note that, for each $f\in\mathcal HL(B_X)$ we have
 \[
 \|\widetilde{A}f\|_L = L(A(f-f(0))) + |f(0)|\le \norm{A} L(f-f(0)) + |f(0)|\le \norm{A} L(f) + |f(0)|.
 \]
 Then, if $\norm{A}\le 1$ we derive $\|\widetilde{A}f\|_L\le  \norm{f}_L$. This, along with the fact that $\widetilde A(1)=1$ yields $\|\widetilde{A}\|=1$. Instead, if  $\norm{A}> 1$ we obtain  $\|\widetilde{A}f\|_L\le \norm{A} \norm{f}_L$ and so $\|\widetilde{A}\|\le \|A\|$. But $\widetilde{A}(f)=A(f)$, for all $f\in \mathcal HL_0(B_X)$ and $\|f\|_L=L(f)$ for these functions. This implies that $\|A\|\le \|\widetilde{A}\|$ and hence, $\|\widetilde{A}\|= \|A\|$.

 (b) Suppose that $A$ is multiplicative. Then, for all $f,g\in\mathcal HL(B_X)$,
 \begin{align*}
     \widetilde{A}(f) \widetilde{A}(g) &= A(f-f(0)) A(g-g(0))+A(f-f(0))g(0)+A(g-g(0)) f(0) + f(0) g(0)\\
     &= A\Big((f-f(0)) \cdot (g-g(0))\Big)+A(f-f(0))g(0)+A(g-g(0)) f(0) + f(0) g(0)\\
     &= A\Big(fg-f(0) g(0)\Big) + f(0) g(0)= \widetilde{A}(fg).
 \end{align*}

\end{proof}

Note that since $\mathcal HL(B_X)$ has a unit (the function $f\equiv 1$), a multiplicative element $\psi\in\mathcal HL(B_X)^*$ different from $\psi\equiv 0$, should satisfy $\psi(1)=1$. Taking this into account, with analogous arguments as in the previous lemma we get: 
 
\begin{lemma}\label{lemma: phi tilde} 
     Let $X$ be a complex Banach space. For $\varphi\in \mathcal HL_0(B_X)^*$ we define $\widetilde\varphi\in \mathcal HL(B_X)^*$ by
    \[
    \widetilde{\varphi}(f)=\varphi(f-f(0)) + f(0).
    \]
    If $\varphi$ is multiplicative,  then so is $\widetilde\varphi$. Moreover, if $\psi\in \mathcal HL(B_X)^*$ is multiplicative and $\psi\not\equiv 0$, then $\psi=\widetilde\varphi$ for $\varphi=\psi|_{\mathcal HL_0(B_X)}$.
    \end{lemma}

As we see in Corollary \ref{cor:multiplicativebounded}, multiplicative forms in $\mathcal HL_0(B_X)$ have norm smaller than or equal to 1. Instead, in $\mathcal HL(B_X)$, they have unit norm (except for the null functional).

\begin{lemma}\label{lemma:multiplicativeboundedHL}
  Let $X$ be a complex Banach space and let $\psi\in \mathcal{H}L(B_X)^*$. If $\psi$ is multiplicative and $\psi\not\equiv 0$, then $\norm{\psi}= 1$.  
\end{lemma}

\begin{proof}
    By Lemma~\ref{lemma: phi tilde}, if $\psi\in \mathcal{H}L(B_X)^*$ is multiplicative, then $\varphi=\psi|_{\mathcal HL_0(B_X)}$ belongs to $\mathcal HL_0(B_X)^*$ and it is also multiplicative. Since $\|\varphi\|\le 1$, due to Corollary~\ref{cor:multiplicativebounded}, and $\psi=\widetilde{\varphi}$ an appeal to Lemma \ref{lemma: A tilde}  yields  $\norm{\psi}= 1$.  
\end{proof}

 Now we are ready to give versions of Proposition~\ref{prop:multiplicativeisdelta} and Theorem~\ref{thm:compositionoperatorequivalence} for $\mathcal HL(B_X)$.

 \begin{proposition}\label{prop:multiplicativeisdeltaHL}
     Let $X$ be a complex Banach space with the BAP. Let $0\not=\psi\in \mathcal G(B_X)$ be multiplicative. Then, there exists $x_0\in \overline{B}_X$ such that $\psi=\delta(x_0)$.
 \end{proposition}

 \begin{proof}
     Let $\varphi=\psi|_{\mathcal HL_0(B_X)}$. Then,  $\varphi\in\mathcal G_0(B_X)$ and $\varphi$ is multiplicative. By Proposition~\ref{prop:multiplicativeisdelta}, there exists $x_0\in\overline{B}_X$ such that $\varphi=\delta(x_0)$. It then follows from Lemma~\ref{lemma: phi tilde}  that $\psi=\widetilde{\varphi}= \widetilde{\delta(x_0)}=\delta(x_0)$.
 \end{proof}

The proof of the next theorem follows the same steps as the proof of Theorem \ref{thm:compositionoperatorequivalence}, using Proposition \ref{prop:Gdef} instead of Proposition \ref{prop:G0def}, Lemma~\ref{lemma:compositionoperatorpropertiesHL} instead of Lemma~\ref{lemma:compositionoperatorproperties}, Proposition \ref{prop:automaticcontinuity} instead of Proposition~\ref{prop:multiplicativebounded},  and Proposition~\ref{prop:multiplicativeisdeltaHL} instead of Proposition~\ref{prop:multiplicativeisdelta}.  Note also that the condition $A\not= 0$ means $A(1)=1$, which implies $\delta_Y(y)\circ A\not= 0$, for each $y\in B_Y$ (this condition is needed in order to apply Proposition~\ref{prop:multiplicativeisdeltaHL}).

\begin{theorem}\label{thm:compositionoperatorequivalenceHL}
    Let $X, Y$ be complex Banach spaces.
    Let $A \colon \mathcal HL(B_X)\rightarrow \mathcal HL(B_Y)$ be a multiplicative linear mapping, $A\not= 0$, and consider the following statements.
    \begin{enumerate}
        \item $A$ is a composition operator.
        \item There is a linear continuous map $T\colon \mathcal G(B_Y)\rightarrow \mathcal G(B_X)$ such that $A=T^*$.
        \item $A$ is $\tau_0$-$\tau_0$-continuous. 
        \item $A$ is $\tau_p$-$\tau_p$-continuous.
    \end{enumerate}
    Then 
    \begin{itemize}
        \item $(1)$ implies $(2)$, $(3)$ and $(4)$,
        \item $((3)$ or $(4))$ implies $(2)$.
    \end{itemize}
    When $X$ has the BAP, the statements $(1)$-$(4)$ are equivalent.
\end{theorem}

\subsection{Spectrum of \texorpdfstring{$\mathcal HL(B_X)$}{HL(BX)}}\label{subsect:spectrum}
For a commutative unital Banach algebra  $\mathcal A$, a classical object of study is its \emph{spectrum} (or \emph{maximal ideal space}) which is defined as the set
\[
\mathcal M(\mathcal A)=\{\varphi\in\mathcal A^*:\, \varphi \text{ is multiplicative}\}\setminus\{0\}.
\]
We extend this notion to a unital commutative Gelfand algebra  as $\mathcal HL(B_X)$. Accordingly, the spectrum of $\mathcal HL(B_X)$ is the set $\mathcal M(\mathcal HL(B_X))$ consisting of the nonzero multiplicative elements of $\mathcal HL(B_X)^*$. We show that there is a close relationship between $\mathcal M(\mathcal HL(B_X))$ and the spectrum of the Banach algebra $\mathcal A_u(B_X)$. Given that $\mathcal M(\mathcal A_u(B_X))$ has been extensively studied (see, e.g. \cite{ACG, ACLM, AFGM, DLM}), this connection allows us to describe $\mathcal M(\mathcal HL(B_X))$ as well as the multiplicative elements of $\mathcal HL_0(B_X)^*$.
\begin{remark} \label{rmk:psi-norma-infinito}
    From Lemma \ref{lemma: phi tilde} we know that each $\psi\in \mathcal M(\mathcal HL(B_X))$ is $\psi=\widetilde{\varphi}$, where $ \varphi=\psi|_{\mathcal HL_0(B_X)}$ is multiplicative. Now, Remark \ref{rmk:acotada-norma-infinito} tells us that $|\varphi(f)|\le \|f\|_\infty$, for all $f\in \mathcal HL_0(B_X)$. This easily implies that $|\psi(f)|\le 3\|f\|_\infty$, for all $f\in \mathcal HL(B_X)$.
\end{remark}
\begin{proposition}\label{prop:relacion-espectros}
    Let $X$ be a complex Banach space. Then, the function
   \begin{align*}
	\Theta\colon \mathcal M(\mathcal A_u(B_X)) & \to \mathcal M(\mathcal HL(B_X))\\
	\xi\quad  & \mapsto \ \xi|_{\mathcal HL(B_X)}.
	\end{align*}
    is bijective.
\end{proposition}
\begin{proof}
    The norm one inclusion $\mathcal HL(B_X)\hookrightarrow\mathcal A_u(B_X)$ has dense image (because polynomials are dense in $\mathcal A_u(B_X)$). Thus, $\Theta$ is well defined and injective. To see that it is surjective, let $\psi\in  \mathcal M(\mathcal HL(B_X))$. By Remark \ref{rmk:psi-norma-infinito}, $\psi$ is continuous in $\mathcal HL(B_X)$ with respect to the $\|\cdot\|_\infty$-norm. We can extend $\psi$, by density, to $\overline{\psi}\in \mathcal A_u(B_X)^*$. Using that each $f\in\mathcal A_u(B_X)$ can be approximated (in the $\|\cdot\|_\infty$-norm) by a bounded net of polynomials, it is easy to prove that $\overline{\psi}$ is multiplicative. In consequence, $\psi=\Theta(\overline\psi)$, where $\overline{\psi}\in \mathcal M(\mathcal A_u(B_X))$, showing that $\Theta$ is surjective.
\end{proof}
 As a direct consequence of Proposition \ref{prop:relacion-espectros} and Lemma \ref{lemma: phi tilde} we have
\begin{corollary} \label{cor:biyeccion-espectro}
     Let $X$ be a complex Banach space. Then, the function
   \begin{align*}
	\Lambda\colon \mathcal M(\mathcal A_u(B_X)) & \to \{\varphi\in\mathcal HL_0(B_X)^* \text{ multiplicative}\}\\
	\xi\quad  & \mapsto \ \xi|_{\mathcal HL_0(B_X)}.
	\end{align*}
    is bijective.
\end{corollary}

In summary, through restriction we can identify the three sets:
\[
\{\varphi\in\mathcal HL_0(B_X)^* \text{ multiplicative}\} \simeq \mathcal M(\mathcal HL(B_X))\simeq \mathcal M(\mathcal A_u(B_X)).
\]

\begin{example}
    Based on on the known characterization of $\mathcal M(\mathcal A_u(B_X))$ we can describe the multiplicative elements of $\mathcal HL_0(B_X)^*$. By means of the canonical extension of each $f\in \mathcal A_u(B_X)$ to $\overline B_{X^{**}}$ (see \cite{AB,ACG,Aron-survey} for an explanation), for every $z\in \overline B_{X^{**}}$ we have the evaluation functional $\delta(z)$ which belongs to $\mathcal M(\mathcal A_u(B_X))$. In particular,
    \begin{enumerate}
        \item $\mathcal M(\mathcal A_u(\disk))=\{\delta(z): z\in\overline\disk\}$. Hence, every multiplicative functional in $\mathcal HL_0(\disk)^*$ actually belongs to $\mathcal G_0(\disk)$.
        \item $\mathcal M(\mathcal A_u(B_{c_0}))=\{\delta(z): z\in\overline B_{\ell_\infty}\}$.
        \item For any $1<p<\infty$, $\mathcal M(\mathcal A_u(B_{\ell_p}))\supsetneq\{\delta(z): z\in\overline B_{\ell_p}\}$. Since $\mathcal M(\mathcal A_u(B_{\ell_p}))$ is $w^*$-compact, there exist many $w^*$-limits of nets of evaluations that are not themselves evaluations.
        For instance, if $(e_n)$ denotes the unit vector basis of $\ell_p$, then any $w^*$-accumulation point of $(\delta(e_n))$ is a multiplicative functional that is not an evaluation at any $x\in \overline B_{\ell_p}$ (for details, see e.g. \cite{ACG, Aron-survey,CGMS}).
    \end{enumerate}
\end{example}

\subsection{Onto and into composition operators}

Using the same arguments as in Proposition \ref{p:CompositionOnto} and invoking once again Lemma \ref{lemma:biLipschitz-biHolomorphic}, we obtain the following parallel result.

\begin{proposition}\label{p:CompositionOntoHL}
    Let $X,Y$ be complex Banach spaces and let $\phi\in \mathcal HL(B_Y,X)$ with $\phi(B_Y)\subseteq \overline{B}_X$. 
    \begin{itemize}
     \item[(a)] If  $C_\phi$ is onto (equiv. $\widehat{\phi}$ is an isomorphism onto its image),  then $\phi$ is bi-Lipschitz onto its image and $\norm{d\phi(y_1)(y_2)}\geq L(\phi^{-1})^{-1} $ for every $y_1\in B_Y$ and $y_2\in S_Y$. 
     \item[(b)] If $\phi(B_Y)$ is dense in $B_X$, or $\phi(B_Y)$ has non-empty interior, then $C_\phi$ is injective.
     \item[(c)] If $C_\phi$ is injective, then $\phi(B_Y)$ has dense span.
     \item[(d)] If $\phi$ is a bi-Lipschitz bijection between $B_Y$ and $B_X$, then $\phi^{-1}$ is holomorphic and $C_\phi$ is an onto isomorphism.
     \item[(e)] If  $\phi(B_Y)=B_X$ and $C_\phi$ is an onto isometry, then $\phi$ is the restriction of an onto linear isometry $T\colon Y\to X$.
     \end{itemize}
\end{proposition}

\begin{proof}
(e) It suffices to show that $\phi(0)=0$, and the conclusion will follow from Proposition \ref{p:CompositionOnto} (e). Take $x^*\in X^*$. For  $c\in\mathbb C$, we have
\begin{align*}
|c|+\norm{x^*}=\norm{c+x^*}_L=\norm{C_\phi(c+x^*)}_L=|c+x^*(\phi(0))|+L(x^*\circ \phi)
\end{align*}
Thus $|c+x^*(\phi(0))|-|c|$ is constant for every $c\in \mathbb C$, so it follows that $x^*(\phi(0))=0$. Since $x^*$ is arbitrary, we get $\phi(0)=0$. 
\end{proof}
We are also able to prove an analogue of Corollary \ref{c:BAP-characterizationOfCompositionIsomorphism}. 
\begin{corollary}\label{c:BAP-characterizationOfCompositionIsomorphismHL}
    Let $X,Y$ be complex Banach spaces such that $Y$ has the BAP. Let $\phi\in \mathcal HL(B_Y,X)$ with $\phi(B_Y)\subseteq \overline{B}_X$. Then, the following assertions are equivalent.
    \begin{enumerate}
        \item $C_\phi \colon \mathcal HL(B_X)\rightarrow \mathcal HL(B_Y)$ is an onto isomorphism.
        \item $\phi \colon \overline{B}_Y \rightarrow \overline{B}_X$ is a bijection and $\phi^{-1}\in \mathcal HL(B_X,Y)$.
        \item $\phi : B_Y \rightarrow B_X$ is a bijection and $\phi^{-1}\in \mathcal HL(B_X,Y)$.
        \item There is a linear onto isometry $T\colon Y\to X$ and a biholomorphic bi-Lipschitz function  $\varphi\colon B_X\to B_X$  such that $\phi=\varphi\circ T$. 
        \end{enumerate}
\end{corollary}

\begin{proof}
    $(1)\Rightarrow (2)$ It is the same proof as in Corollary \ref{c:BAP-characterizationOfCompositionIsomorphism}, observing that we have to consider the closed unit balls and apply Theorem \ref{thm:compositionoperatorequivalenceHL}.

    $(2)\Rightarrow(3)$ If $\phi \colon \overline{B}_Y \rightarrow \overline{B}_X$ is a bijection and $\phi^{-1}\in \mathcal HL(B_X,Y)$, in particular, $\phi(B_Y)\subseteq B_X$ or $\phi(B_Y)\subseteq S_X$, due to the Maximum Modulus Theorem. However, if $\phi(B_Y)\subseteq S_X$, using the continuity of $\phi$ we have $\phi\left(\overline{B}_Y\right)\subseteq S_X$, which contradicts the bijectivity of $\phi$. It is clear that the same happens to $\phi^{-1}$.  Hence, the only possibility is that $\phi(B_Y)\subseteq B_X$ and $\phi^{-1}(B_X)\subseteq B_Y$, which proves (3).

    $(3)\Rightarrow(4)$ By the theorem of Kaup and Upmeier (see the paragraph after Theorem~2.3 in \cite{Arazy}), there is an onto linear isometry $T\colon Y\to X$ and an automorphism $\varphi$ of $B_X$ such that $\phi=\varphi\circ T$. From the properties of $T$, it follows that $L(\varphi)= L(\phi)$ and $L(\varphi^{-1})= L(\phi^{-1})$. 

    $(4)\Rightarrow (1)$ is clear. 

\end{proof}

\subsection{Compactness of composition operators on \texorpdfstring{$\mathcal HL(B_X)$}{HL(BX)}}

First, we present a general characterization of the compactness and weak compactness of composition operators, which should be compared with Proposition \ref{prop:compunitball}, together with its consequences corresponding to Propositions \ref{prop:compRelComp} and \ref{prop:compRelWComp}.

\begin{proposition}\label{prop:compunitballHL}
 Let $X, Y$ be  complex Banach spaces, and let $\phi\in \mathcal HL(B_Y, X)$ with $\phi(B_Y)\subseteq \overline{B}_X$.  The following statements are equivalent:
\begin{itemize}
\item[(1)] $C_\phi\colon \mathcal HL(B_X)\to \mathcal HL(B_Y)$ is a compact (resp. weakly compact) operator.
\item[(2)] $\widehat{\phi}\colon \mathcal G(B_Y)\to \mathcal G(B_X)$ is a compact (resp. weakly compact) operator. 
\item[(3)] The set 
\[\left\{\frac{\delta(\phi(y_1))-\delta(\phi(y_2))}{\norm{y_1-y_2}}: y_1, y_2\in B_Y, y_1\neq y_2\right\}\]
is relatively compact (resp. relatively weakly compact) in $\mathcal G(B_X)$. 
\item[(4)] The set
\[\left\{\widehat\phi(d\delta(y_1)(y_2)): y_1\in B_Y, y_2\in S_Y\right\}\]
is relatively compact (resp. relatively weakly compact) in $\mathcal G(B_X)$.
\end{itemize}
\end{proposition}

\begin{proof}
    Just follow the proof of Proposition 2.1 in \cite{ACPcomp}, using the expressions for $\overline{B}_{\mathcal G(B_Y)}$ obtained in Proposition \ref{prop:unitballconvHL}. Finally, note that $A\subseteq \mathcal G(B_X)$ is relatively (weakly) compact if and only if so is $A+\delta(\phi(0))$.
\end{proof}

\begin{proposition}\label{prop:compRelCompHL}
     Let $X, Y$ be complex Banach spaces and $\phi\in \mathcal HL(B_Y,X)$ such that $\overline{\phi(B_Y)}\subseteq B_X$. Then the following assertions are equivalent:
    \begin{itemize}
\item[(1)] $C_\phi\colon \mathcal HL(B_X)\to \mathcal HL(B_Y)$ is compact. 
\item[(2)] $\phi(B_Y)$ is relatively compact and $d\phi(B_Y)(S_Y)$ is relatively compact.
    \end{itemize}
\end{proposition}

\begin{proposition}\label{prop:compRelWCompHL}
     Let $X, Y$ be complex Banach spaces and $\phi\in \mathcal HL(B_Y,X)$ such that $\overline{\phi(B_Y)}\subseteq B_X$. 
    \begin{itemize}
\item[(a)] If $C_\phi\colon \mathcal HL(B_X)\to \mathcal HL(B_Y)$ is weakly compact, then $\phi(B_Y)$ and $d\phi(B_Y)(S_Y)$ are relatively weakly compact sets. 
\item[(b)] If $\phi(B_Y)$ is relatively compact and $d\phi(B_Y)(S_Y)$ is relatively weakly compact, then $C_\phi\colon \mathcal HL(B_X)\to \mathcal HL(B_Y)$ is weakly compact. 
    \end{itemize}   
\end{proposition}

When the symbol $\phi$ belongs to $\mathcal HL_0(B_Y,X)$ the (weak) compactness of the composition operator between $\mathcal HL$ spaces is equivalent to the same property for the mapping between $\mathcal HL_0$ spaces.

\begin{lemma}\label{lemma:compHLHL0}
    Let $X,Y$ be complex Banach spaces, and let $\phi\in\mathcal HL_0(B_Y,X)$ satisfy $\phi(B_Y)\subseteq B_X$. Then, for the composition operator $C_\phi : \mathcal HL_0(B_X)\rightarrow \mathcal HL_0(B_Y)$, the induced mapping $\widetilde{C}_\phi : \mathcal HL(B_X)\rightarrow \mathcal HL(B_Y)$ is the composition operator with symbol $\phi$. Moreover, $C_\phi$ is (weakly) compact if and only if $\widetilde{C}_\phi$ is (weakly) compact; and $\widetilde{C}_\phi$ fixes no copy of $\ell_\infty$ if and only if $C_\phi$ fixes no copy of $\ell_\infty$. 
\end{lemma}

\begin{proof}
    For the first part, observe that $$\widetilde{C}_\phi(f)=C_\phi(f-f(0))+f(0)=f\circ \phi -f(0)\circ \phi+f(0)=f\circ \phi, \quad \forall f\in \mathcal HL(B_X),$$
    where we considered $f(0)$ as a constant mapping, which is an element of $\mathcal HL(B_X)$.

  Now, let $T\colon \mathcal HL(B_X)\to \mathcal HL_0(B_X)\oplus_1 \mathbb C$ be given by $Tf=(f-f(0), f(0))$. Clearly $T$ is an onto isometry and $\widetilde{C}_\phi= T^{-1}\circ (C_\phi\oplus I)\circ T$. This shows that  $C_\phi$ satisfies any of the properties in the statement if and only if so does  $\widetilde{C}_\phi$. 
    
\end{proof}

It is evident that a composition operator with constant symbol is compact:

\begin{lemma}
    Let $X,Y$ be complex Banach spaces and let $C_\phi: \mathcal HL(B_X)\rightarrow \mathcal HL(B_Y)$ be a composition operator. If there is some $x_0\in \overline{B}_X$ such that  $\phi\equiv x_0$, then $C_\phi$ is compact. 
\end{lemma}

\begin{proof}
    Observe that $C_\phi(f)=f(x_0)\in \overline{\disk}\subseteq \mathcal HL(B_Y)$, for all $f\in B_{\mathcal HL(B_X)}$, and $\overline{\disk}$ is a compact set.
\end{proof}

Now we characterize (weak) compactness of composition operators in the scalar case $(X=Y=\C)$, as we did in Theorem \ref{thm:compactscalarcase}.  Recall that if $\phi \in \mathcal{H}L(\disk)$ satisfies $\phi(\disk)\subseteq \overline\disk$ then $\phi(\disk)\subseteq \disk$ or $\phi$ is constant. For constant symbols, we are done by the previous lemma, so we restrict to the case where $\phi(\disk)\subseteq \disk$.

Given $a\in \mathbb D$, we denote $\varphi_a\colon \mathbb D\to \mathbb D$ the automorphism of $\mathbb D$ given by $\varphi_a(z)=\frac{z-a}{1-\overline{a}z}$. Note that
\[ \varphi_a'(z)= \frac{1-\overline{a}z+\overline{a}(z-a)}{(1-\overline{a}z)^2}= \frac{1-|a|^2}{(1-\overline{a}z)^2} \quad \forall z\in\mathbb D\]
and thus 
\[L(\varphi_a)= \sup_{z\in \mathbb D} \frac{1-|a|^2}{|1-\overline{a}z|^2}= \frac{1-|a|^2}{(1-|a|)^2}=\frac{1+|a|}{1-|a|},\]
so $\varphi_a\in \mathcal HL(\mathbb D)$ with $\|\varphi_a\|_L=\frac{1+|a|}{1-|a|} + |a|$. Recall also that $(\varphi_a)^{-1}= \varphi_{-a}$.

\begin{theorem}\label{thm:compactscalarcaseHL} 
    Let $\phi \in \mathcal{H}L(\disk)$ such that $\phi(\disk)\subseteq \disk$ and consider $C_\phi$ the associated composition operator. Then, the following assertions are equivalent. 
    \begin{enumerate}
    \item $C_\phi \colon \mathcal HL(\disk)\rightarrow \mathcal HL(\disk)$ is compact.
        \item $C_\phi \colon \mathcal HL(\disk)\rightarrow \mathcal HL(\disk)$  does not fix any copy of $\ell_\infty$. 
        \item $C_\phi \colon \mathcal H^\infty(\disk)\rightarrow \mathcal H^\infty(\disk)$ is compact.
        \item $\norm{\phi}_\infty<1$.
    \end{enumerate}
\end{theorem}

\begin{proof}
     Denote $a=\phi(0)$ and let $\psi = \varphi_{a}\circ\phi$. Observe that $\psi\in \mathcal HL_0(\mathbb D)$, $\psi(\mathbb D)\subseteq \mathbb D$ and $\phi=\varphi_{-a}\circ \psi$. We will use this factorization to show that each of the conditions in the statement is equivalent to the corresponding one in Theorem \ref{thm:compactscalarcase}.
   
   First, note that the composition operator $\widetilde{C}_\psi\colon \mathcal HL(\mathbb D)\to \mathcal HL(\mathbb D)$ satisfies $\widetilde{C}_\psi=C_\phi\circ C_{\varphi_a}$, where $C_{\varphi_a}$ is an isomorphism from $\mathcal HL(\mathbb D)$ to $\mathcal HL(\mathbb D)$ with inverse $C_{\varphi_{-a}}$. Hence, $C_\phi \colon \mathcal HL(\disk)\rightarrow \mathcal HL(\disk)$ is  compact (resp. $\ell_\infty$-strictly singular) if and only if  $\widetilde{C}_\psi\colon \mathcal HL(\mathbb D)\to \mathcal HL(\mathbb D)$ is compact (resp. $\ell_\infty$-strictly singular). By Lemma \ref{lemma:compHLHL0}, this happens precisely if and only if the composition operator $C_{\psi}\colon \mathcal HL_0(\mathbb D)\to \mathcal HL_0(\mathbb D)$ is compact (resp. $\ell_\infty$-singular). 
   Moreover, since $\psi=\varphi_{a}\circ\phi$, the  compactness of $C_\phi \colon \mathcal H^\infty(\disk)\rightarrow \mathcal H^\infty(\disk)$ is equivalent to the compactness of $C_\psi\colon \mathcal H^\infty(\mathbb D)\to \mathcal H^\infty(\mathbb D)$. 
   Finally, it is clear that $\norm{\psi}_\infty<1$ if and only if $\norm{\phi}_\infty<1$.
   The conclusion now follows from Theorem \ref{thm:compactscalarcase}.
\end{proof}

In order to study the (weak) compactness of composition operators whose symbols satisfy  $\phi(B_Y)\subseteq B_X$, we need  a counterpart of Lemma \ref{lemma:factor}, together with an additional statement concerning the differential mapping.  Notice that, unlike the situation for $\mathcal HL_0$ spaces, the map $d$ is not an isometry  in the $\mathcal HL$ setting (indeed, it is not even injective). Nevertheless,  it still behaves well with respect to compactness, as we show below. This allows the following lemma to play the role of  Corollary \ref{cor:ideals} in the subsequent proofs.

\begin{lemma}\label{lemma:factorHL} Let $X, Y$ be complex Banach spaces and $\phi\in \mathcal HL(B_Y,X)$ with $\phi(B_Y)\subseteq B_X$. Then the following diagram commutes, 

\[\begin{tikzcd}
	{\mathcal HL(B_X)} &&&& {\mathcal HL(B_Y)} \\
	{\mathcal H^\infty(B_X,X^*)} && {\mathcal H^\infty(B_Y,X^*)} && {\mathcal H^\infty(B_Y, Y^*)}
	\arrow["{C_\phi}", from=1-1, to=1-5]
	\arrow["d"', from=1-1, to=2-1]
	\arrow["d", from=1-5, to=2-5]
	\arrow["{C_\phi^{X^*}}"', from=2-1, to=2-3]
	\arrow["M"', from=2-3, to=2-5]
\end{tikzcd}\]
where $d$ maps each function to its differential, and $M\colon \mathcal H^\infty(B_Y, X^*)\to \mathcal H^\infty(B_Y, Y^*)$ is given by $M(g)(y)= g(y)\circ d\phi(y)$. Moreover, if $d\circ C_\phi$ is (weakly) compact, then $C_\phi$ is (weakly) compact.
\end{lemma}

\begin{proof} The proof of the commutative diagram follows the same arguments as in Lemma \ref{lemma:factor}. For the behavior of $d$, consider the map $\bar{d}: \mathcal HL(B_Y)\rightarrow \mathcal H^\infty(B_Y,Y^*)\oplus_1\C$ given by $$\bar{d}(f)=(df,f(0)), \quad \forall f\in\mathcal HL(B_Y). $$ It is easy to see that $\bar{d}$ is linear and that $\norm{f}_L=\|\bar{d}(f)\|$ for every $f\in\mathcal HL(B_Y)$. Hence,  $\bar{d}$ is a linear isometry. Therefore, the (weak) compactness of $C_\phi$ is implied by the (weak) compactness of $\bar{d}\circ C_\phi$. Thus, it remains to show that if $d\circ C_\phi$ is (weakly) compact, then so is $\bar{d}\circ C_\phi$. 

    Pick a net $(f_\alpha)_\alpha\subseteq B_{\mathcal HL(B_X)}$. Since $d\circ C_\phi$ is (weakly) compact, there exists a subnet $(f_{\alpha_\beta})_\beta$ such that $(d(C_\phi(f_{\alpha_\beta})))_\beta$ is (weakly) convergent. Moreover, by the compactness of $\overline{\disk}$ and after passing to a further subnet if necessary, we may assume that $f_{\alpha_\beta}(\phi(0))$ converges in $\overline{\disk}$. Consequently,  $(\bar{d}(C_\phi(f_{\alpha_\beta})))_\beta = (d(C_\phi(f_{\alpha_\beta})),f_{\alpha_\beta}(\phi(0)))_\beta$ is (weakly) convergent, proving that $\bar{d}\circ C_\phi$ is (weakly) compact.
\end{proof}

Next, we establish the $\mathcal HL$ counterparts of Theorem \ref{theo:compactnesscomp}, Proposition \ref{prop:comp and weakly comp} and Corollary \ref{cor:compactfindimcase}.

\begin{theorem}\label{theo:compactnesscompHL} Let $X, Y$ be complex Banach spaces and let $\phi\in \mathcal HL(B_Y, X)$ with $\phi(B_Y)\subseteq B_X$. Consider the following statements: 
\begin{enumerate}
\item[(1)] $C_\phi\colon \mathcal HL(B_X)\to \mathcal HL(B_Y)$ is compact.
\item[(2)] $C_\phi\colon \mathcal HL(B_X)\to \mathcal HL(B_Y)$ does not fix any copy of $\ell_\infty$ and $\phi(B_Y)$ is relatively compact.  
\item[(3)] $\norm{\phi(y)}<1$ for every $y\in \overline{B}_Y$ and $\phi(B_Y)$ is relatively compact.
\item[(4)] $C_\phi\colon \mathcal H^\infty(B_X)\to \mathcal H^\infty(B_Y)$ is compact.
\end{enumerate}
Then $(1)\Rightarrow (2)\Rightarrow (3)$. If $\dim(Y)<\infty$, then $(3)\Rightarrow (4)$. If $\dim(X)<\infty$, then $(4)\Rightarrow (1)$. Thus, for $X$ and $Y$ finite-dimensional, all the statements are equivalent.

\end{theorem}

\begin{proof}
    Just follow the arguments in the proof of Theorem \ref{theo:compactnesscomp} replacing the referred results from Sections \ref{sect:composition operators HL0} and \ref{Section:Compactness} by the analogous statements on this section. Note that the condition $\norm{\phi(y)}<\norm{y}$ appeared in Theorem \ref{theo:compactnesscomp} (3) has to be  translated to $\norm{\phi(y)}<1$ (if $\phi(0)=0$ they are equivalent). 
\end{proof}

The proofs of the next two results follow by copying, mutatis mutandis,  their relatives in $\mathcal HL_0$. 

\begin{proposition}\label{prop:comp and weakly compHL} Let $X, Y$ be complex Banach spaces, where $X$ is reflexive and $\dim(Y)<\infty$. Let $\phi\in \mathcal HL(B_Y, X)$ such that $\phi(B_Y)\subseteq B_X$ and consider $C_\phi$ the associated composition operator. Then, the following assertions are equivalent.
\begin{enumerate}
    \item $C_\phi\colon \mathcal HL(B_X)\to\mathcal HL(B_Y)$ is weakly compact. 
    \item $C_\phi\colon \mathcal HL(B_X)\to\mathcal HL(B_Y)$ does not fix any copy of $\ell_\infty$.
    \item $C_\phi\colon \mathcal H^\infty(B_X)\to\mathcal H^\infty(B_Y)$ is weakly compact.
    \item $C_\phi\colon \mathcal H^\infty(B_X)\to\mathcal H^\infty(B_Y)$ is compact.
    \item $\norm{\phi}_\infty<1$.
\end{enumerate}

\end{proposition}

\begin{corollary}\label{cor:compactfindimcaseHL}
    Let $X,Y$ be complex finite-dimensional Banach spaces. Let $\phi \in \mathcal{H}L(B_Y,X)$ such that $\phi(B_Y)\subseteq B_X$ and consider $C_\phi$ the associated composition operator. Then, the following assertions are equivalent.
    \begin{enumerate}
        \item $C_\phi \colon \mathcal HL(B_X)\rightarrow \mathcal HL(B_Y)$ is compact. 
         \item $C_\phi \colon \mathcal HL(B_X)\rightarrow \mathcal HL(B_Y)$ does not fix any copy of $\ell_\infty$.
        \item $C_\phi \colon \mathcal H^\infty(B_X)\rightarrow \mathcal H^\infty(B_Y)$ is  weakly compact.
        \item $C_\phi \colon \mathcal H^\infty(B_X)\rightarrow \mathcal H^\infty(B_Y)$ is compact.
        \item $\norm{\phi}_\infty<1$.
    \end{enumerate}
\end{corollary}

For a strictly convex Banach space $X$, the only symbols $\phi \in \mathcal{H}L(B_Y,X)$ satisfying $\phi(B_Y)\subseteq S_X$ are the constant functions. When $X$ fails to be strictly convex,  non-constant functions $\phi$  mapping the ball into a sphere may exist. The following example shows that, in this case, the associated composition operator may or may not be compact. 

\begin{example}
    Let $\phi\in\mathcal HL(\disk, \ell_\infty^2)$ such that $\phi(\disk)\subseteq S_{\ell_\infty^2}$. Then, $\phi(z)=(\xi(z), \lambda)$ or $\phi(z)=(\lambda, \xi(z))$ with $|\lambda|=1$ and $\xi\in\mathcal HL(\disk)$ satisfying $\xi(\disk)\subseteq\overline\disk$. Then, the following assertions are equivalent.
    \begin{enumerate}
        \item $C_\phi \colon \mathcal HL(B_{\ell_\infty^2})\rightarrow \mathcal HL(\disk)$ is compact. 
         \item $C_\xi \colon \mathcal HL(\disk)\rightarrow \mathcal HL(\disk)$  is compact.
          \item $\norm{\xi}_\infty<1$ or $\xi$ is constant.
    \end{enumerate} 
    We have already known that (2) and (3) are equivalent. For $(1)\Leftrightarrow (2)$ note that $\phi$ is equal to $\xi$ composed with an inclusion, and $\xi$ coincides with $\phi$ composed with a projection.  
\end{example}

\subsection{Iterates of composition operators}
We conclude by studying the behavior of the iterates of  composition operators, in the same spirit as in Section \ref{sect:iteration}. In the present setting, since the point $0$ is no longer fixed, it is more difficult to identify a natural candidate for the limit of the iterates of a composition operator. Nevertheless, we obtain positive results in certain particular cases. First, if the symbol $\phi$ satisfies $\phi(0)=0$, then the situation is very similar to that of $\mathcal{H}L_0$. As a consequence of the following proposition, we show that $C_{\phi}^{(n)}$ converges to $C_0$ as an operator on $\mathcal{H}L$ if and only if $C_{\phi}^{(n)}$ converges to $0$ as an operator on $\mathcal{H}L_0$. 

\begin{proposition}\label{prop:normHL-C0}
    Let $X,Y$ be complex Banach spaces and let $\phi\in \mathcal HL_0(B_Y,X)$ with $\phi(B_Y)\subseteq B_X$. Consider $C_\phi : \mathcal HL(B_X)\rightarrow \mathcal HL(B_Y)$ the associated composition operator. Then, \begin{equation*}
        \norm{C_\phi-C_0}_{\mathcal L(\mathcal HL(B_X),\mathcal HL(B_Y))}=\norm{C_\phi}_{\mathcal L(\mathcal HL_0(B_X),\mathcal HL_0(B_Y))}=L(\phi).
    \end{equation*}
\end{proposition}

\begin{proof}
    On the one hand,
    \begin{align*}
        \norm{C_\phi(f)-C_0(f)}_L&=L(f\circ \phi-f(0))+\norm{f(\phi(0))-f(0)}\\&=L(f\circ \phi-f(0))\leq L(f)L(\phi)\leq \norm{f}_LL(\phi), \quad \forall f\in \mathcal HL(B_X),
    \end{align*}
    so $\norm{C_\phi-C_0}\leq L(\phi)$. On the other hand, observe that 
    \begin{align*}
        \norm{C_\phi-C_0}&=\sup_{f\in \mathcal HL(B_X)}\frac{\norm{C_\phi(f)-C_0(f)}_L}{\norm{f}_L}\geq \sup_{f\in \mathcal HL_0(B_X)}\frac{\norm{C_\phi(f)-C_0(f)}_L}{\norm{f}_L}\\&=\sup_{f\in \mathcal HL_0(B_X)}\frac{L(C_\phi(f))}{L(f)}=\norm{C_\phi}_{\mathcal L(\mathcal HL_0(B_X),\mathcal HL_0(B_Y))}=L(\phi).
    \end{align*}
    Hence, $\norm{C_\phi-C_0}=L(\phi)$.
\end{proof}

\begin{remark}
    In particular, under the conditions of Proposition \ref{prop:normHL-C0}, if $X=Y$, we have that $C_\phi^{(n)}$ converges in $\mathcal L(\mathcal HL(B_X),\mathcal HL(B_X))$ to $C_0$ if and only if $C_\phi^{(n)}$ converges in $\mathcal L(\mathcal HL_0(B_X),\mathcal HL_0(B_X))$ to $0$, where the latter case was characterized in Theorem \ref{theo:equivalent-convergent-0} and Proposition \ref{prop:equivalent-convergent-0-findim}.
\end{remark}

Finally, when $\phi$ is contractive, the iterates of $C_\phi$ turn out to converge to the composition operator associated to the (unique) fixed point of $\phi$.

\begin{proposition}
    Let $X$ be a complex Banach space and let $\phi \in \mathcal HL(B_X,X)$ such that $\phi(B_X)\subseteq \overline{B}_X$. If $L(\phi)<1$, then  the associated composition operator $C_\phi : \mathcal HL(B_X)\rightarrow \mathcal HL(B_X)$ satisfies $$\lim_{n\to\infty}\norm{C_{\phi}^{(n)}-C_a}=0,$$
    where $a\in \overline{B}_X$ is the unique fixed point of $\phi$ in $\overline{B}_X$.
\end{proposition}

\begin{proof}
    Consider the Lipschitz extension of $\phi$ from $\overline{B}_X$ to $\overline{B}_X$. By assumption, it is contractive, and therefore Banach's fixed-point theorem ensures the existence of a unique fixed point $a\in \overline{B}_X$. Moreover, $a=\lim_{n\to\infty} \phi^{(n)}(0)$.
    Hence, for each $n\in \N$, we have
    \begin{align*}
        \norm{C_{\phi}^{(n)}-C_a}&=\sup_{\norm{f}_L=1} \norm{f\circ \phi^{(n)}-f(a)}_L\\&= \sup_{\norm{f}_L=1} L(f\circ \phi^{(n)}-f(a))+\norm{f(\phi^{(n)}(0))-f(a)}\\&=\sup_{\norm{f}_L=1} L(f\circ \phi^{(n)})+\norm{f(\phi^{(n)}(0))-f(a)}\\&\leq \sup_{\norm{f}_L=1} L(f)L( \phi)^n+L(f)\norm{\phi^{(n)}(0)-a}\\
        &\le L( \phi)^n+\norm{\phi^{(n)}(0)-a}.
    \end{align*}
    Clearly, the right hand side tends to $0$ as $n$ tends to $\infty$, since $L(\phi)<1$.
\end{proof}

\medskip
\section*{Acknowledgements} 

We are grateful to Colin Petitjean for his useful comments. 

This work was initiated during the first-named author's visit to the Universidad de Zaragoza in September and October 2025. She would like to thank everyone, both within and outside the university, who helped make her visit such a delightful experience. This stay was partially supported by a grant from Programa Ibercaja-CAI. 

The research of Verónica Dimant was partially supported by
CONICET PIP 11220200101609CO and UdeSA-PAI 2025.

The research of Luis C. García-Lirola was supported by grants PID2021-122126NB-C31 and PID2022-137294NB-I00 funded by MCIN/AEI/\\ 
10.13039/501100011033 and by ``ERDF A way of making Europe''; and by grant E48-23R funded by Diputación General de Aragón (DGA), and by Fundaci\'on S\'eneca: ACyT Regi\'on de Murcia grant 21955/PI/22.

The research of Juan Guerrero-Viu was supported by FPU24/02284 predoctoral grant funded by MCIU; by grant PID2022-137294NB-I00 funded by MCIN/AEI/
10.13039/501100011033 and by ``ERDF A way of making Europe''; by grant E48-23R funded by Diputación General de Aragón (DGA).

The research of Antonín Procházka was supported  by the French ANR project no. ANR-24-CE40-0892-01.
The LmB receives support from the EIPHI Graduate School (contract ANR-17-EURE-0002).

\end{document}